\documentclass[a4paper,10pt,envcountsame,envcountsect]{llncs}

\usepackage{hyperref}
\usepackage{amsmath,amssymb,amsfonts,amscd}
\usepackage[latin1]{inputenc}
\usepackage{xcolor}
\usepackage{cite} 
\usepackage{url}
\usepackage{pdflscape}
\usepackage{fullpage}
\usepackage{graphicx, psfrag}
\usepackage[english]{babel}
\usepackage[latin1]{inputenc}
\usepackage{siunitx}
\usepackage[textsize=tiny,colorinlistoftodos]{todonotes}
\usepackage{multido}
\usepackage{longtable}
\usepackage{tabu}
\usepackage{float}
\usepackage{multirow}

\usepackage[T1]{fontenc}
\usepackage[bitstream-charter]{mathdesign}
\usepackage{rotating}
\usepackage{arydshln}
\usepackage{url}
\usepackage{algorithm}
\usepackage{algorithmicx}
\usepackage[noend]{algpseudocode}
\usepackage{textcomp}
\usepackage{enumerate}
\usepackage{enumitem}
\usepackage{tikz}
\usetikzlibrary{matrix,decorations.pathreplacing,backgrounds,positioning,calc,shapes.misc,mindmap}

\setlist{noitemsep,leftmargin=20pt,topsep=0pt,parsep=1pt,partopsep=0pt}
\renewcommand{\arraystretch}{1.4}

\spnewtheorem{convention}[theorem]{Convention}{\bfseries}{\itshape}
\spnewtheorem{notation}[theorem]{Notation}{\bfseries}{\itshape}

\spnewtheorem{characteristics}[remark]{Algorithmic Property}{\itshape}{}
\newcommand{\characteristic}{Property}

\spnewtheorem{variants}[remark]{Variants \& Specifications}{\itshape}{}

\spnewtheorem{vocabulary}[remark]{Vocabulary}{\itshape}{}
\DeclareMathOperator{\sigSym}{{\mathfrak{s}}}

\DeclareMathOperator{\headSym}{lt}
\DeclareMathOperator{\indexSym}{ind}
\DeclareMathOperator{\lcm}{lcm}

\newcommand{\comment}[1]{}

\newcommand{\field}{\mathcal{K}}
\newcommand{\ring}{\mathcal{R}}
\newcommand{\invring}{\mathcal{R}^G}
\newcommand{\gl}[1]{\mathrm{GL}(#1,\field)}
\newcommand{\rey}{\mathfrak{R}}
\newcommand{\numberGenerators}{m}
\newcommand{\gen}[1]{{#1}_1,\ldots,{#1}_\numberGenerators}
\newcommand{\mgen}{\boldsymbol{e}_1,\ldots,\boldsymbol{e}_\numberGenerators}
\newcommand{\module}{\ring^\numberGenerators}
\newcommand{\mon}{\mathcal{M}}
\newcommand{\modmon}{\mathcal{N}}
\newcommand{\const}{\kappa}
\newcommand{\hdsyz}{\ensuremath{\mathcal{H}}}
\newcommand{\pairset}{\mathcal{P}}
\newcommand{\F}{\mathbb{F}}

\newcommand{\Q}{\mathbb{Q}}
\newcommand{\Z}{\mathbb{Z}}

\newcommand{\proj}[1]{\overline{#1}}
\newcommand{\mbasis}[1]{\boldsymbol{e}_{#1}}
\newcommand{\gbasissome}{\alpha}
\newcommand{\gbasis}[1]{\gbasissome_{#1}}
\newcommand{\sig}[1]{\sigSym\left({#1}\right)}
\newcommand{\syz}[1]{\mathrm{syz}\left({#1}\right)}
\newcommand{\origBasis}{\mathcal G}
\newcommand{\basis}{\ensuremath{\origBasis}}

\newcommand{\ind}[1]{\indexSym\left({#1}\right)}
\newcommand{\rewriters}[1]{\mathcal{C}_{#1}}

\newcommand{\pleq}{\preceq}

\newcommand{\pleqs}{\preceq_{\sigSym}}

\newcommand{\rleqff}{\trianglelefteq_{\mathrm{add}}}

\newcommand{\rleqsb}{\trianglelefteq_{\mathrm{rat}}}

\newcommand{\rleq}{\trianglelefteq}

\newcommand{\set}[1]{\left\{{#1}\right\}}
\newcommand{\setBuilder}[2]{\left\{{#1}\left|{#2}\right.\right\}}

\newcommand{\ideal}[1]{\left\langle{#1}\right\rangle}
\newcommand{\lc}[1]{\mathrm{lc}\left({#1}\right)}
\newcommand{\hd}[1]{\headSym\left({#1}\right)}
\newcommand{\hdp}[1]{\hd{\proj{#1}}}
\newcommand{\card}[1]{\left|#1\right|}
\newcommand{\spair}[2]{\ensuremath{\mathrm{spair}\left({#1},{#2}\right)}}
\newcommand{\spoly}[2]{\ensuremath{\mathrm{spol}\left({#1},{#2}\right)}}
\newcommand{\jac}[2]{\ensuremath{\mathrm{jac}_{#1}\left({#2}\right)}}
\newcommand{\spp}[1]{\ensuremath{\left(\sig{#1},\proj{#1}\right)}}
\newcommand{\koz}[2]{\ensuremath{\mathrm{ksyz}\left({#1},{#2}\right)}}

\newcommand{\potg}{\ensuremath{>_\textrm{pot}}}
\newcommand{\potl}{\ensuremath{<_\textrm{pot}}}

\newcommand{\topl}{\ensuremath{<_\textrm{top}}}

\newcommand{\dpotl}{\ensuremath{<_\textrm{d-pot}}}

\newcommand{\dtopl}{\ensuremath{<_\textrm{d-top}}}

\newcommand{\hdpotl}{\ensuremath{<_\textrm{lt-pot}}}

\newcommand{\hdtopl}{\ensuremath{<_\textrm{lt-top}}}
\newcommand{\potlp}{\ensuremath{<_\textrm{pot'}}}

\newcommand{\schl}{\hdpotl}

\newcommand{\simplify}{\ensuremath{\mathrm{\bf Simplify}}}
\newcommand{\select}[1]{\ensuremath{\mathrm{\bf
  Select}\left({#1}\right)}}
\newcommand{\symbolicpreprocessing}{\ensuremath{\text{\bf
  Symbolic Preprocessing}}}
\newcommand{\rewritablenoarg}{\ensuremath{\mathrm{\bf Rewritable}}}

\newcommand{\notrewritable}[1]{\ensuremath{\mathrm{\bf not\;
  Rewritable}\left({#1,\basis\cup\hdsyz,\rleq}\right)}}
\newcommand{\hrandom}{\ensuremath{\mathrm{
  HRandom}}}
\newcommand{\random}{\ensuremath{\mathrm{
  Random}}}

\newcommand{\gsb}{{\bf genSB}}
\newcommand{\sba}{{\bf SB}}
\newcommand{\sgb}{{\bf SGB}}
\newcommand{\rba}{{\bf RB}}
\newcommand{\ffrba}{{\bf F4-RB}}
\newcommand{\trb}{{\bf TRB}}
\newcommand{\trbff}{{\bf TRB-F5}}
\newcommand{\trbeff}{{\bf TRB-EF5}}
\newcommand{\trbgvw}{{\bf TRB-GVW}}
\newcommand{\trbmj}{{\bf TRB-MJ}}
\newcommand{\gbgc}{{\bf GBGC}}
\newcommand{\ap}{{\bf AP}}
\newcommand{\gvwhs}{{\bf GVWHS}}
\newcommand{\gvw}{{\bf GVW}}
\newcommand{\ggv}{{\bf G2V}}
\newcommand{\gggv}{{\bf ImpG2V}}
\newcommand{\iggv}{{\bf iG2V}}
\newcommand{\ffgen}{{\bf F5GEN}}
\newcommand{\ssg}{{\bf SSG}}

\newcommand{\ffo}{\ensuremath{\text{\bf F4}}}
\newcommand{\ff}{\ensuremath{\text{\bf F5}}}
\newcommand{\ffr}{\ensuremath{\text{\bf F5R}}}
\newcommand{\ffb}{\ensuremath{\text{\bf F5B}}}
\newcommand{\ffc}{\ensuremath{\text{\bf F5C}}}
\newcommand{\iffc}{\ensuremath{\text{\bf iF5C}}}
\newcommand{\iffa}{\ensuremath{\text{\bf iF5A}}}
\newcommand{\ffa}{\ensuremath{\text{\bf F5A}}}
\newcommand{\ffs}{\ensuremath{\text{\bf F4/5}}}
\newcommand{\fft}{\ensuremath{\text{\bf F5t}}}
\newcommand{\fftwo}{\ensuremath{\text{\bf F5/2}}}
\newcommand{\ffp}{\ensuremath{\text{\bf F5+}}}
\newcommand{\ffpr}{\ensuremath{\text{\bf F5'}}}
\newcommand{\ffprpr}{\ensuremath{\text{\bf F5''}}}
\newcommand{\mff}{\ensuremath{\text{\bf MatrixF5}}}

\newcommand{\mac}[2]{\ensuremath{M_{#1,#2}}}

\newcommand{\singular}{\ensuremath{\text{\sc Singular}}}
\newcommand{\magma}{\ensuremath{\text{\sc Magma}}}

\newcommand*{\defeq}{\mathrel{\vcenter{\baselineskip0.5ex \lineskiplimit0pt
                     \hbox{\scriptsize.}\hbox{\scriptsize.}}}%
                     =}

\newcommand{\sreduction}{\ensuremath{\sigSym}-reduction}
\newcommand{\sreductions}{\ensuremath{\sigSym}-reductions}
\newcommand{\sreduce}{\ensuremath{\sigSym}-reduce}
\newcommand{\sreduces}{\ensuremath{\sigSym}-reduces}
\newcommand{\sreduced}{\ensuremath{\sigSym}-reduced}
\newcommand{\sreducible}{\ensuremath{\sigSym}-reducible}
\newcommand{\sreducing}{\ensuremath{\sigSym}-reducing}
\newcommand{\sreducibility}{\ensuremath{\sigSym}-reducibility}
\newcommand{\sreducer}{\ensuremath{\sigSym}-reducer}
\newcommand{\sreducers}{\ensuremath{\sigSym}-reducers}

\newcommand{\grobner}{Gr\"obner}
\newcommand{\faugere}{Faug\`ere}

\newcommand{\hideAppendix}[1]{}\newcommand{\appRef}[1]{the online appendix of \cite{pracgb}}\newcommand{\AppRef}[1]{The online appendix of \cite{pracgb}}

\newcommand{\zero}{\ensuremath{\textcolor{vlgrey}{0}}}

\newcommand{\bo}[1]{\ensuremath{\textcolor{ggrey}{#1}}}

\definecolor{dblue}{HTML}{205488}
\definecolor{lblue}{HTML}{adc0d3}
\definecolor{dgreen}{HTML}{4F8355}
\definecolor{lgreen}{HTML}{BBCFBD}
\definecolor{yellow}{HTML}{F9EB0F}
\definecolor{red}{HTML}{FF4848}
\definecolor{greenalg}{HTML}{2cca2c}
\definecolor{ggrey}{HTML}{303030}
\definecolor{lgrey}{HTML}{404040}
\definecolor{vlgrey}{HTML}{707070}
\definecolor{nearwhite}{HTML}{DFDFDF}

\definecolor{mybl}{HTML}{62ABAB}
\definecolor{myblued}{HTML}{236DFF}
\definecolor{mybluedl}{HTML}{239BFF}
\definecolor{mybluel}{HTML}{697Fd4}
\definecolor{myredd}{HTML}{c31313}
\definecolor{myredl}{HTML}{c36767}
\definecolor{mygreend}{HTML}{2ca92c}
\definecolor{mygreenl}{HTML}{79a979}
\definecolor{myyellowd}{HTML}{F9EB0F}
\definecolor{mypinkd}{HTML}{E227A7}
\definecolor{myyellowl}{HTML}{e2a97c}
\definecolor{mybrown}{HTML}{442e0c}
\definecolor{mydpink}{HTML}{4d4d7c}
\definecolor{mylblue}{HTML}{338DF3}
\definecolor{mygreenyellow}{HTML}{ABF484}

\def\term{}
\def\extterm#1{\edef\term{\term#1}}
\newcommand{\matBorder}[4] {
\def\gen{#1}
\def\d{#2}
\def\di{#3}
\def\matName{#4}
\node[anchor=south east] (cornernode) at (\matName-1-1.north west) {\hspace*{3em}};
\def\d{#2}
\pgfmathparse{int(\d - \di)}
\global\let\d\pgfmathresult
\ifnum\pgfmathresult>0
  \pgfmathparse{int(1)}
  \global\let\idex\pgfmathresult
  \foreach[count=\xi] \x in \gen{ 
    \foreach[count=\ki] \k in {0,...,\d}{
    \pgfmathparse{int(\d-\k)}
    \global\let\dk\pgfmathresult
    \foreach[count=\ii] \i in {\dk,...,0}{
    \ifnum\i>0\relax
      \ifnum\i>1\relax
        \extterm{x^\i}
      \else
        \extterm{x}
      \fi
    \fi
    \pgfmathparse{int(\d-\k-\i)}
    \let\dki\pgfmathresult
    \ifnum\dki>0\relax
      \ifnum\dki>1\relax
        \extterm{y^\dki}
      \else
        \extterm{y}
      \fi
    \fi
    \ifnum\k>0\relax
      \ifnum\k>1\relax
        \extterm{z^\k}
      \else
        \extterm{z}
      \fi
    \fi
    \node[text depth=0pt,font=\scriptsize,right]
    (\matName-\idex-0) [left=1.5em of \matName-\idex-1] {$\bo{\term \mbasis \x}$};
    \pgfmathparse{int(\idex + 1)}
    \global\let\idex\pgfmathresult
      }
    }
  }
\else
  \foreach[count=\xi] \x in \gen{ 
    \node[font=\scriptsize,right] (\matName-\xi-0) [left=1.5em of
      \matName-\xi-1] {\bo{\mbasis \x}}; 
  }
\fi

\def\d{#2}

\node[right] (mat-node) [left=3.0em of \matName] {$\matName_\d =$};
\pgfmathparse{int(1)}
\global\let\idex\pgfmathresult

\foreach[count=\ki] \k in {0,...,\d}{
\pgfmathparse{int(\d-\k)}
\global\let\dk\pgfmathresult
\foreach[count=\ii] \i in {\dk,...,0}{
\ifnum\i>0\relax
  \ifnum\i>1\relax
    \extterm{x^\i}
  \else
    \extterm{x}
  \fi
\fi
\pgfmathparse{int(\d-\k-\i)}
\let\dki\pgfmathresult
\ifnum\dki>0\relax
  \ifnum\dki>1\relax
    \extterm{y^\dki}
  \else
    \extterm{y}
  \fi
\fi
\ifnum\k>0\relax
  \ifnum\k>1\relax
    \extterm{z^\k}
  \else
    \extterm{z}
  \fi
\fi
\node[above=0.3em of \matName-1-\idex,text depth=0pt,font=\scriptsize]
(\matName-0-\idex) {$\bo{\term}$};
\pgfmathparse{int(\idex + 1)}
\global\let\idex\pgfmathresult
  }
}
} 

\newcommand{\matBorderNoName}[4] {
\def\gen{#1}
\def\d{#2}
\def\di{#3}
\def\matName{#4}
\node[anchor=south east] (cornernode) at (\matName-1-1.north west) {\hspace*{3em}};
\def\d{#2}
\pgfmathparse{int(\d - \di)}
\global\let\d\pgfmathresult
\ifnum\pgfmathresult>0
  \pgfmathparse{int(1)}
  \global\let\idex\pgfmathresult
  \foreach[count=\xi] \x in \gen{ 
    \foreach[count=\ki] \k in {0,...,\d}{
    \pgfmathparse{int(\d-\k)}
    \global\let\dk\pgfmathresult
    \foreach[count=\ii] \i in {\dk,...,0}{
    \ifnum\i>0\relax
      \ifnum\i>1\relax
        \extterm{x^\i}
      \else
        \extterm{x}
      \fi
    \fi
    \pgfmathparse{int(\d-\k-\i)}
    \let\dki\pgfmathresult
    \ifnum\dki>0\relax
      \ifnum\dki>1\relax
        \extterm{y^\dki}
      \else
        \extterm{y}
      \fi
    \fi
    \ifnum\k>0\relax
      \ifnum\k>1\relax
        \extterm{z^\k}
      \else
        \extterm{z}
      \fi
    \fi
    \node[text depth=0pt,font=\scriptsize,right]
    (\matName-\idex-0) [left=1.5em of \matName-\idex-1] {$\bo{\term \mbasis \x}$};
    \pgfmathparse{int(\idex + 1)}
    \global\let\idex\pgfmathresult
      }
    }
  }
\else
  \foreach[count=\xi] \x in \gen{ 
    \node[font=\scriptsize,right] (\matName-\xi-0) [left=1.5em of
      \matName-\xi-1] {\bo{\mbasis \x}}; 
  }
\fi

\def\d{#2}

\pgfmathparse{int(1)}
\global\let\idex\pgfmathresult

\foreach[count=\ki] \k in {0,...,\d}{
\pgfmathparse{int(\d-\k)}
\global\let\dk\pgfmathresult
\foreach[count=\ii] \i in {\dk,...,0}{
\ifnum\i>0\relax
  \ifnum\i>1\relax
    \extterm{x^\i}
  \else
    \extterm{x}
  \fi
\fi
\pgfmathparse{int(\d-\k-\i)}
\let\dki\pgfmathresult
\ifnum\dki>0\relax
  \ifnum\dki>1\relax
    \extterm{y^\dki}
  \else
    \extterm{y}
  \fi
\fi
\ifnum\k>0\relax
  \ifnum\k>1\relax
    \extterm{z^\k}
  \else
    \extterm{z}
  \fi
\fi
\node[above=0.3em of \matName-1-\idex,text depth=0pt,font=\scriptsize]
(\matName-0-\idex) {$\bo{\term}$};
\pgfmathparse{int(\idex + 1)}
\global\let\idex\pgfmathresult
  }
}
} 

\newcommand{\matTopBorder}[4] {
\def\gen{#1}
\def\d{#2}
\def\di{#3}
\def\matName{#4}
\node[anchor=south east] (cornernode) at (\matName-1-1.north west) {\hspace*{3em}};
\def\d{#2}
\pgfmathparse{int(\d - \di)}
\global\let\d\pgfmathresult
\def\d{#2}

\pgfmathparse{int(1)}
\global\let\idex\pgfmathresult

\foreach[count=\ki] \k in {0,...,\d}{
\pgfmathparse{int(\d-\k)}
\global\let\dk\pgfmathresult
\foreach[count=\ii] \i in {\dk,...,0}{
\ifnum\i>0\relax
  \ifnum\i>1\relax
    \extterm{x^\i}
  \else
    \extterm{x}
  \fi
\fi
\pgfmathparse{int(\d-\k-\i)}
\let\dki\pgfmathresult
\ifnum\dki>0\relax
  \ifnum\dki>1\relax
    \extterm{y^\dki}
  \else
    \extterm{y}
  \fi
\fi
\ifnum\k>0\relax
  \ifnum\k>1\relax
    \extterm{z^\k}
  \else
    \extterm{z}
  \fi
\fi
\node[above=0.3em of \matName-1-\idex,text depth=0pt,font=\scriptsize]
(\matName-0-\idex) {$\bo{\term}$};
\pgfmathparse{int(\idex + 1)}
\global\let\idex\pgfmathresult
  }
}
} 

\newcommand\reducedRows[4]{
  \begin{scope}[on background layer]
  \draw [fill=mygreend!90!yellow,draw=mygreend!60!lgrey,rounded corners=2pt,fill
  opacity=0.5,draw opacity=0.7]
  (#4-#1-1.north west) rectangle (#4-#2-#3.south east);
  \draw [fill=mygreend!90!yellow,draw=mygreend!60!lgrey,rounded corners=2pt,fill
  opacity=0.5,draw opacity=0.7]
  (#4-#1-0.north west) rectangle (#4-#2-0.south east);
  \end{scope}
}

\newcommand\reducedRowsWrong[4]{
  \begin{scope}[on background layer]
  \draw [fill=mygreend!90!yellow,draw=mygreend!60!lgrey,rounded corners=2pt,fill
  opacity=0.5,draw opacity=0.7]
  (#4-#1-1.north west) rectangle (#4-#2-#3.south east);
  \draw [fill=myredd!90!yellow,draw=myredd!60!lgrey,rounded corners=2pt,fill
  opacity=0.5,draw opacity=0.7]
  (#4-#1-0.north west) rectangle (#4-#2-0.south east);
  \end{scope}
}

\newcommand\rewriteRows[4]{
  \begin{scope}[on background layer]
  \draw [fill=myblued!90!yellow,draw=myblued!60!lgrey,rounded corners=2pt,fill
  opacity=0.5,draw opacity=0.7]
  (#4-#1-1.north west) rectangle (#4-#2-#3.south east);
  \draw [fill=myblued!90!yellow,draw=myblued!60!lgrey,rounded corners=2pt,fill
  opacity=0.5,draw opacity=0.7]
  (#4-#1-0.north west) rectangle (#4-#2-0.south east);
  \end{scope}
}

\newcommand\rewrite[6]{
  \def\r{#1}
  \def\c{#2}
  \def\m{#3}
  \def\rr{#4}
  \def\cc{#5}
  \def\mm{#6}
  \begin{scope}[on background layer]
  \draw [draw=vlgrey,dashed]
    (\m-\r-\c.south west) -- (\mm-\rr-\cc.north west);
  \draw [draw=vlgrey,dashed]
    (\m-\r-\c.south east) -- (\mm-\rr-\cc.north east);
  \draw [draw=none,fill=vlgrey!30!white,fill opacity=0.2,rounded corners=2pt]
    (\m-\r-\c.north west) rectangle (\mm-\rr-\cc.south east);
  \end{scope}
  \begin{scope}
  \draw [fill=vlgrey,fill opacity=0.2,draw=red!90!yellow,rounded corners=2pt]
  (\m-\r-\c.north west) rectangle (\m-\r-\c.south east);
  \end{scope}
  \begin{scope}
  \draw [fill=vlgrey,fill opacity=0.2,draw=red!50!yellow,rounded corners=2pt]
  (\mm-\rr-\cc.north west) rectangle (\mm-\rr-\cc.south east);
  \end{scope}
}

\newcommand\notRestricted[4]{
  \begin{scope}[on background layer]
  \draw [fill=mygreend!90!yellow,draw=mygreend!60!lgrey,rounded corners=2pt,fill
  opacity=0.5,draw opacity=0.7]
  (#4-#1-1.north west) rectangle (#4-#2-#3.south east);
  \draw [fill=mygreend!90!yellow,draw=mygreend!60!lgrey,rounded corners=2pt,fill
  opacity=0.5,draw opacity=0.7]
  (#4-#1-0.north west) rectangle (#4-#2-0.south east);
  \end{scope}
}

\newcommand\restricted[4]{
  \begin{scope}[on background layer]
  \draw [fill=myyellowd!90!vlgrey,draw=myyellowd!60!lgrey,rounded corners=2pt,fill
  opacity=0.5,draw opacity=0.7]
  (#4-#1-1.north west) rectangle (#4-#2-#3.south east);
  \draw [fill=myyellowd!90!vlgrey,draw=myyellowd!60!lgrey,rounded corners=2pt,fill
  opacity=0.5,draw opacity=0.7]
  (#4-#1-0.north west) rectangle (#4-#2-0.south east);
  \end{scope}
}

\newcommand\completelyRestricted[4]{
  \begin{scope}[on background layer]
  \draw [fill=myredd!90!yellow,draw=myredd!60!lgrey,rounded corners=2pt,fill
  opacity=0.5,draw opacity=0.7]
  (#4-#1-1.north west) rectangle (#4-#2-#3.south east);
  \draw [fill=myredd!90!yellow,draw=myredd!60!lgrey,rounded corners=2pt,fill
  opacity=0.5,draw opacity=0.7]
  (#4-#1-0.north west) rectangle (#4-#2-0.south east);
  \end{scope}
}

\title{A survey on signature-based \grobner{} basis computations}

\author{Christian Eder\thanks{The author was supported by the
  EXACTA grant (ANR-09-BLAN-0371-01) of the French National
  Research Agency.}\inst{1} \and Jean-Charles Faug\`ere\inst{2}}
\institute{
University of Kaiserslautern\\
Department of Mathematics\\
PO box 3049\\
67653 Kaiserslautern\\
\email{ederc@mathematik.uni-kl.de}
\ \\
\and
INRIA, Paris-Rocquencourt Center, PolSys Project\\
UPMC, Univ. Paris 06, LIP6\\
CNRS, UMR 7606, LIP6\\
UFR Ing\'enierie 919, LIP6\\
Case 169, 4, Place Jussieu, F-75252 Paris\\
\email{Jean-Charles.Faugere@inria.fr}
}

\begin{document}
\nocite{ACFP12}
\nocite{BFS04}
\nocite{BFS05}
\nocite{ahExtF52010}
\nocite{apF452010}
\nocite{apF5CritRevised2011}
\nocite{arsPhd2005}
\nocite{bGroebner1965}
\nocite{bardetComplexity2002}
\nocite{bardetPhD}
\nocite{bcpMagma}
\nocite{bcr2011}
\nocite{blsr1999}
\nocite{buchberger2ndCriterion1985}
\nocite{ckmConvertingBases1997}
\nocite{eF5Criteria2008}
\nocite{eF5Criterion2008}
\nocite{ederImprovedF52013}
\nocite{ederPhD2012}
\nocite{egpF52011}
\nocite{epF5C2009}
\nocite{epSig2011}
\nocite{erF5SB2013}
\nocite{fF41999}
\nocite{faugeresvartz-2013}
\nocite{faugeresafeyverron-2013}
\nocite{fglmFGLM1993}
\nocite{F03}
\nocite{FJ03}
\nocite{FR09}
\nocite{FSS10}
\nocite{FSS10b}
\nocite{galkinSimple2012}
\nocite{galkinTermination2012}
\nocite{gashPhD2008}
\nocite{gerdtHashemiG2V}
\nocite{gmInstallation1988}
\nocite{gmStaggered1986}
\nocite{gpSingularBook2007}
\nocite{gvwGVW2010}
\nocite{hbrInvariantG2V}
\nocite{huang2010}
\nocite{kollBuchberger1978}
\nocite{mmm93}
\nocite{mmtSyzygies1992}
\nocite{moraBook2}
\nocite{panhuwang2012}
\nocite{panhuwang2013}
\nocite{posso}
\nocite{stF5Rev2005}
\nocite{sun2013}
\nocite{sw2009}
\nocite{sw2010}
\nocite{sw2011a}
\nocite{sw2011b}
\nocite{sw2011c}
\nocite{swmz2012}
\nocite{sw2013a}
\nocite{sw2013b}
\nocite{volnyGVW2011}
\nocite{wichmannFGLM1997}
\nocite{zobnin2010}

\maketitle

\begin{abstract}
This paper is a survey on the area of signature-based \grobner{} basis
algorithms that was initiated by \faugere{}'s \ff{} algorithm in 2002.
We explain the general ideas behind the usage of signatures.
We show how to classify the various known variants by $3$ different orderings.
For this we give translations between different
notations and show that besides notations many approaches are just the same.
Moreover, we give a general description of how the idea of signatures is quite
natural when performing the reduction process using linear algebra. This survey
shall help to outline this field of active research.
\end{abstract}

\section{Introduction}
\grobner{} bases are a fundamental tool in computer algebra with many
applications in various areas. In 1965 Buchberger introduced a first algorithmic
approach for their computation~\cite{bGroebner1965}. Over the years many
improvements and optimizations were found, for example, criteria to remove
useless elements during the computation~\cite{bGroebnerCriterion1979,buchberger2ndCriterion1985,gmInstallation1988}.

In 2002 \faugere{} presented the \ff{} algorithm~\cite{fF52002Corrected} which
was a significant development in \grobner{} basis computation. This algorithm
used for the first time signatures to detect efficiently useless data. The
\ff{} algorithm is well-known for computing no zero reduction, that means no
useless computation if the input system is regular.

Beginning 2008, many researchers worked on understanding the new criteria behind
\ff{}, which lead to new insights, but also optimizations and new variants of
the signature-based
approach~\cite{eF5Criterion2008,eF5Criteria2008,epF5C2009,apF5CritRevised2011}.

While the question of \ff{}'s termination was still an open one until recently
~\cite{galkinTermination2012,panhuwang2012,panhuwang2013}, many new variants of
\ff{} were introduced, for example, \ggv{}~\cite{ggvGGV2010} resp.
\gvw{}~\cite{gvwGVW2010,volnyGVW2011,gvwGVW2011,gvwGVW2013} or
\sba{}~\cite{rs-2012,rs-ext-2012}. 

Moreover, first papers trying to classify all the different variants of
signature-based \grobner{} basis algorithms came
up~\cite{huang2010,epSig2011,sw2011b,panhuwang2012,panhuwang2013,erF5SB2013}.

At the moment the area of signature-based \grobner{} basis algorithms
is confusing and vast. More and more papers are published proving
statements already proven before, and even more publications can be found on
``new'' variants that boil down to be a known one just with a different
notation.

In this paper we try to give a rigorous survey on signature-based
\grobner{} basis theory, including all variants known up to now. We lay an
emphasis on understanding and we show how the variants presented over the last
years are mostly differ in small parts only. Moreover, we give the reader a
vocabulary book at hand which helps to understand how notations, varying
for different authors, coincide.

Since this is a survey, we do not give proofs if they are long, complex, or do not help in
understanding the topic. We always explain the idea behind the proofs and refer
to the related publication which includes a complete proof. There the reader is
then, with our descriptions and
explanations, able to understand the proof in the used notation and language.
Table~\ref{table:variants-overview} gives the outline of this paper and can be
used as an index for finding the variant the reader might be interested in.
Moreover, Figure~\ref{fig:decade-in-sig-based-algorithms} gives a
graphical overview on the connection between the different algorithms that are
explained in the following.

In Section~\ref{sec:f5-termination} we give the problem of proving \ff{}'s termination
an in-depth discussion, where we also explain how termination-ensuring variants
as described in~\cite{arsPhd2005,gashF5t2009,egpF52011} are still useful from an
algorithmic point of view.

Moreover, we present descriptions of signature-based computations using linear
algebra for the reduction process, see Sections~\ref{sec:matrix-f5}
and~\ref{sec:f4-f5}. Besides~\cite{apF452010} which is restricted to
\ff{} this is the first known discussion on this topic and shows how the ideas
of signatures rather naturally come up in this setting.

Furthermore, we give in Section~\ref{sec:available-implementations} detailed experimental
results generated with various variants of signature-based \grobner{} basis
algorithms presented in this survey. There we do not focus on timings, but on the
characteristics of the different variants, like size of the resulting
\grobner{} basis, size of the recovered syzyy module, number of zero reductions
and number of operations overall. The code those computations are done with is
implemented in \singular{}~\cite{singular400} and available open-source. Thus the
implementation is transparent and the reader is able to understand the
different outcomes in the various algorithms.

All in all, this is the first extensive classification of signature-based
\grobner{} basis algorithms and we hope that it can be used as a useful handbook
for researchers and students.

\begin{center}
\begin{longtable}{r p{0.45\textwidth} c c}
{\bf Name/case} & {\bf modification w.r.t. \ff{}~\cite{fF52002Corrected}(Section~\ref{sec:efficient-implementations-f5})}
    & {\bf Section} & {\bf Reference}\\
\hline
\endhead

\hline \multicolumn{4}{c}{ \textcolor{ggrey}{Continued on the next page
  $\longrightarrow$}}
\endfoot

\endlastfoot
\mff{} & uses Macaulay matrices and linear algebra for reduction purposes,
does not build S-pairs but generates all multiples of the generators for
a given degree step & \ref{sec:matrix-f5} & \cite{FJ03}\\

\rba{} & generalized algorithm to compare $\rleqff$ and $\rleqsb$, special case
of \rba{} defined in Section~\ref{sec:rewrite-bases} &
\ref{sec:rewrite-basis-algorithm} & \cite{erF5SB2013}\\

\ffpr{} & homogenizes inhomogeneous input, interreduces intermediate \grobner{}
basis& \ref{sec:efficient-implementations-f5} & \cite{fF52002Corrected} \\
\ffprpr{} & uses $\dpotl$ instead of $\potl$ &
\ref{sec:efficient-implementations-f5} & \cite{fF52002Corrected}\\
\ffr{} & interreduces intermediate \grobner{} basis, uses it only for reduction purposes &
\ref{sec:f5r} & \cite{stF5Rev2005}\\
\ffc{} & interreduces intermediate \grobner{} basis, uses it for reduction purposes and for
creation of new S-pairs & \ref{sec:f5c} & \cite{epF5C2009}\\
\ffa{} & variant of \ffc{} directly using a zero reduction as signature for the
syzygy module & \ref{sec:f5c} & \cite{epSig2011}\\
\iffa{} & variant of \ffa{} recomputing signatures after interreducing between
two incremental steps, also \iggv{}, \ldots & \ref{sec:f5c} & \cite{ederImprovedF52013}\\
Extended \ff{} criteria & uses different module monomial orders &
\ref{sec:ext-f5} & \cite{ahExtF52010}\\
\fftwo{} & adds field equations to the input systems for computations over
$\mathbb{F}_2$ & \ref{sec:f52} & \cite{FJ03}\\
bihomogeneous case & uses maximal minors of Jacobian matrices to enlarge
system of syzygies & \ref{sec:f5-bihomog} & \cite{FSS10b}\\
SAGBI \grobner{} bases & uses the Reynolds operator on the syzygy
criterion & \ref{sec:f5-sagbi} & \cite{FR09}\\

\ffgen{} & generalized algorithm for different rewrite orders, applicable with
any compatible module monomial order & \ref{sec:f5-termination-directly} &
\cite{panhuwang2012,panhuwang2013}\\

\fft{} & uses the Macaulay bound, once it is exceeded the algorithm transforms
to Buchberger's algorithm & \ref{sec:f5-termination-algorithmically} &
\cite{gashPhD2008,gashF5t2009}\\

\ffb{} & uses two lists of S-pairs: one for usual \ff{}, another one for
computing a lower degree bound using Buchberger's chain criterion &
\ref{sec:f5-termination-algorithmically} & \cite{arsPhd2005}\\

\ffp{} & distinguishes S-pairs needed for the \grobner{} basis and those needed
for \ff{}'s correctness only, once only the later ones are left it uses the idea
of \ffb{} & \ref{sec:f5-termination-algorithmically} & \cite{egpF52011}\\

Arri \& Perry's work & introduces rewrite order $\rleqsb$, works for any
compatible module monomial order, directly uses zero reduction as signature
for syzygy module, also known by \ap{} & \ref{sec:arri-perry} & \cite{apF5CritRevised2011} \\

\trb{} & generalized algorithm to compare \ff{} and \gvw{}, also introduces
$\rleqsb$ as rewrite order, applicable with any compatible module monomial order
& \ref{sec:huang} & \cite{huang2010}\\

\gbgc{} & generalized algorithm, uses $\rleqsb$ but also generalizes to
use partial rewrite orders, applicable with any compatible module monomial
order, later on further generalized to work on algebras of solvable type
& \ref{sec:sun-wang} & \cite{sw2011b,swmz2012}\\

\ggv{} & directly uses zero reduction as signature for syzygy module,
rewriting is done implicitly w.r.t. $\rleqff$ & \ref{sec:g2v} &
\cite{ggvGGV2010}\\

\gvw{} & generalizes \ggv{} to be applicable with any compatible module monomial
order, uses $\rleqsb$ since 2011 (and thus coincides with \ap{}; also known as
    \gvwhs{}) & \ref{sec:gvw} &
\cite{gvwGVW2010,gvwGVW2011,gvwGVW2013}\\

\sba{} & coincides with \gvw{} and \ap{}, $\schl$ only & \ref{sec:sb} & \cite{rs-2012}\\

\ssg{} & coincides with \sba{}, \gvw{} and \ap{} &
\ref{sec:ssg} & \cite{galkinSimple2012}\\

\gggv{} & uses Buchberger's Product and Chain criterion in \ggv{} (this is also
    introduced in the 2013 revision of \gvw{}) &
\ref{sec:buch} & \cite{gerdtHashemiG2V,gvwGVW2013}\\

\ffs{} & uses \ffo{}-style \sreduction{} & \ref{sec:f4-f5} & \cite{apF452010}\\
\hline
\caption{Variants of \ff{} and their modifications (in the order of appearance
    in this survey)}
\label{table:variants-overview}
\end{longtable}
\end{center}

\begin{landscape}
\begin{figure}
\begin{center}
\begin{tikzpicture}[scale=.61,transform shape]
\tikzset{level1 concept/.append style={level distance=350,sibling 
    angle=10, minimum size = 2cm}}
\tikzset{level2 concept/.append style={level distance=350,sibling 
    angle=10, minimum size = 4cm}}
\tikzset{level3 concept/.append style={level distance=350,sibling 
    angle=10, minimum size = 4cm}}
\tikzset{level4 concept/.append style={level distance=350,sibling 
    angle=10, minimum size = 4cm}}

\node[dashed, circle, radius=3cm, text
centered,  inner ysep=1em, inner xsep=1em,
  draw=white, fill=, fill opacity=0, draw opacity=0.5, text opacity=1,
  align=center] (trb) at (11,8.5)
{
\begin{minipage}{50pt}
\begin{center}
\textcolor{ggrey}{\large{\textbf{\trb{}\\\cite{huang2010}\\(2010)}}}
\end{center}
\end{minipage}
};

\node[dashed, circle, radius=3cm, text
centered,  inner ysep=1em, inner xsep=1em,
  draw=white, fill=, fill opacity=0, draw opacity=0.5, text opacity=1,
  align=center] (gbgc) at (12.5,6)
{
\begin{minipage}{50pt}
\begin{center}
\textcolor{ggrey}{\large{\textbf{\gbgc{}\\\cite{sw2011b}\\(2011)}}}
\end{center}
\end{minipage}
};

\node[dashed, circle, radius=3cm, text
centered,  inner ysep=1em, inner xsep=1em,
  draw=white, fill=, fill opacity=0, draw opacity=0.4, text opacity=1,
  align=center] (rba) at (13.5,3.25)
{
\begin{minipage}{50pt}
\begin{center}
\textcolor{ggrey}{\large{\textbf{\rba{}\\\cite{erF5SB2013}\\(2013)}}}
\end{center}
\end{minipage}
};

  [huge mindmap]
  [decoration={start radius=6cm,end radius=5cm,amplitude=7mm,angle=20}]
  \path[mindmap,concept color=ggrey!85!white,text=white,
  minimum size=5cm,outer sep=0pt]
    node[concept] (ff) { {\Large \textbf{\ff{}}\\\cite{fF52002Corrected}\\(2002)}}
    [clockwise from=180]
    child[grow=65,concept color=myblued!35!ggrey,text=white] { node[concept] 
      (ffprpr) { {\scriptsize \textbf{\ffprpr{}}\\\cite{fF52002Corrected}\\(2002)}}
      child[grow=0,concept color=myblued!55!ggrey,text=white] { node[concept] 
        at (2,1.5)
        (extended) { {\scriptsize {\textbf{Extended \ff{}
          Criteria}}\\\cite{ahExtF52010}\\(2010)}}
      }
      child[grow=0,concept color=myblued!50!mygreend,text=white] { node[concept] 
        at (0,3)
        (ffgen) {
          {\scriptsize{\textbf{\ffgen{}}}\\\cite{panhuwang2012,panhuwang2013}\\(2013)}}
      }
    }
    child[concept color=mygreend!35!myredd,text=white] { node[concept] at
      (-6.5,-9.5)
      (fftwo) { {\scriptsize{\textbf{\fftwo{}}}\\\cite{FJ03}\\(2003)}}}
    child[grow=-98,concept color=myyellowd!20!myredd,text=white] { node[concept] 
      (ffb) { {\scriptsize{\textbf{\ffb{}}}\\\cite{arsPhd2005}\\(2005)}}
      child[concept color=myyellowd!30!myredd,text=white] { node[concept] at
        (-3,0.2)
        (fft) { \scriptsize{{\textbf{\fft{}}}\\\cite{gashPhD2008,gashF5t2009}\\(2009)}}}
      child[grow=259,concept color=myyellowd!40!myredd,text=white] { node[concept] 
        (ffp) { {\scriptsize{\textbf{\ffp{}}}\\\cite{egpF52011}\\(2011)}}}
    }
    child[concept color=mygreend!45!myredd,text=white] { node[concept] at
      (-14,-2.6)
      (sagbi) { {\scriptsize {\textbf{SAGBI \grobner{} bases}}\\\cite{FR09}\\(2009)}}
      child[grow=0,concept color=mygreend!55!myredd,text=white] { node[concept]
      at (-7,-2)
        (symmetry1) { {\scriptsize {\textbf{globally invariant (group
            action)}}\\\cite{faugere-svartz-2012}\\(2012)}}
      child[grow=0,concept color=mygreend!65!myredd,text=white, minimum
        size=2.2cm] { node[concept] 
        at (-1,-3)
        (symmetry2) { {\scriptsize {\textbf{invariant (group
            action)}}\\\cite{faugeresvartz-2013}\\(2013)}}}
      }
    }
    child[concept color=mypinkd!15!ggrey,text=white] { node[concept] 
    at (-4,10)
      (matrixff) { {\scriptsize {\textbf{\mff{}}}\\\cite{FJ03}\\(2003)}}
    }
    child[concept color=mypinkd!25!ggrey,text=white] { node[concept] 
    at (-1.9,11.5)
      (ffs) { {\scriptsize {\textbf{\ffs{}}}\\\cite{apF452010}\\(2010)}}
    }
    child[concept color=mygreenyellow!25!ggrey,text=white] { node[concept] at
      (-4.5,0)
      (bihomog) { {\scriptsize {\textbf{bihomogeneous case}}\\\cite{FSS10b}\\(2010)}}
    }
    child[concept color=mygreenyellow!35!ggrey,text=white] { node[concept] 
      at (-9.5,-2.5)
      (quasihomog) { {\scriptsize {\textbf{quasi\-homogeneous
        case}}\\\cite{faugeresafeyverron-2013}\\(2013)}}
    }
    child[concept color=ggrey,text=white] { node[concept] 
      at (-13,2.2)
      (involutive) { {\scriptsize {\textbf{involutive
        bases}}\\\cite{gerdtHashemiG2V,ghmInvolutiveF5-2013}\\(2013)}}
    }
    child[grow=0,concept color=mygreend!60!ggrey,text=white] { node[concept] 
      at (0,3)
      (ap) { {\scriptsize \textbf{\ap{}}}\\\cite{apF5CritRevised2011}\\(2009)}
      [clockwise from=-30]
      child[grow=0,concept color=mygreend!50!ggrey,text=white] { node[concept] 
        at (1.3,-2.5)
        (gvw2) { {\scriptsize
          \textbf{\gvw{}(HS)}\\\cite{gvwGVW2011,volnyGVW2011,gvwGVW2013}\\(2011)}}
        child[grow=0,concept color=mygreend!10!ggrey,text=white, minimum
      size=2cm] { node[concept] 
        at (0.3,0.4)
          (solv) { {\scriptsize {\textbf{on\\solvable algebras}}\\\cite{swmz2012}\\(2012)}}
        }
      }
      child[grow=0,concept color=mygreend!50!ggrey,text=white] { node[concept] 
        at (0,0)
        (sba) { {\scriptsize \textbf{\sba{}}\\\cite{rs-2012}\\(2012)}}}
      child[grow=0,concept color=mygreend!50!ggrey,text=white] { node[concept] 
        at (1.7,1.9)
        (ssg) { {\scriptsize \textbf{\ssg{}}\\\cite{galkinSimple2012}\\(2012)}}}
    }
    child[grow=-24, concept color=mypinkd!45!ggrey,text=white] {
      node[concept] (ggv) { \scriptsize{\textbf{\ggv{}}\\\cite{ggvGGV2010}\\(2010)}}
      child[grow=64, concept color=myblued!10!mydpink] { node[concept] (gvw1) 
        {{\scriptsize \textbf{\gvw{}(v1)}\\\cite{gvwGVW2010}\\(2010)}}
      }
      child[grow=5, concept color=mypinkd!60!ggrey] { node[concept] (iggv) 
        {\scriptsize{\textbf{\iggv{}}\\\cite{ederImprovedF52013}\\(2012)}}
      child[grow=0, concept color=mypinkd!75!ggrey, minimum size=1.7cm] { node[concept] 
        at (1,-0.3)
        (gggv) 
        {\scriptsize{\textbf{\gggv{}}\\\cite{gerdtHashemiG2V}\\(2013)}}}
      }
    }
    child[grow=-55,concept color=myredd!35!ggrey,text=white] { node[concept] 
      (ffpr) { {\scriptsize \textbf{\ffpr{}}\\\cite{fF52002Corrected}\\(2002)}}
      [clockwise from=-30]
    child[grow=-45,concept color=myredd!35!ggrey,text=white] { node[concept] 
      (ffr) { {\scriptsize \textbf{\ffr{}}\\\cite{stF5Rev2005}\\(2005)}}
      [clockwise from=-30]
    child[grow=0,concept color=myredd!35!ggrey,text=white, minimum size=1.7cm] { node[concept] 
      at (0.4,0)
      (ffc) { {\scriptsize \textbf{\ffc{}}\\\cite{epF5C2009}\\(2009)}}
      [clockwise from=-30]
      child[grow=0,concept color=myredd!50!ggrey,text=white, minimum size=1.7cm] { node[concept] 
        at (-3,2.2)
        (ffa) { {\scriptsize \textbf{\ffa{}}\\\cite{epSig2011}\\(2011)} }
        child[grow=0,concept color=myredd!60!ggrey,text=white, minimum size=1.7cm] { node[concept] 
        at (0.4,0.2)
          (iffa) { {\scriptsize\textbf{\iffa{}}\\\cite{ederImprovedF52013}\\(2012)}} }
      }
      child[grow=0,concept color=myredd!60!ggrey,text=white, minimum size=1.7cm] { node[concept] 
        at (1.9,2.2)
        (iffc) { {\scriptsize
          \textbf{\iffc{}}\\\cite{ederImprovedF52013}\\(2012)}} }
    }
    }
    }
;

\node[text=ggrey] at (-10,10.5) {
\begin{minipage}{150px}
\begin{center}
\large{\textbf{\color{ggrey}{Modifications not specific to signature-based \grobner{} basis algorithms}}}
\end{center}
\end{minipage}
};

\node[text=ggrey] at (1.5,10.5) {
\begin{minipage}{105px}
\begin{center}
\large{\textbf{\color{ggrey}{Variants covered by this survey}}}
\end{center}
\end{minipage}
};

\node[text=ggrey] at (-2,8) {
\begin{minipage}{105px}
\begin{center}
\large{\textbf{\color{ggrey}{\ffo{}-style\\reduction}}}
\end{center}
\end{minipage}
};

\node[text=ggrey] at (-3.8,-4.7) {
\begin{minipage}{100px}
\begin{center}
\large{\textbf{\color{ggrey}{Algorithmically\\ensuring\\termination}}}
\end{center}
\end{minipage}
};

\node[text=ggrey] at (-13,-1.2) {
\begin{minipage}{100px}
\begin{center}
\large{\textbf{\color{ggrey}{Exploit\\algebraic\\structures}}}
\end{center}
\end{minipage}
};

\begin{pgfonlayer}{background}
\filldraw[fill=lgrey, draw=vlgrey, fill opacity=0.3, densely dashed, thick]
(8,1) ellipse (25.5em and 35em);
\filldraw[fill=lgrey, fill opacity=0.3, draw=white, rounded corners=3pt, thick]
(-10,-10.7) -- (-11.2,0.3) -- (-4.2,12.8) --
(18.2,12.8) -- (18.2,-10.7) -- (-10,-10.7);
\fill[fill=lgrey, fill opacity=0.2, rounded corners=5pt, thick]
(-16.5,-11) rectangle  (18.5,13.1);

\filldraw[fill=mypinkd!50!ggrey, draw=mypinkd!50!ggrey, fill opacity=0.3, densely dashed, thick]
(-2.7,6.8) ellipse (10.5em and 12em);

\filldraw[fill=orange!50!ggrey, draw=orange!50!ggrey, fill opacity=0.3, densely dashed, thick]
(-2.5,-6.2) ellipse (13.2em and 11.5em);

\filldraw[fill=mygreend!50!ggrey, draw=mygreend!50!ggrey, fill opacity=0.3, densely dashed, thick]
(-11,-2.5) ellipse (13em and 22em);

  \path (iffc) to[circle connection bar switch color=from (myredd!60!ggrey) to 
  (myredd!60!ggrey)] (iffa);
  \path (iffc) to[circle connection bar switch color=from (myredd!60!ggrey) to 
  (mypinkd!60!ggrey)] (iggv);
  \path (iffa) to[circle connection bar switch color=from (myredd!60!ggrey) to 
  (mypinkd!60!ggrey)] (iggv);
  \path (ffa) to[circle connection bar switch color=from (myredd!50!ggrey) to 
  (mypinkd!45!ggrey)] (ggv);
  
  \path (gvw1) to[circle connection bar switch color=from (myblued!10!mydpink) to 
  (mygreend!50!ggrey)] (gvw2);
  
  \path (gggv) to[circle connection bar switch color=from (mypinkd!75!ggrey) to 
  (mygreend!50!ggrey)] (gvw2);
  
  \path (sba) to[circle connection bar switch color=from (mygreend!50!ggrey) to 
  (mygreend!50!ggrey)] (gvw2);
  \path (ssg) to[circle connection bar switch color=from (mygreend!50!ggrey) to 
  (mygreend!50!ggrey)] (gvw2);
  \path (ssg) to[circle connection bar switch color=from (mygreend!50!ggrey) to 
  (mygreend!50!ggrey)] (sba);

  \path (gvw1) to[circle connection bar switch color=from (myblued!10!mydpink) to 
  (myblued!35!ggrey)] (ffprpr);
  \path (gvw1) to[circle connection bar switch color=from (myblued!10!mydpink) to 
  (myblued!55!ggrey)] (extended);
  \path (ffgen) to[circle connection bar switch color=from (myblued!50!mygreend) to 
  (myblued!55!ggrey)] (extended);
  \path (ffgen) to[circle connection bar switch color=from (myblued!50!mygreend) to 
  (mygreend!60!ggrey)] (ap);

  \path (fft) to[circle connection bar switch color=from (myyellowd!30!myredd) to 
  (myyellowd!40!myredd)] (ffp);

  \path (matrixff) to[circle connection bar switch color=from (mypinkd!15!ggrey) to 
  (mypinkd!25!ggrey)] (ffs);
\end{pgfonlayer}
\end{tikzpicture}
\end{center}
\caption{A good decade in signature-based \grobner{} basis algorithms (status:
    March 2014)}
\label{fig:decade-in-sig-based-algorithms}
\end{figure}
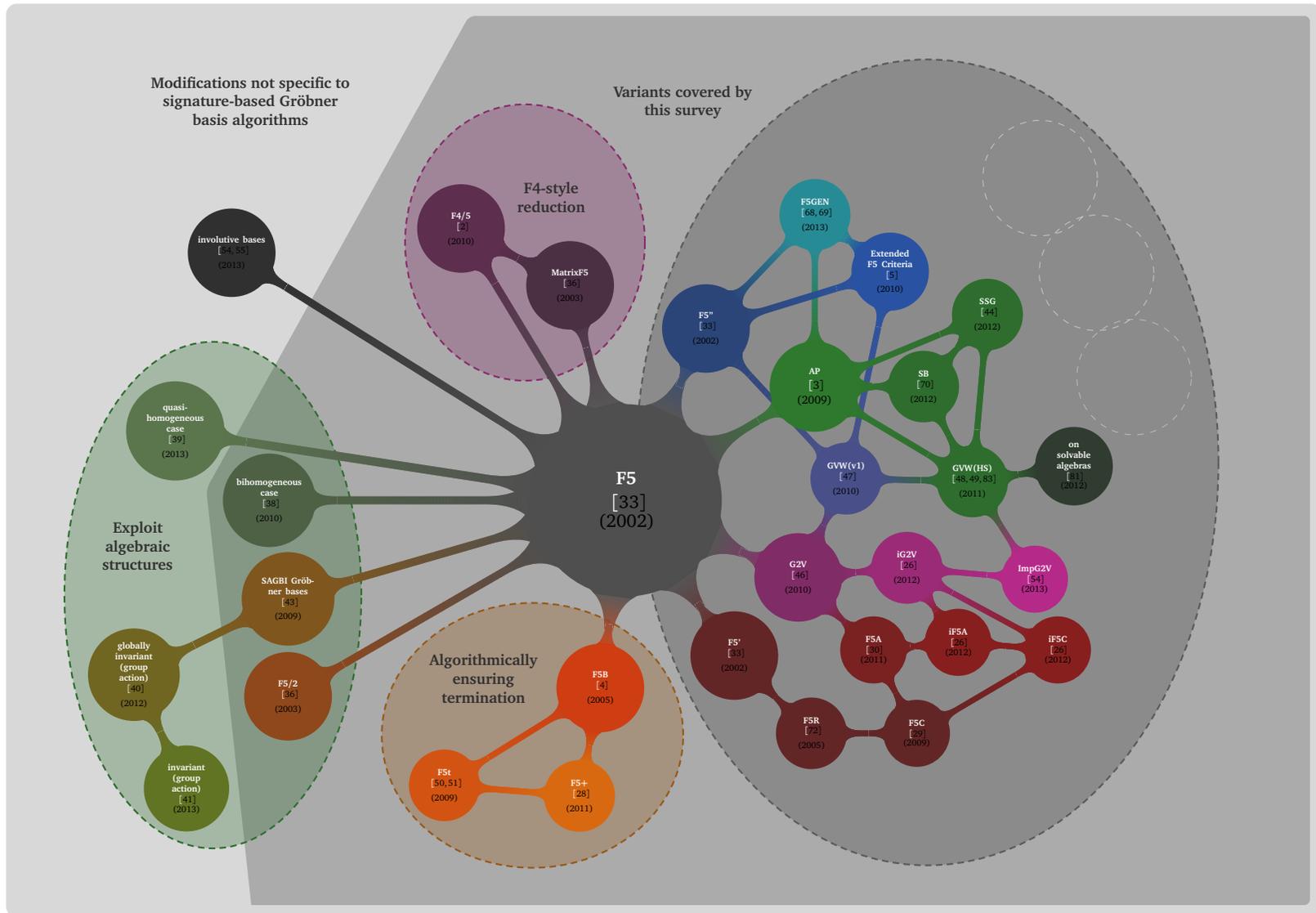
\end{landscape}

\section{Notations and terminology}
\label{sec:notations}
In this section we introduce notations and basic terminology used in this
survey. Readers already familiar with signature-based algorithms
might skip this section. Still note that notations itself play an important role
in the following, especially when comparing different variants of
signature-based algorithms. We extend the notation introduced in~\cite{erF5SB2013}.

Let $\ring$ be a polynomial ring over a field $\field$. All polynomials
$f\in \ring$ can be uniquely written as a finite sum $f=\sum_{\const_v x^v\in
  \mon}\const_v x^v$ where $\const_v \in\field$, $x^v\defeq\prod_ix_i^{v_i}$ and $\mon$ is
minimal. The elements of $\mon$ are the \emph{terms} of $f$. A
\emph{monomial} is a polynomial with exactly one term. A monomial with
a coefficient of 1 is \emph{monic}. Neither monomials nor terms of
polynomials are necessarily monic. We write $f\simeq g$ for $f,g\in \ring$
if there exists a non-zero $\const \in\field$ such that $f=\const g$.

Let $\module$ be a free $\ring$-module and let $\mgen$ be
the standard basis of unit vectors in $\module$. All module elements
$\alpha\in \module$ can be uniquely written as a finite sum
$\alpha=\sum_{a\mbasis i\in \modmon}a\mbasis i$ where the $a$ are monomials
and $\modmon$ is minimal. The elements of $\modmon$ are the \emph{terms} of
$\alpha$. A \emph{module monomial} is an element of $\module$ with exactly
one term. A module monomial with a coefficient of 1 is
\emph{monic}. Neither module monomials nor terms of module elements
are necessarily monic. Let $\alpha\simeq\beta$ for
$\alpha,\beta\in \module$ if $\alpha=\const \beta$ for some non-zero $\const \in\field$.

Let $\leq$ denote two different orders -- one for $\ring$ and one for
$\module$:
The order for $\ring$ is a monomial order, which means that it is a
well-order on the set of monomials in $\ring$ such that $a\leq b$ implies
$ca\leq cb$ for all monomials $a,b,c\in \ring$.
The order for $\module$ is a module monomial order which means that it
is a well-order on the set
of module monomials in $\module$ such that $S \leq T$ implies $cS\leq cT$
for all module monomials $S,T\in \module$ and monomials $c\in \ring$.
We require the two orders to be \emph{compatible} in the sense that $a\leq b$ if
and only if $a\mbasis i\leq b\mbasis i$ for all monomials $a,b\in \ring$
and $i=1,\ldots,m$.

Consider a finite sequence of polynomials $\gen f \in \ring$ that
we call the \emph{input (polynomials)}. We call $\gen f$ a regular sequence if
$f_i$ is a non-zero-divisor on $\ring / \ideal{f_1,\ldots,f_{i-1}}$ for
$i=2,\ldots,m$.
For $\alpha = \sum_{i=1}^m a_i \mbasis i$, $a_i \in \ring$ we define the homomorphism
$\alpha\mapsto\proj\alpha$ from $\module$ to $\ring$ by
$\proj\alpha\defeq\sum_{i=1}^m a_if_i$. An element $\alpha\in \module$
with $\proj\alpha=0$ is called a \emph{syzygy}. The module of all syzygies of
$\gen f$ is denoted by $\syz{\gen f}$.

Next we introduce the notion of signatures together with related
structures in the plain polynomial setting.

\begin{definition}\
\label{def:signature}
\begin{enumerate}
\item The \emph{lead term} $\hd f$ of $f\in
\ring\setminus\set 0$ is the $\leq$-maximal term of $f$. The \emph{lead
  coefficient} $\lc f$ of $f$ is the coefficient of $\hd f$. For a set $F \in
  \ring$ we define the \emph{lead ideal of $F$} by $L(F) := \ideal{\hd f \mid f
    \in F}.$
\item The \emph{lead term} resp. \emph{signature} $\sig\alpha$ of
$\alpha\in \module\setminus\set 0$ denotes the
$\leq$-maximal term of $\alpha$. If $a\mbasis i=\sig\alpha$ then we call
$\ind{\alpha}\defeq i$ the \emph{index} of $\alpha$.
\item For $\alpha \in \module$ we define the \emph{sig-poly pair} of $\alpha$ by
$\spp\alpha \in \module \times \ring$.
\item $\alpha,\beta\in \module$ are \emph{equal up to
  sig-poly pairs} if $\sig\alpha=\sig{\const \beta}$ and
$\proj\alpha=\proj{\const\beta}$ for some non-zero $\const\in\field$.
Correspondingly, $\alpha,\beta$ are said to be \emph{equal up to
  sig-lead pairs} if $\sig\alpha=\sig{\const\beta}$ and
$\hdp\alpha=\hdp{\const\beta}$ for some non-zero $\const \in\field$.
\end{enumerate}
\end{definition}

With these definitions
every non-syzygy module element $\alpha\in \module$ has two main
associated characteristics -- the signature $\sig\alpha\in \module$ and
the lead term $\hdp\alpha\in \ring$ of its image $\proj\alpha$.
Lead terms and signatures include a coefficient for mathematical
convenience, though an implementation of an signature-based \grobner{} Basis
algorithm need not store the signature coefficients as we discuss in
Sections~\ref{sec:efficient-implementations-f5}
and~\ref{sec:efficient-implementations-sb}.

We define some canonical module monomial orders that are
useful in the following.
\begin{definition}
Let $<$ be a monomial order on $\ring$ and let $a e_i, b e_j$ be two module
monomials in $\module$.
\begin{enumerate}
\item $a e_i \potl b e_j$ if and only if either $i < j$ or $i = j$ and $a < b$.
\item $a e_i \topl b e_j$ if and only if either $a < b$ or $a = b$ and $i < j$.
\end{enumerate}
These two orders can be combined with either a weighted degree or a weighted
leading monomial:
\begin{enumerate}
\item $a e_i \dpotl b e_j$ if and only if either $\deg\left(\proj{a e_i}\right)
< \deg\left(\proj{b e_j}\right)$ or
$\deg\left(\proj{a e_i}\right) = \deg\left(\proj{b e_j}\right)$ and $a\mbasis i
\potl b \mbasis j$. In the same way we define $a \mbasis i \dtopl b \mbasis j$.
\item $a e_i \hdpotl b e_j$ if and only if either $\hdp{a e_i} < \hdp{b e_j}$ or
$\hdp{a e_i} = \hdp{b e_j}$ and $a \mbasis i \potl b \mbasis j$. In the same way
we define $a\mbasis i \hdtopl b \mbasis j$.
\end{enumerate}
Note that $\hdpotl$ is also known as Schreyer's order, for example,
see~\cite{gpSingularBook2007}.
\end{definition}

The above introduced notation of the orders represent that the position in the module resp. the
lead term in the polynomial ring are preferred.

\begin{example}
Note that a polynomial can have infinitely many different module representations with
distinct signatures. Consider the three input polynomials
$f_1 = x^2 - y^2$, $f_2 = xyz -z^3$, and $f_3 = yz^2 - xy$
in $\ring=\Q[x,y,z]$ where $<$ denotes the graded
reverse lexicographical monomial order. Moreover, assume $<$ to extend to
$\potl$ on the set of monomials of $\ring^3$.
For example, we can represent $f_2$ by $\mbasis 2$. Since $\proj{f_1\mbasis 3 -
f_3\mbasis 1} = 0$ another representation of $f_2$ might be $f_1\mbasis{3} +
\mbasis 2 - f_3\mbasis{1}$. Note that the two representations of $f_2$ have two
different signatures, $\mbasis 2$ and $\hd{f_1} \mbasis 3$, respectively.
We also want to point out that $\hdp{\sig\alpha}\neq\hdp\alpha$ is possible: In the above
example $\hdp{\sig{\mbasis 2}} = \hd{f_2}$, but
$\hdp{\sig{f_1\mbasis 3 + \mbasis 2 - f_3 \mbasis 1}} = \hd{\hd{f_1} f_3}
\neq \hd{f_2}$.
\end{example}

Finally, we introduce the notion of \grobner{} bases. For this, the reduction
of polynomials is essential.

\begin{definition}
\label{def:reduce}
Let $f \in \ring$ and let $t$ be a term of~$f$. Then we
can \emph{reduce} $t$ by $g \in \ring$ if there exists a monomial $b$
such that $\hd{b g}=t$.
The outcome of the reduction step is then $f-bg$ and
$g$ is called the \emph{reducer}.  When $g$ reduces $t$ we also say for convenience
that $bg$ reduces $f$. That way $b$ is introduced
implicitly instead of having to repeat the equation $\hd{bg}=t$.
\end{definition}

The result of an reduction of $f \in \ring$ is an element $h \in \ring$ that has
been calculated from $f$ by a sequence of reduction steps. Thus, reductions can
always be assumed to be done w.r.t. some finite subset $G \subset \ring$.

\begin{definition}
\label{def:gb}
Let $I= \langle \gen f \rangle$ be an ideal in $\ring$.
A finite subset $G$ of $\ring$ is a \emph{\grobner{} basis up to degree $d$}
for $I$ if $G \subset I$ and for all $f \in I$ with $\deg(f)\leq d$ $f$ reduces
to zero w.r.t. $G$. $G$ is a \emph{\grobner{} basis} for $I$ if $G$ is a
\grobner{} basis in all degrees.
\end{definition}

In the very same way one can define \grobner{} basis with the notion of
standard representations:

\begin{definition}
\label{def:standard-representation}
Let $f \in \ring$ and $G \subset \ring$ finite.
A representation $f = \sum_{i=1}^k m_i g_i$ with monomials $m_i \neq 0$, $g_i
\in G$ pairwise different is called a \emph{standard representation} if
\[\max_\leq\left\{\hd{m_ig_i} \mid 1 \leq i \leq k\right\} \leq \hd f.\]
\end{definition}

One can show that if for any $f \in \ideal G$ with $f\neq 0$ $f$ has a standard
representation w.r.t. $G$ and
$\leq$ then $G$ is a \grobner{} basis for $\ideal G$. Moreover, note that the
existence of a standard representation does not imply reducibility to zero, see,
for example, Exercise~5.63 in~\cite{bwkGroebnerBases1993}).

Luckily, Buchberger also gave an algorithmic description of \grobner{} bases using
the notion of so-called S-polynomials:

\begin{definition}
\label{def:spoly}
Let $f \neq 0,g\neq 0 \in \ring$ and let $\lambda = \lcm\left(\hd f,\hd g\right)$ be the monic
least common multiple of $\hd f$ and $\hd g$. The
\emph{S-polynomial} between $f$ and $g$ is given by
\[\spoly{f}{g} \defeq \frac{\lambda}{\hd f} f - \frac{\lambda}{\hd g}g.\]
\end{definition}

\begin{theorem}[Buchberger's criterion]
\label{thm:gb}
Let $I= \langle \gen f \rangle$ be an ideal in $\ring$.
A finite subset $G$ of $\ring$ is a \emph{\grobner{} basis}
for $I$ if $G \subset I$ and for all $f,g \in G$ $\spoly f g$ reduces
to zero w.r.t. $G$.
\end{theorem}

\section{Matrix \ff{}}
\label{sec:matrix-f5}
Before we approach signature-based \grobner{} basis algorithms theoretically let
us look at a small \grobner{} basis computation. 
We start with a slightly simplified version of the \ff{} algorithm, the \mff{}.
With this introduction to the topic we are able to give an easy description of the main
ideas behind the classification of signature-based algorithms which is discussed in detail 
later on.
In order to keep this section plain and easy we keep signature-based details at
a minimum and focus on presenting their usefulness discarding useless elements
from the computation. 

Descriptions of \mff{} can be also found, for example, in~\cite{bardetPhD,FR09}. It
is first publicly mentioned in~\cite{FJ03} and known for breaking challenge $1$
of the hidden field equations (HFE) crypto system.

Algebraic systems are solved by computing a \grobner{} basis for a corresponding
ideal, \cite{bGroebner1965,buchberger2ndCriterion1985}. The link between solving
such systems and linear algebra is already very old, see, for example,
\cite{macaulay1916,laz83}. In~1999 Faug\`ere introduced the \ffo{} algorithm,
\cite{fF41999}. A simplified description of this algorithm using signature-based
criteria is \mff{} which we present here. The important fact is that polynomial
reduction coincides with Gaussian elimination in \mff{} and thus the process of
computing the basis can be illustrated nicely.

Let $I = \ideal{\gen f} \subset \ring$ be the \emph{homogeneous} input ideal.
We want to compute a \grobner{} basis for $I$ w.r.t. a given monomial order $<$. 
The idea is to incrementally construct \emph{Macaulay matrices} $M_d$ which are
generalizations of the Sylvester matrix for finitely many ($>2$ possible),
multivariate polynomials. In the above setting the rows of $M_d$ represent the
polynomials $t_{j,k} f_k$ where $t_{j,k}$ are monomials in $\ring$ such that
$\deg(t_{j,k} f_k) \leq d$ for all $1 \leq k \leq m$.
The columns of $M_d$ are labelled by all possible terms $x^v$ such that
$\deg(x^v) \leq d$. Moreover, the columns are sorted, from left to right, by
decreasing monomial order $<$. Thus a row of $M_d$
labelled by $t_{j,k} f_k= \sum_{x^v \in \mon,\\ \deg(x^v) \leq d} \const_v x^v$
has in column $x^v$ entry $\const_v \in \field$. Note that by the this representation of
$t_{j,k} f_k$ $\const_v = 0$ is possible. Once $M_d$ is generated, the row
echelon form $N_d$ of $M_d$ is computed. The rows of $N_d$ now correspond to
polynomials in $\ring$ that generate a \grobner{} basis up to degree $d$ for
$I$. So, in contrast to \grobner{} basis algorithms in the vein of Buchberger's
description, \mff{} needs another parameter, a degree bound $D$ up to which the
computations are carried out. We introduce the variant of this algorithm using
signature-based criteria to improve computations by an example.

Consider the three homogeneous input polynomials $f_1= y^2+4yz$, $f_2 = 2x^2+3xy+4y^2+3z^2$,
and $f_3 = 3x^2+4xy+2y^2$ in $\ring=\F_5[x,y,z]$ where $<$ denotes the graded
reverse lexicographical monomial order. By the above description it is clear
that the labels $t_{j,k} f_k$ of the rows coincide with the corresponding
signatures $t_{j,k} \mbasis k$. We want to use these
signatures to label the rows of the Macaulay matrices built in the following.
Thus we need to extend $<$ on $\ring^3$, say we use $\potl$. Let us assume we
want to compute a \grobner{} basis up to degree $D=4$.

\begin{figure}[h]
\begin{center}
\begin{tikzpicture}[]
\begin{scope}[shift={(0,7)}]
\matrix (M) [ matrix of math nodes,left delimiter={(},right delimiter={)},
              column sep=0.7em]{ 
                              3 & 4 & 2 & \zero & \zero & \zero \\[0.5em]
                              2 & 3 & 4 & \zero & \zero & 3 \\[0.5em]
                              \zero & \zero & 1 & \zero & 4 & \zero \\
};
\matBorder {3,2,1} 2 2 M
\end{scope}
\begin{scope}[shift={(7,7)}]
\matrix (M2) [ matrix of math nodes,left delimiter={(},right delimiter={)},
              column sep=0.7em]{ 
                              3 & 4 & 2 & \zero & \zero & \zero \\[0.5em]
                              2 & 3 & \zero & \zero & 4 & 3 \\[0.5em]
                              \zero & \zero & 1 & \zero & 4 & \zero \\
};
\matTopBorder {3,2,1} 2 2 {M2}
\node[font=\scriptsize,right] (\matName-1-0) [left=1.5em of
  \matName-1-1] {\bo{\mbasis 3}}; 
\node[font=\scriptsize,right] (\matName-2-0) [left=1.5em of
  \matName-2-1] {\bo{\mbasis 2 + \mbasis 1}}; 
\node[font=\scriptsize,right] (\matName-3-0) [left=1.5em of
  \matName-3-1] {\bo{\mbasis 1}}; 
\reducedRows 2 2 6 {M2}
\end{scope}
\begin{scope}[shift={(0,3)}]
\matrix (M3) [ matrix of math nodes,left delimiter={(},right delimiter={)},
              column sep=0.7em]{ 
                              3 & 4 & \zero & \zero & 2 & \zero \\[0.5em]
                              2 & 3 & \zero & \zero & 4 & 3 \\[0.5em]
                              \zero & \zero & 1 & \zero & 4 & \zero \\
};
\matTopBorder {3,2,1} 2 2 {M3}
\node[font=\scriptsize,right] (\matName-1-0) [left=1.5em of
  \matName-1-1] {\bo{\mbasis 3 + 3\mbasis 1}}; 
\node[font=\scriptsize,right] (\matName-2-0) [left=1.5em of
  \matName-2-1] {\bo{\mbasis 2 + \mbasis 1}}; 
\node[font=\scriptsize,right] (\matName-3-0) [left=1.5em of
  \matName-3-1] {\bo{\mbasis 1}}; 
\reducedRows 1 1 6 {M3}

\end{scope}
\begin{scope}[shift={(7,3)}]
\matrix (M4) [ matrix of math nodes,left delimiter={(},right delimiter={)},
              column sep=0.7em]{ 
                              \zero & 2 & \zero & \zero & 1 & 3 \\[0.5em]
                              2 & 3 & \zero & \zero & 4 & 3 \\[0.5em]
                              \zero & \zero & 1 & \zero & 4 & \zero \\
};
\matTopBorder {3,2,1} 2 2 {M4}
\node[font=\scriptsize,right] (\matName-1-0) [left=1.5em of
  \matName-1-1] {\bo{\mbasis 3 + \mbasis 2 + 4\mbasis 1}}; 
\node[font=\scriptsize,right] (\matName-2-0) [left=1.5em of
  \matName-2-1] {\bo{\mbasis 2 + \mbasis 1}}; 
\node[font=\scriptsize,right] (\matName-3-0) [left=1.5em of
  \matName-3-1] {\bo{\mbasis 1}}; 
\reducedRows 1 1 6 {M4}

\end{scope}
\begin{scope}[shift={(3.5,-1)}]
\matrix (M5) [ matrix of math nodes,left delimiter={(},right delimiter={)},
              column sep=0.7em]{ 
                              \zero & 2 & \zero & \zero & 1 & 3 \\[0.5em]
                              2 & \zero & \zero & \zero & \zero & 1 \\[0.5em]
                              \zero & \zero & 1 & \zero & 4 & \zero \\
};
\matTopBorder {3,2,1} 2 2 {M5}
\node[font=\scriptsize,right] (\matName-1-0) [left=1.5em of
  \matName-1-1] {\bo{\mbasis 3 + \mbasis 2 + 4\mbasis 1}}; 
\node[font=\scriptsize,right] (\matName-2-0) [left=1.5em of
  \matName-2-1] {\bo{\mbasis 3 + 2\mbasis 2 + 5 \mbasis 1}}; 
\node[font=\scriptsize,right] (\matName-3-0) [left=1.5em of
  \matName-3-1] {\bo{\mbasis 1}}; 
\reducedRowsWrong 2 2 6 {M5}

\end{scope}
\draw[->,draw=vlgrey,fill=vlgrey,shorten >=2em,shorten <=1.5em]
 (M) edge [out=25,in=155] (M2);
\draw[->,draw=vlgrey,fill=vlgrey,shorten >=2em,shorten <=1.5em]
 (M2) edge [out=-150,in=80] (M3);
\draw[->,draw=vlgrey,fill=vlgrey,shorten >=2em,shorten <=1.5em]
 (M3) edge [out=-25,in=205] (M4);
\draw[->,draw=vlgrey,fill=vlgrey,shorten >=2em,shorten <=0.5em]
 (M4) edge [out=-110,in=80] (M5);
\end{tikzpicture}
\end{center}
\caption{Computing the row echelon form of $M_2$.}
\label{fig:row-echelon-form-deg-2-mf5}
\end{figure}

The main idea of using Macaulay matrices is now to calculate all possible elements in $I$ for
a given degree $d$. In Buchberger's attempt (Theorem~\ref{thm:gb}) one considers
S-polynomials of degree $d$ and has to find reducers of these. Here we do not need to search
for such elements, all possible reducers are already in $M_d$. So we can focus
on the main question: How do signature-based criteria work to improve
\grobner{} basis computations?

Let us start with the lowest possible degree, $d=2$. Building the Macaulay matrix
$M_2$ in Figure~\ref{fig:row-echelon-form-deg-2-mf5} we label the rows by the
corresponding signatures. Throughout the steps of reducing $M_2$ we keep track
in the label of the rows what computational steps have been done.

One of the reduction steps differs, the last step: Looking at the label of the second
row after reducing it with the first row we see that there is a change:
\[ \mbasis 2 + \mbasis 1 \quad\Longrightarrow\quad \mbasis 3 + 2\mbasis 2 + 5\mbasis 1.\]
Since the labels change also in the other reduction steps, the question is, what
is special in this step ? Looking at the lead term of the module element we see the
difference: Before the reduction the label of the row has a lead term of
$\mbasis 2$ w.r.t. $\potl$, afterwards it is $\mbasis 3$. In none of the other
reduction steps above the lead term changed. And that is the general idea of
the signature: We want to easily keep track of where the new rows are coming
from. Storing the complete module representation as done above, the overhead of
computing a \grobner{} basis is too big (see also~\cite{mmtSyzygies1992}). Thus,
instead of keeping the full module representation, we only store the lead
term of it, the signature. Applied to our example above the last step would lead
to the situation illustrated in Figure~\ref{fig:signature-problem-deg-2-mf5}.

In other words, we would loose the connection between the second row and $\mbasis
2$ resp. $f_2$. As we see in following, to remember this connection is crucial
for the strength of signature-based criteria to remove useless computations.

We agree to not do any such reduction. In terms of the Macaulay matrix
this means that
\begin{enumerate}
\item rows are sorted from top to bottom by decreasing signatures, and
\item the row we reduce with must be below the row to be reduced.
\end{enumerate}

\begin{figure}[h]
\begin{center}
\hspace*{-15mm}
\begin{tikzpicture}[]
\begin{scope}[shift={(0,0)}]
\matrix (M4) [ matrix of math nodes,left delimiter={(},right delimiter={)},
              column sep=0.7em]{ 
                              \zero & 2 & \zero & \zero & 1 & 3 \\[0.5em]
                              2 & 3 & \zero & \zero & 4 & 3 \\[0.5em]
                              \zero & \zero & 1 & \zero & 4 & \zero \\
};
\matTopBorder {3,2,1} 2 2 {M4}
\node[font=\scriptsize,right] (\matName-1-0) [left=1.5em of
  \matName-1-1] {\bo{\mbasis 3}}; 
\node[font=\scriptsize,right] (\matName-2-0) [left=1.5em of
  \matName-2-1] {\bo{\mbasis 2}}; 
\node[font=\scriptsize,right] (\matName-3-0) [left=1.5em of
  \matName-3-1] {\bo{\mbasis 1}}; 

\end{scope}
\begin{scope}[shift={(7,0)}]
\matrix (M5) [ matrix of math nodes,left delimiter={(},right delimiter={)},
              column sep=0.7em]{ 
                              \zero & 2 & \zero & \zero & 1 & 3 \\[0.5em]
                              2 & \zero & \zero & \zero & \zero & 1 \\[0.5em]
                              \zero & \zero & 1 & \zero & 4 & \zero \\
};
\matTopBorder {3,2,1} 2 2 {M5}
\node[font=\scriptsize,right] (\matName-1-0) [left=1.5em of
  \matName-1-1] {\bo{\mbasis 3}}; 
\node[font=\scriptsize,right] (\matName-2-0) [left=1.5em of
  \matName-2-1] {\bo{\mbasis 3}}; 
\node[font=\scriptsize,right] (\matName-3-0) [left=1.5em of
  \matName-3-1] {\bo{\mbasis 1}}; 
\reducedRowsWrong 2 2 6 {M5}
\draw [draw=vlgrey,rounded corners=2pt,densely dashed]
  (M4-2-0.south east) rectangle (M4-2-0.north west);
\draw[->,draw=vlgrey,fill=vlgrey,shorten >=0.3em,shorten <=0.3em]
 (M4-2-0) edge [out=-150,in=-120] (M5-2-0);
\end{scope}
\end{tikzpicture}
\end{center}
\vspace*{-10mm}
\caption{Change of signature due to a reduction step.}
\label{fig:signature-problem-deg-2-mf5}
\end{figure}

Thus for our purpose to keep the signatures, the row echelon form of $M_2$ 
received by restricting reductions is

\begin{center}
\begin{tikzpicture}[]
\begin{scope}[shift={(0,0)}]
\matrix (N) [ matrix of math nodes,left delimiter={(},right delimiter={)},
              column sep=0.7em]{ 
                              \zero & 2 & \zero & \zero & 1 & 3 \\[0.5em]
                              2 & 3 & \zero & \zero & 4 & 3 \\[0.5em]
                              \zero & \zero & 1 & \zero & 4 & \zero \\
};
\matBorder {3,2,1} 2 2 {N}
\node[font=\scriptsize,right] (\matName-1-0) [left=1.5em of
  \matName-1-1] {\bo{\mbasis 3}}; 
\node[font=\scriptsize,right] (\matName-2-0) [left=1.5em of
  \matName-2-1] {\bo{\mbasis 2}}; 
\node[font=\scriptsize,right] (\matName-3-0) [left=1.5em of
  \matName-3-1] {\bo{\mbasis 1}}; 

\end{scope}
\end{tikzpicture}.
\end{center}

After computing the row echelon form $N_2$ of the Macaulay matrix $M_2$ we get
two new polynomials, namely $f_4 = 2xy+yz+3z^2$ and $f_5 = 2x^2+3xy+4yz+3z^2$,
corresponding to the first and the second row of $N_2$.
$f_3$ and $f_4$ have the same signature $\mbasis 3$, thus we can say that there
is a connection between them. The same holds for $f_2$ and $f_5$.

\begin{figure}[H]
\begin{center}
\begin{tikzpicture}[]
\matrix (M) [ matrix of math nodes,left delimiter={(},right delimiter={)},
              column sep=0.7em]{ 
3 & 4 & 2 & \zero & \zero & \zero & \zero  & \zero & \zero & \zero \\
\zero & 3 & 4 & 2 & \zero & \zero & \zero  & \zero & \zero & \zero \\
\zero & \zero & \zero & \zero  &3 & 4 & 2 &  \zero & \zero & \zero \\[0.5em]
2 & 3 & 4 & \zero & \zero & \zero & \zero  & 3     & \zero & \zero \\
\zero & 2 & 3 & 4 & \zero & \zero & \zero  & \zero & 3     & \zero \\
\zero & \zero & \zero & \zero  & 2 & 3& 4 & \zero & \zero & 3     \\[0.5em]
\zero & \zero & 1 & \zero & \zero & 4 & \zero  & \zero & \zero & \zero \\
\zero & \zero & \zero & 1 & \zero & \zero & 4 & \zero  & \zero & \zero \\
\zero & \zero & \zero & \zero & \zero & \zero & 1 & \zero & 4  & \zero \\
};

\matBorder {3,2,1} 3 2 M
\rewrite 4 3 M 7 3 M
\rewrite 5 4 M 8 4 M
\rewrite 6 7 M 9 7 M
\end{tikzpicture}
\end{center}
\caption{Initial Macaulay matrix $M_3$}
\label{fig:deg-3-mf5}
\end{figure}

\begin{figure}[H]
\begin{center}
\begin{tikzpicture}[]
\begin{scope}[shift={(0,0)}]
\matrix (M) [ matrix of math nodes,left delimiter={(},right delimiter={)},
              column sep=0.7em]{ 
3 & 4 & 2 & \zero & \zero & \zero & \zero  & \zero & \zero & \zero \\
\zero & 3 & 4 & 2 & \zero & \zero & \zero  & \zero & \zero & \zero \\
\zero & \zero & \zero & \zero  &3 & 4 & 2 &  \zero & \zero & \zero \\[0.5em]
2 & 3 & \zero & \zero & \zero & 4 & \zero  & 3     & \zero & \zero \\
\zero & 2 & 3 & \zero & \zero & \zero & 4  & \zero & 3     & \zero \\
\zero & \zero & \zero & \zero  & 2 & 3& \zero & \zero & 4 & 3     \\[0.5em]
\zero & \zero & 1 & \zero & \zero & 4 & \zero  & \zero & \zero & \zero \\
\zero & \zero & \zero & 1 & \zero & \zero & 4 & \zero  & \zero & \zero \\
\zero & \zero & \zero & \zero & \zero & \zero & 1 & \zero & 4  & \zero \\
};

\matBorder {3,2,1} 3 2 M

\rewriteRows 4 6 {10} M
\end{scope}
\begin{scope}[shift={(0,-6)}]
\matrix (N) [ matrix of math nodes,left delimiter={(},right delimiter={)},
              column sep=0.7em]{ 
\zero & 2 & \zero & \zero & \zero & 1 & \zero  & 3 & \zero & \zero \\
\zero & \zero & 2 & \zero & \zero & \zero & 1  & \zero & 3 & \zero \\
\zero & \zero & \zero & \zero  & \zero & 2 & \zero &  \zero & 1 & 3 \\[0.5em]
2 & 3 & \zero & \zero & \zero & 4 & \zero  & 3     & \zero & \zero \\
\zero & 2 & 3 & \zero & \zero & \zero & 4  & \zero & 3     & \zero \\
\zero & \zero & \zero & \zero  & 2 & 3& \zero & \zero & 4 & 3     \\[0.5em]
\zero & \zero & 1 & \zero & \zero & 4 & \zero  & \zero & \zero & \zero \\
\zero & \zero & \zero & 1 & \zero & \zero & 4 & \zero  & \zero & \zero \\
\zero & \zero & \zero & \zero & \zero & \zero & 1 & \zero & 4  & \zero \\
};

\matBorderNoName {3,2,1} 3 2 {N}

\rewriteRows 1 3 {10} {N}
\draw[->,draw=vlgrey,fill=vlgrey,shorten >=1.5em,shorten <=0em]
 (M) edge [out=-90,in=90] (N);
\end{scope}
\end{tikzpicture}
\end{center}
\caption{Rewriting rows: $f_2 \longrightarrow f_5$(top), $f_3 \longrightarrow
  f_4$ (bottom)}
\label{fig:rewrite-deg-3-mf5}
\end{figure}
At this point we have not done any reduction in $M_3$ but just used the
information stored in the signatures. Let us rearrange the rows of $M_3$ to see
how near we are already to a row echelon form:
\begin{center}
\begin{tikzpicture}
\begin{scope}[shift={(0,0)}]
\matrix (N) [ matrix of math nodes,left delimiter={(},right delimiter={)},
              column sep=0.7em]{ 
2 & 3 & \zero & \zero & \zero & 4 & \zero  & 3     & \zero & \zero \\[0.2em]
\zero & 2 & \zero & \zero & \zero & 1 & \zero  & 3 & \zero & \zero \\[0.2em]
\zero & 2 & 3 & \zero & \zero & \zero & 4  & \zero & 3     & \zero \\[0.2em]
\zero & \zero & 2 & \zero & \zero & \zero & 1  & \zero & 3 & \zero \\[0.2em]
\zero & \zero & 1 & \zero & \zero & 4 & \zero  & \zero & \zero & \zero \\[0.2em]
\zero & \zero & \zero & 1 & \zero & \zero & 4 & \zero  & \zero & \zero \\[0.2em]
\zero & \zero & \zero & \zero  & 2 & 3& \zero & \zero & 4 & 3     \\[0.2em]
\zero & \zero & \zero & \zero  & \zero & 2 & \zero &  \zero & 1 & 3 \\[0.2em]
\zero & \zero & \zero & \zero & \zero & \zero & 1 & \zero & 4  & \zero \\
};

\matTopBorder {3,2,1} 3 2 {N}
\node[font=\scriptsize,right] (\matName-1-0) [left=1.5em of
  \matName-1-1] {\bo{x \mbasis 2}}; 
\node[font=\scriptsize,right] (\matName-2-0) [left=1.5em of
  \matName-2-1] {\bo{x \mbasis 3}}; 
\node[font=\scriptsize,right] (\matName-3-0) [left=1.5em of
  \matName-3-1] {\bo{y \mbasis 2}}; 
\node[font=\scriptsize,right] (\matName-4-0) [left=1.5em of
  \matName-4-1] {\bo{y \mbasis 3}}; 
\node[font=\scriptsize,right] (\matName-5-0) [left=1.5em of
  \matName-5-1] {\bo{x \mbasis 1}}; 
\node[font=\scriptsize,right] (\matName-6-0) [left=1.5em of
  \matName-6-1] {\bo{y \mbasis 1}}; 
\node[font=\scriptsize,right] (\matName-7-0) [left=1.5em of
  \matName-7-1] {\bo{z \mbasis 2}}; 
\node[font=\scriptsize,right] (\matName-8-0) [left=1.5em of
  \matName-8-1] {\bo{z \mbasis 3}}; 
\node[font=\scriptsize,right] (\matName-9-0) [left=1.5em of
  \matName-9-1] {\bo{z \mbasis 1}}; 

\completelyRestricted 1 2 {10} {N}
\restricted 4 4 {10} {N}
\restricted 7 8 {10} {N}
\notRestricted 3 3 {10} {N}
\notRestricted 5 6 {10} {N}
\notRestricted 9 9 {10} {N}
\draw[draw=ggrey,shorten <=0.2em]
 (N-1-1.north west) -- (N-1-1.south west) -- (N-1-2.south west) --
 (N-3-2.south west) -- (N-3-3.south west) --
 (N-5-3.south west) -- (N-5-4.south west) --
 (N-6-4.south west) -- (N-6-5.south west) --
 (N-7-5.south west) -- (N-7-6.south west) --
 (N-8-6.south west) -- (N-8-7.south west) --
 (N-9-7.south west);

\fill[fill=vlgrey, fill opacity=0.3]
 (N-1-1.north west) -- (N-1-1.south west) -- (N-1-2.south west) --
 (N-3-2.south west) -- (N-3-3.south west) --
 (N-5-3.south west) -- (N-5-4.south west) --
 (N-6-4.south west) -- (N-6-5.south west) --
 (N-7-5.south west) -- (N-7-6.south west) --
 (N-8-6.south west) -- (N-8-7.south west) --
 (N-9-7.south west) -- (N-9-1.south west) -- (N-2-1.north west);

\draw[->,draw=vlgrey,shorten <=0.3em, shorten >=0.3em]
  (N-3-0) edge [out=170,in=190] (N-1-0);
\draw[->,draw=vlgrey,shorten <=0.3em, shorten >=0.3em,densely dashed]
  (N-8-0) edge [out=170,in=190] node[pos=0.5,sloped,text=vlgrey]
{\(\tiny\hspace*{-3mm}\not\;\not\)} (N-5-0);
\end{scope}
\end{tikzpicture}
\end{center}

Next we can go on with degree $3$. Generating $M_3$ we get all multiples
$xf_i$, $yf_i$ and $zf_i$ for $1\leq i \leq 3$ as it is shown in
Figure~\ref{fig:deg-3-mf5}. Looking at $M_3$ more closely we see some relation
to $M_2$. The three steps highlighted correspond to reduction
steps that have already occured in degree $2$:
\begin{align*}
x\mbasis 2 - x \mbasis 1 &= x (\mbasis 2 - \mbasis 1)\\
y\mbasis 2 - y \mbasis 1 &= y (\mbasis 2 - \mbasis 1)\\
z\mbasis 2 - z \mbasis 1 &= z (\mbasis 2 - \mbasis 1).
\end{align*}
Since we have done these reductions already it makes sense to not redo them
again, but use the information from $M_2$. We know that $f_5$ comes from $f_2$,
both share the same signature. So we just \emph{rewrite} $xf_2,yf_2,zf_2$ by
$xf_5,yf_5,zf_5$ in
$M_3$, respectively. The very same holds for $f_3$ and $f_4$.
Figure~\ref{fig:rewrite-deg-3-mf5} illustrates this process.

Only rewriting $f_2$ and $f_3$ with ``better'' elements lead to this matrix in
near row echelon form. Again, note that not all elements in the above picture
are allowed to reduce freely: The rows highlighted in green can reduce any other
row above them. So, for example, the row with signature $y \mbasis 2$ can reduce
the rows with signatures $x \mbasis 3$ and $x \mbasis 2$, respectively. In none of
these reductions the signature of any row changes. On the other hand, the row
labelled by signature $z \mbasis 3$ is not allowed to reduce the row labelled by
$x \mbasis 1$. Otherwise the signature might change. Nevertheless, this row is
still allowed to reduce the one labelled with $x \mbasis 3$. Thus we highlighted
this row in yellow to illustrate this restriction. The row labelled with $x
\mbasis 3$, and highlighted in red, is not allowed to reduce any other row.
$x \mbasis 3$ is the highest signature in degree $3$ w.r.t. $\potl$, thus any
reduction of another row would lead to a change in signatures.

Executing all not signature changing reduction steps we end up with a
\grobner{} basis up to degree $3$ represented by the row
echelon form

\begin{center}
\begin{tikzpicture}
\begin{scope}[shift={(0,0)}]
\matrix (N) [ matrix of math nodes,left delimiter={(},right delimiter={)},
              column sep=0.7em]{ 
2 & \zero & \zero & \zero & \zero & 4 & \zero  & 3     & 3 & 3 \\[0.2em]
\zero & 2 & \zero & \zero & \zero & \zero & \zero  & \zero & 3 & 3 \\[0.2em]
\zero & \zero & 1 & \zero & \zero & \zero &  \zero  & \zero & 3 & 4 \\[0.2em]
\zero & \zero & \zero & 1 & \zero & \zero & \zero  & \zero & 4 & \zero \\[0.2em]
\zero & \zero & \zero & \zero  & 2 & 3& \zero & \zero & 4 & 3     \\[0.2em]
\zero & \zero & \zero & \zero  & \zero & 2 & \zero &  \zero & 1 & 3 \\[0.2em]
\zero & \zero & \zero & \zero & \zero & \zero & 1 & \zero & 4  & \zero \\
\zero & \zero & \zero & \zero & \zero & \zero & \zero & 3  & \zero & 2\\
\zero & \zero & \zero & \zero & \zero & \zero & \zero & \zero & 3 & 2\\
};

\matTopBorder {3,2,1} 3 2 {N}
\node[font=\scriptsize,right] (\matName-1-0) [left=1.5em of
  \matName-1-1] {\bo{x \mbasis 2}}; 
\node[font=\scriptsize,right] (\matName-2-0) [left=1.5em of
  \matName-2-1] {\bo{y \mbasis 2}}; 
\node[font=\scriptsize,right] (\matName-3-0) [left=1.5em of
  \matName-3-1] {\bo{x \mbasis 1}}; 
\node[font=\scriptsize,right] (\matName-4-0) [left=1.5em of
  \matName-4-1] {\bo{y \mbasis 1}}; 
\node[font=\scriptsize,right] (\matName-5-0) [left=1.5em of
  \matName-5-1] {\bo{z \mbasis 2}}; 
\node[font=\scriptsize,right] (\matName-6-0) [left=1.5em of
  \matName-6-1] {\bo{z \mbasis 3}}; 
\node[font=\scriptsize,right] (\matName-7-0) [left=1.5em of
  \matName-7-1] {\bo{z \mbasis 1}}; 
\node[font=\scriptsize,right] (\matName-8-0) [left=1.5em of
  \matName-8-1] {\bo{x \mbasis 3}}; 
\node[font=\scriptsize,right] (\matName-9-0) [left=1.5em of
  \matName-9-1] {\bo{y \mbasis 3}}; 

\node[right] (mat-node) [left=3.0em of \matName] {$\matName_\d =$};

\draw[draw=ggrey]
 (N-1-1.south west) -- (N-1-2.south west) --
 (N-2-2.south west) -- (N-2-3.south west) --
 (N-3-3.south west) -- (N-3-4.south west) --
 (N-4-4.south west) -- (N-4-5.south west) --
 (N-5-5.south west) -- (N-5-6.south west) --
 (N-6-6.south west) -- (N-6-7.south west) --
 (N-7-7.south west) -- (N-7-8.south west) --
 (N-8-8.south west) -- (N-8-9.south west) --
 (N-9-9.south west);

\fill[fill=vlgrey, fill opacity=0.3]
 (N-1-1.south west) -- (N-1-2.south west) --
 (N-2-2.south west) -- (N-2-3.south west) --
 (N-3-3.south west) -- (N-3-4.south west) --
 (N-4-4.south west) -- (N-4-5.south west) --
 (N-5-5.south west) -- (N-5-6.south west) --
 (N-6-6.south west) -- (N-6-7.south west) --
 (N-7-7.south west) -- (N-7-8.south west) --
 (N-8-8.south west) -- (N-8-9.south west) --
 (N-9-9.south west) -- (N-9-1.south west) --
 (N-1-1.south west);

\end{scope}
\end{tikzpicture}.
\end{center}

Next we are computing a \grobner{} basis for $\ideal{f_1,f_2,f_3}$ up to degree
$4$. Again we generate the matrix $M_4$ building all combinations of monomials
of degree $2$ and $f_i$ for $1\leq i \leq 3$. This time we note that $M_4$
consists of $\binom 3 2 \cdot 3 = 18$ rows and $\binom 3 4 = 15$ columns. This means
that when we are reducing $M_4$ we might end up with rows that reduced to zero
(or rows that are multiples of others due to the restricted reduction process).
Nevertheless, these rows correspond to useless steps during a \grobner{} basis
algorithm. So how can we find out which to remove?

A polynomial reduction to zero corresponds to a syzygy in $\ring^3$. There are
principal (or trivial syzygies) we know already without any previous
computations: $f_1 \mbasis 2 - f_2 \mbasis 1$, $f_1 \mbasis 3 - f_3 \mbasis 1$ and 
$f_2 \mbasis 3 - f_3 \mbasis 2$.
Let us look at the signatures of these syzygies w.r.t. $\potl$:
\begin{center}
$
\setlength\arraycolsep{0.5em}
\begin{array}[]{ccccc}
\sig{f_1 \mbasis 2 - f_2 \mbasis 1} &=& \hd{f_1} \mbasis 2
&=& y^2 \mbasis 2 \\
\sig{f_1 \mbasis 3 - f_3 \mbasis 1} &=& \hd{f_1} \mbasis 3
&=& y^2 \mbasis 3 \\
\sig{f_2 \mbasis 3 - f_3 \mbasis 2} &=& \hd{f_2} \mbasis 3
&=& x^2 \mbasis 3.
\end{array}
$
\end{center}
We have seen in the degree $3$ case that we can rewrite elements with a given
signature $T$ by other elements that have the same signature. Thus for $T \in
\{y^2 \mbasis 2, y^2 \mbasis 3, x^2 \mbasis 3\}$ we can just use the above
syzygies resp. corresponding trivial relations in $\ring$. So the following $3$
elements are exchanged, respectively, in $M_4$
\begin{center}
$
\setlength\arraycolsep{0.5em}
\begin{array}[]{ccc}
y^2 f_2 &\longrightarrow& f_1 f_2 - f_2 f_1,\\
y^2 f_3 &\longrightarrow& f_1 f_3 - f_3 f_1,\\
x^2 f_3 &\longrightarrow& f_2 f_3 - f_3 f_2.
\end{array}
$
\end{center}
This means that we would add rows that have only zero entries for $y^2 \mbasis
2$, $y^2 \mbasis 3$ and $x^2 \mbasis 3$. Those rows do not play any role during
the reduction process of $M_4$, so we can remove them directly from the matrix.
In the end we receive a matrix $M_4$ of dimensions $15 \times 15$, thus we
know that when reducing $M_4$ to its row reduced echelon form $N_4$, all rows
are useful. Clearly, as we have done for $M_3$, we try to rewrite the
$15$ rows remaining in $M_4$ that correspond to elements $x^j y^k z^\ell f_i$
with $1\leq i \leq 3$ and $j+k+\ell = 2$ with elements from $N_2$ and $N_3$
in order to not repeat calculations already done at a lower degree. Computing
the row echelon form of $M_4$ we then receive a \grobner{} basis for
$\ideal{f_1,f_2,f_3}$ up to degree $4$.

Let us try to summarize the main ideas behind using signatures when computing
\grobner{} bases:
\begin{itemize}
\item Try to rewrite data and reuse already done calculations.
\item Keep track of this rewriting by not changing the signatures during
  the reduction process.
\item If the rewritten data is trivial resp. corresponds to a syzygy
  (relations that are already known) then discard this data.
\end{itemize}

\begin{remark}
Note that building Macaulay matrices as done in \mff{} is useful and efficient
only if the corresponding polynomial system is dense. Otherwise it makes more
sense to combine Buchberger's S-polynomial attempt with linear algebra. That
means, one first searches for all S-polynomials in a given degree $d$ and
\emph{all} needed reducers and generates a corresponding matrix afterwards. This
is the main idea behind Faug\`ere's \ffo{} algorithm (\cite{fF41999}).
In Section~\ref{sec:f4-f5} we present an efficient way of combining
signature-based criteria for discarding useless data with \ffo{}.
\end{remark}

With this in mind we are able to give a more theoretical introduction to
signature-based \grobner{} basis computations in the vein of Buchberger's
algorithm.

\section{\grobner{} bases with signatures}
\label{sec:signature-gb}

In this section we give an introduction to signature-based \grobner{} basis
algorithms from a mathematical point of view. Thus the content is dedicated to a
complete and correct description of the algorithms' underlying ideas.
Motivated by the specialized row echelon forms we presented in
Section~\ref{sec:matrix-f5} the notion of a polynomial reduction process
taking care of the signatures is introduced. Connections and differences to
classic polynomial \grobner{} basis theory are explained in detail.

Readers interested in the optimized variants only might skip most of this
section, but should at least consider notations introduced in
Section~\ref{sec:notations} and here as we agree on those throughout the paper.

\subsection{Signature reduction}
\label{sec:sreduction}
In order to keep track of the signatures when reducing corresponding polynomial
data we need to adjust Definition~\ref{def:reduce}. Sloppy speaking we get a
classic polynomial reduction together with a further condition.

\begin{definition}
\label{def:sreduce}
Let $\alpha\in \module$ and let $t$ be a term of~$\proj\alpha$. Then we
can \emph{\sreduce{}} $t$ by $\beta\in \module$ if
\begin{enumerate}
\item\label{cond:preduce} there exists a monomial $b$ such that $\hdp{b\beta}=t$ and
\item\label{cond:sreduce} $\sig{b\beta}\leq\sig\alpha$.
\end{enumerate}
The outcome of the \sreduction{} step is then $\alpha-b\beta$ and
$\beta$ is called the \emph{\sreducer{}}.  When $\beta$ \sreduces{} $t$ we also say for convenience
that $b\beta$ \sreduces{} $\alpha$. That way $b$ is introduced
implicitly instead of having to repeat the equation $\hdp{b\beta}=t$.
\end{definition}

\begin{remark}
Note that Condition~(\ref{cond:preduce}) from Definition~\ref{def:sreduce} defines
a classic polynomial reduction step (see~\ref{def:reduce}). It implies that $\hdp{b\beta} \leq \hdp{\alpha}$.
Moreover, Condition~(\ref{cond:sreduce}) lifts the above implication to $\module$ so that it
involves signatures. Since we are interested in computing \grobner{} Bases in
$\ring$ one
can interpret an \sreduction{} of $\alpha$ by $\beta$ as classic polynomial
reduction of $\proj\alpha$ by $\proj\beta$ together with
Condition~(\ref{cond:sreduce}). Thus an \sreduction{} represents a connection
between data in $\ring$ and corresponding data in $\module$ when a polynomial reduction
takes place.
\end{remark}

Just as for classic polynomial reduction, if
$\hdp{b\beta}\simeq\hdp\alpha$ then the \sreduction{} step is a
\emph{top \sreduction{} step} and otherwise it is a \emph{tail
\sreduction{} step}. Analogously we define the distinction for
signatures: If $\sig{b\beta}\simeq\sig\alpha$ then the reduction
step is a \emph{singular \sreduction{} step} and otherwise it is a
\emph{regular \sreduction{} step}.

The result of \sreduction{} of $\alpha\in \module$ is a $\gamma\in \module$
that has been calculated from $\alpha$ through a sequence of
\sreduction{} steps such that $\gamma$ cannot be further
\sreduced{}. The reduction is a \emph{tail \sreduction} if only tail
\sreduction{} steps are allowed and it is a \emph{top \sreduction{}}
if only top \sreduction{} steps are allowed. The reduction is a
\emph{regular \sreduction} if only regular \sreduction{} steps are
allowed. A module element $\alpha\in \module$ is \emph{\sreducible{}} if
it can be \sreduced{}.

If $\alpha$ \sreduces{} to $\gamma$ and
$\gamma$ is a syzygy then we say that $\alpha$ \emph{\sreduces{} to
zero} even if $\gamma\neq 0$.

\begin{example}
\label{ex:sreduction}
Assume an ideal $I=\langle f_1,f_2,f_3 \rangle \subset \ring = \Q[x,y,z,t]$ with
$f_1 = xyz - z^2t$, $f_2 = x^2y-y^3$ and $f_3 = y^3-zt^2-t^3$.
Furthermore, $<$ denotes the graded reverse lexicographical monomial order
which we extend to $\potl$ on the set of monomials of $\ring^3$.
Clearly, we have $\alpha_i = \mbasis i$ with $\proj{\alpha_i} = f_i$ for $i \in
\{1,2,3\}$. We start with $\basis = \{\alpha_1,\alpha_2,\alpha_3\}$.

Looking at $z\alpha_2$ we can regular top \sreduce{} $\hdp{z\alpha_2}$
with $x\alpha_1$ since $\hdp{x\alpha_1} = \hdp{z\alpha_2}$ and
$\sig{x\alpha_1} \potl \sig{z\alpha_2}$.
Call the resulting element $\alpha_4 = z\alpha_2 - x\alpha_1$. We can see that
we cannot further \sreduce{} $\proj{\alpha_4} = -y^3z+xz^2t$: The only
possible candidate is $\alpha_3$ but $\sig{z\alpha_3} = z\mbasis 3 \potg
z\mbasis 2 =
\sig{\alpha_4}$. Note that $\proj{\alpha_4} + z\proj{\alpha_3}$ would be a correct
classical polynomial reduction step, but it contradicts
Condition~(\ref{cond:sreduce}) of an \sreduction{}.
On the other hand, adding $\alpha_4$ to $\basis$ we are able to regular top
\sreduce{} $z\alpha_3$ w.r.t. $\basis$, namely by $\alpha_4$. We see that
whereas from a pure polynomial point of view reducing
$\proj{\alpha_4}+z\proj{\alpha_3}$ is the same as
$z\proj{\alpha_3}+\proj{\alpha_4}$ taking
the signatures into account destroys this equality. Only the second operation is
a valid \sreduction{}.

Again, we can regular top \sreduce{} $x\alpha_4$ with $y^2\alpha_1$. This gives
a new element $\alpha_5 = x\alpha_4 + y^2\alpha_1$ whereas $\proj{\alpha_5} =
x^2z^2t-y^2z^2t$.

Looking at $x^2\proj{\alpha_4} = -x^2y^3z+x^3z^2t$ one can use $\hdp{x\alpha_5}$
to tail \sreduce{}. Note that this \sreduction{} is singular due to
$\sig{x\alpha_5} = x^2z \mbasis 2 = \sig{x^2\alpha_4}$. In other words,
$x^2\alpha_4 - x\alpha_5 = (x^2z\mbasis 2 - x^3\mbasis 1) - (x^2z\mbasis 2 -
  x^3\mbasis 1 + xy^2\mbasis 1) = -xy^2\mbasis 1$.
Thus we see that $x^2\alpha_4$ \sreduces{} to a syzygy $\gamma = x^2\alpha_4 -
x\alpha_5 - x^2y\alpha_1$.
\end{example}

\begin{remark} \
\begin{enumerate}
\item The implied condition
$\hdp{b\beta}\leq\hdp\alpha$ is equivalent to
$\hdp{\alpha-b\beta}\leq \hdp\alpha$, so during \sreduction{} it is
not allowed to increase the lead term. For tail \sreduction{} we
perform only those \sreduction{} steps that do not change the lead
term at all. Analogously, the condition $\sig{b\beta}\leq\sig\alpha$
is equivalent to $\sig{\alpha-b\beta}\leq\sig{\alpha}$, so during
\sreduction{} it is not allowed to increase the signature. For regular
\sreduction{}, we perform only those \sreduction{} steps that do not
change the signature at all.
\item Note that by Lemma~15 in \cite{epSig2011} the notion of ``being singular top
\sreducible{}'' is equivalent to what is sometimes in the literature also called
``sig-redundant''.
\end{enumerate}
\end{remark}

Note that analogously to the classic polynomial reduction \sreduction{}
is always with respect to a finite \emph{basis} $\basis\subseteq
\module$. The \sreducers{} in \sreduction{} are chosen from the basis
$\basis$.

\subsection{Signature \grobner{} bases}
\label{sec:sgb}
Having defined a polynomial reduction process taking signatures into account we
are now able to define signature \grobner{} bases analogously to classic
polynomial \grobner{} bases.

\begin{definition}
\label{def:sgb}
Let $I= \langle \gen f \rangle$ be an ideal in $\ring$.
A finite subset \basis{} of $\module$ is a \emph{signature \grobner{} basis in signature $T$} for $I$ if
all $\alpha\in \module$ with $\sig\alpha=T$ \sreduce{} to zero w.r.t.
\basis{}.
\basis{}
is a \emph{signature \grobner{} basis up to signature $T$} for $I$ if \basis{}
is a signature \grobner{} basis in all signatures $S$ such that
$S<T$.
\basis{} is a \emph{signature \grobner{} basis} for $I$ if it is a
signature \grobner{} basis for $I$ in all signatures.
\end{definition}

\begin{lemma}
Let $I= \langle \gen f \rangle$ be an ideal in $\ring$.
If \basis{} is a signature \grobner{} basis for $I$ then
$\setBuilder{\proj\alpha}{\alpha\in\basis{}}$ is a \grobner{} basis for $I$.
\end{lemma}

\begin{proof}
For example, see Section~$2.2$ in \cite{rs-2012}.\qed
\end{proof}

\begin{convention}
In the following, when denoting $\basis{}\subseteq\module$ ``a signature
\grobner{} basis (up to signature $T$)'' we always mean ``a signature
\grobner{} basis (up to signature $T$) for $I=\langle \gen f \rangle$''. We omit
the explicit notion of the input ideal whenever it is clear from the context.
\end{convention}

As in the classic polynomial setting we want to give an algorithmic description
of signature \grobner{} bases. For this we introduce the notion of S-pairs,
similar to Definition~\ref{def:spoly}.

\begin{definition}\
\begin{enumerate}
\item Let $\alpha,\beta \in \module$ such that $\proj\alpha \neq 0$, $\proj\beta
\neq 0$ and let
$\lambda=\lcm\left(\hdp\alpha,\hdp\beta\right)$ be the monic least common
multiple of $\hdp\alpha$ and $\hdp\beta$. The \emph{S-pair} between $\alpha$ and
$\beta$ is given by
\[\spair\alpha\beta \defeq \frac{\lambda}{\hdp\alpha}\alpha -
\frac{\lambda}{\hdp\beta}\beta.\]
\item $\spair\alpha\beta$ is \emph{singular} if
$\sig{\frac{\lambda}{\hdp\alpha}\alpha} \simeq \sig{\frac{\lambda}{\hdp\beta}\beta}$.
Otherwise it is \emph{regular}.
\end{enumerate}
\end{definition}

Note that $\spair\alpha\beta \in \module$ and $\proj{\spair\alpha\beta} =
\spoly{\proj\alpha}{\proj\beta}$.

\begin{theorem}
\label{thm:spairs}
Let $T$ be a module monomial of $\module$ and let $\basis{}\subseteq \module$ be a
finite basis. Assume that all regular S-pairs $\spair\alpha\beta$ with
$\alpha,\beta\in\basis$ and $\sig{\spair\alpha\beta} < T$ \sreduce{} to zero and all
$\mbasis i$ with $\mbasis i<T$ \sreduce{} to zero. Then $\basis$ is a
signature \grobner{} basis up to signature $T$.
\end{theorem}

\begin{proof}
For example, see Theorem~2 in \cite{rs-ext-2012}.\qed
\end{proof}

Note the similarity of Theorem~\ref{thm:spairs} and Buchberger's Criterion for
Gr\"obner bases (Theorem~\ref{thm:gb}).

The outcome of classic polynomial reduction depends on the choice of
reducer, so the choice of reducer can change what the intermediate
bases are in the classic Buchberger algorithm. Lemma \ref{lem:p3}
implies that all S-pairs with the same signature yield the same
regular \sreduced{} result as long as we process S-pairs in order of
increasing signature.

\begin{lemma}
\label{lem:p3}
Let $\alpha,\beta\in \module$ and let \basis{} be a signature \grobner{}
basis up to signature $\sig\alpha=\sig\beta$. If $\alpha$ and $\beta$
are both regular top \sreduced{} then $\hdp\alpha=\hdp\beta$ or
$\proj\alpha=\proj\beta=0$. Moreover, if $\alpha$ and $\beta$ are both regular
\sreduced{} then $\proj\alpha=\proj\beta$.
\end{lemma}

\begin{proof}
For example, see Lemma~3 in \cite{rs-ext-2012}.\qed
\end{proof}

Let us simplify our notations a bit using facts from the previous statements.

\begin{notation}\
\label{conv:spairs}
\begin{enumerate}
\item Due to Lemma~\ref{lem:p3} we assume in the following that $\basis$ always
denotes a finite subset of $\module$ with the
property that for $\alpha, \beta \in
\basis$ with $\sig\alpha\simeq\sig\beta$ it follows that
$\alpha=\beta$.
\label{conv:spairs:basis}
\item Theorem~\ref{thm:spairs} suggests to consider only regular S-pairs for the
computation of signature \grobner{} bases. Thus in the following ``S-pair''
always refers to ``regular S-pair''.\label{conv:spairs:item}
\end{enumerate}
\end{notation}

\begin{definition}
A signature \grobner{} basis is \emph{minimal} if no basis element top
\sreduces{} any other basis element.
\end{definition}

Lemma~\ref{lem:min-sgb}
implies that the minimal signature \grobner{} basis for an ideal $I \subset
\ring$ is unique and is contained in all signature \grobner{} bases for $I$
up to sig-lead pairs.

\begin{lemma}
\label{lem:min-sgb}
Let $A$ be a minimal signature \grobner{} basis and let $B$ be a
signature \grobner{} basis for $\langle \gen f \rangle$.
Then it holds for all $\alpha\in A$ that
there exists a non-zero scalar $\const \in\field$ and a $\beta\in B$ such that
$\sig\alpha=\const\sig\beta$ and $\hdp\alpha=\const\hdp\beta$.
\end{lemma}

\begin{proof}
This is an easy corollary of Lemma~\ref{lem:p3}.\qed 
\end{proof}

\section{Generic signature \grobner{} basis computation}
\label{sec:gsb}
In the following we present a generic
signature-based Gr\"obner basis algorithm \gsb{} (Algorithm~\ref{alg:gsb}).
This algorithm works the same way as the classic Gr\"obner basis algorihm
presented by Buchberger in \cite{bGroebner1965}. The main difference is that
in \gsb{} the computations
are lifted from $\ring$ to $\module$ in the way presented in
sections~\ref{sec:sreduction} and~\ref{sec:sgb}.

\gsb{} should be understood as a generic description which does not aim on performance.
We see in Section~\ref{sec:rewrite-bases} how we can vary \gsb{} to receive a
template that can be used as a common basis from which all known efficient signature-based
\grobner{} basis algorithms can be derived from.

The classic Buchberger algorithm proceeds by reducing S-polynomials.
If an S-polynomial
reduces to a polynomial $h \in \ring, h \neq 0$ then $h$ is added to the
basis so that the S-polynomial now reduces to zero by this larger basis.
The classic Buchberger algorithm terminates once all S-polynomials between
elements of the basis reduce to zero.

\gsb{} does the very same with S-pairs using \sreductions{}. Based on
Theorem~\ref{thm:spairs}, once all S-pairs
\sreduced{} to zero w.r.t. \basis{}, \gsb{} terminates with a signature
\grobner{} basis.


\begin{algorithm}
\begin{algorithmic}[1]
\Require Ideal $I=\langle \gen f \rangle \subset \ring$, monomial
order $\leq$ on $\ring$ and a compatible extension on $\module$, total order
$\pleq$ on the pairset $\pairset$ of S-pairs
\Ensure Signature \grobner{} basis \basis{} for $I$,
\grobner{} basis $\hdsyz$ for $\syz{\gen f}$
\State $\basis{}\gets\emptyset$, $\hdsyz \gets \emptyset$
\State $\pairset\gets\set{\mbasis 1,\ldots,\mbasis m}$
\While{$\pairset\neq\emptyset$}
  \State $\beta\gets \min_\pleq  \pairset$\label{alg:gsb:pairchoice}
  \State $\pairset\gets \pairset\setminus\set \beta$
  \State $\gamma\gets$ result of regular \sreducing{}
  $\beta$\label{alg:gsb:reduction}
  \If {$\proj\gamma = 0$}
    \State $\hdsyz \gets \hdsyz \cup \set\gamma$\label{alg:gsb:addsyz}
  \ElsIf {$\gamma$ is not singular top reducible}\label{alg:gsb:singular}
    \State $\pairset\gets \pairset\cup\setBuilder
      {\spair\alpha\gamma}
      {\alpha\in\basis{}\text{ and $\spair\alpha\gamma$ is
                               regular}}$\label{alg:gsb:pairs}
  \State $\basis{}\gets\basis\cup\set\gamma$\label{alg:gsb:basis}
  \EndIf
\EndWhile
\State \textbf{return} $(\basis{},\hdsyz{})$
\end{algorithmic}
\caption{Generic signature-based \grobner{} basis algorithm \gsb.}
\label{alg:gsb}
\end{algorithm}

Thinking about correctness and termination of Algorithm~\ref{alg:gsb}
Line~\ref{alg:gsb:reduction} seems to be problematic: Only regular
\sreductions{} are done in \gsb{}. Moreover, if a reduction ends with an element
$\gamma$ that is singular top \sreducible{} w.r.t. \basis{}, $\gamma$ is not
even added to \basis{}. It turns out that singular top \sreductions{} are
useless for the computation of signature \grobner{} bases.

\begin{lemma}
\label{lem:singular}
Let $\alpha \in \module$ and let \basis{} be a signature \grobner{} basis up to
$\sig\alpha$. If $\alpha$ is singular top \sreducible{} w.r.t. \basis{} then $\gamma$
\sreduces{} to zero w.r.t. \basis{}.
\end{lemma}

\begin{proof}
If $\alpha$ is singular top \sreducible{} w.r.t. \basis{} then there exists
$\beta \in \basis{}$ and $b \in \ring$ such that $\sig\alpha = \sig{b\beta}$ and
$\hdp\alpha = \hdp{b\beta}$. If $\gamma$ denotes the result of the reduction of
$\alpha$ by $b\beta$ then $\sig\gamma < \sig\alpha$. Since \basis{} is a
signature \grobner{} basis up to $\sig\alpha$ $\gamma$ \sreduces{} to zero
w.r.t. \basis{}.\qed
\end{proof}

\begin{theorem}
\label{thm:term-cor-gsb}
Given $I=\langle \gen f\rangle \subset \ring$ and a monomial order
$\leq$ on $\ring$ with a compatible extension on $\module$ \gsb{} is an algorithm that
computes a signature \grobner{} basis \basis{} for $I$ and a module $\hdsyz$
generated by a \grobner{} basis for $\syz{\gen f}$.
\end{theorem}

\begin{proof}
Correctness of \gsb{} computing signature \grobner{} basis for $I$
is an easy generalization of Theorem~14 in \cite{epSig2011}. Allowing any
compatible module monomial order on $\module$ does not change the reasoning of
the corresponding proof there. On the other hand, using Lemma~\ref{lem:singular}
and the fact that \gsb{} computes \basis{} by increasing signatures it is an
easy exercise.
$\hdsyz$ being a \grobner{} basis for $\syz{\gen f}$ is clear by
Theorem~\ref{thm:syzygies}.

If $\pleq$ orders $\pairset$ by increasing signatures then termination of \gsb{}
follows by Theorem~20 in~\cite{erF5SB2013}. Otherwise it is possible that
\gsb{} adds several elements to \basis{} with the same signature: Assume
an intermediate state of \basis{} to consist of finitely many elements,
thus $\pairset$ is finite, too.
Next the S-pair $a\alpha - b\beta$ regular \sreduces{} to $\gamma$ w.r.t.
\basis{}.
\begin{enumerate}
\item $\gamma$ is top singular \sreducible{} and thus not added to
\basis{}.
\item There might exists $\delta \in \basis{}$ such that $\sig\gamma =
\sig\delta$ but $\hdp\gamma < \hdp\delta$. Note that $\hdp\gamma \geq 0$ so for
$\sig\gamma$ there are only finitely many elements in \basis{}.
\end{enumerate}
In the second situation there must be some element $\varepsilon$ added to
\basis{} inbetween $\gamma$ and $\delta$ such that such that
$\hdp{e\varepsilon} = \hdp\delta$ and $\sig{e\varepsilon} < \sig\delta$ for some
monomial $e$ in $\ring$.
We need to show that there cannot be infinitely many steps between
$\delta$ and $\gamma$. First of all only finitely many steps of
lower signature can be done due to our above discussion: There are only finitely
many elements in \basis{} per signature and there are only finitely many
signatures below $\sig\gamma$ handled since $\ring$ and $\module$ are
Noetherian. On the other hand, at the moment $\delta$ was added to
\basis{}, there were only finitely many S-pairs of signature $>\sig\delta$ in
$\pairset$. As in the situation above, in order to get a new element in a given
signature $T > \sig\delta$ new elements of signature $<T$ must be added to
\basis{}. Also for $T$ only finitely many sig-poly pairs
are possible. Moreover, between $\sig\delta$ and $T$ \gsb{} handles only
finitely many elements, again due to the Noetherianess of $\ring$ and $\module$.
All in all, between two elements of the same signature \gsb{} executes finitely
many steps. This completes the proof of \gsb{}'s termination.\qed
\end{proof}

The key to prove that \gsb{} computes a \grobner{} basis \hdsyz{} for the syzygy module
is Theorem~\ref{thm:syzygies} which implies that we can determine generators of the module
of syzygies from looking at those S-pairs and $\mbasis i$ that regular
\sreduce{} to zero.

\begin{theorem}
\label{thm:syzygies}
Let $\alpha\in \module$ be a syzygy and let \basis{} be a signature
\grobner{} basis up to signature $\sig\alpha$. Then there exists a
$\beta\in \module$ with $\sig\beta|\sig\alpha$ such that $\beta$ is an
S-pair or has the form $\mbasis i$ and such that $\beta$ regular
\sreduces{} to zero.
\end{theorem}

\begin{proof}
The proof is clear by Definition~\ref{def:sgb}. A variant of
Theorem~\ref{thm:syzygies} is Proposition~2.2
in \cite{gvwGVW2011}.\qed
\end{proof}

\begin{remark}
\label{rem:pairsetorder}
Note that in~\cite{sw2011b} Sun and Wang where the first to introduce a description of
signature-based \grobner{} basis algorithms where the order in which S-pairs are
handled does not matter. Clearly, the above description of \gsb{} covers this,
since we do not restrict $\pleq$. We refer the reader interested in a proof
of Theorem~\ref{thm:term-cor-gsb} with an emphasis on the pair set order $\pleq$
to Theorem~2.2 in \cite{sw2011b}.

Note that due to $\pleq$ \basis{} might not be a signature \grobner{} basis up
to signature $T$ when \gsb{} has just handled an S-pair in signature $T$. There
might be S-pairs of signature $<T$ which are still in $\pairset$. Nevertheless,
once \gsb{} terminates \basis{} is a signature \grobner{} basis for $I$ and thus
a signature \grobner{} basis up to signature $T$ for all $T$.
\end{remark}

For the sake of efficiency one might choose $\pleq$ to order $\pairset$ by
increasing signatures of the S-pairs. As we see in
Section~\ref{sec:spair-elimination}
such an order respects the criteria to remove useless S-pairs best.

\begin{definition}
\label{def:increasing-pairset-order}
$\pleqs$ denotes the order $\pleq$ in a signature-based
\grobner{} basis algorithm which sorts $\pairset$ by increasing signature.
\end{definition}

\begin{lemma}
Let \gsb{} with $\pleqs$ pick the next S-pair $\alpha$ to be regular
\sreduced{} such that $\sig\alpha = T$. Then \basis{} is a signature
\grobner{} basis up to signature $T$.
\label{lem:gsb-with-pleqs}
\end{lemma}

\begin{proof}
Since \gsb{} with $\pleqs$ handles S-pairs by increasing signature this is clear
by Definition~\ref{def:sgb}.\qed
\end{proof}

\begin{corollary}
\label{cor:gsb-minsigbasis}
\gsb{} with $\pleqs$ computes a minimal signature \grobner{} basis for the corresponding input.
\end{corollary}

\begin{proof}
A new element $\gamma$ with $\proj\gamma \neq 0$ is added to $\basis{}$ only if
$\gamma$ is not singular top \sreducible{} w.r.t. $\basis{}$. The minimality
then follows by Lemma~\ref{lem:gsb-with-pleqs}.\qed
\end{proof}

\begin{characteristics} \
\label{char:gsbend}
\begin{enumerate}
\item Note that Corollary~\ref{cor:gsb-minsigbasis} does not hold for arbitrary pair
set orders $\pleq$: Assume S-pair $\alpha$ being regular \sreduced{} by \gsb{} to
$\gamma$ and $\proj\gamma \neq 0$. W.l.o.g. we can assume that $\gamma$ is not
singular top \sreducible{}.\footnote{Otherwise there exists
$\delta\in\basis{}$ with $\sig\delta = c \sig\gamma$ and $\hdp\delta = c
\hdp\gamma$, thus we can just replace $\gamma$ by $\delta$ in the above
situation.} If, later on, \gsb{} regular \sreduces{} an S-pair from $\pairset$
to $\beta$ with $\sig\beta < \sig\gamma$ such that $\hdp\beta \mid \hdp\gamma$
then a new S-pair $\varepsilon = b\beta - \gamma$ is handled. Hereby
$\sig\varepsilon = \sig\gamma$ but $\hdp\varepsilon < \hdp\gamma$. 
So \basis{} can have several elements in the same signature $T$. 

Still, Lemma~\ref{lem:p3} is valid and makes sense: Once all S-pairs of
signature $T$ are handled by \gsb{} (in any given order $\pleq$)
  \basis{} is a signature \grobner{} basis up to signature $T$ and $\gamma$
can be further regular \sreduced{}, namely to $\varepsilon$.
Thus, in Section~\ref{sec:signature-gb} and in the following we often consider
signature \grobner{} bases
up to some signature $T$. Due to our considerations here, the second part of
Remark~\ref{rem:pairsetorder} and Lemma~\ref{lem:gsb-with-pleqs} this makes sense.
\label{char:gsbend:pairsetorders}
\item\label{char:gsbend:ssp} As one can easily see, Algorithm~\ref{alg:gsb} does only rely on data provided
by $\proj\alpha$ and $\sig\alpha$, but it does not need to store $\alpha$
completely. Thus instead of using $\alpha$ one can optimize an implementation
of \gsb{} by using $\spp\alpha$.

Moreover, if one is only interested in a \grobner{} basis for $\gen f$, \gsb{} can 
be optimized in the sense that one can restrict $\hdsyz$ to store only the initial module of the
corresponding syzygy module: Using only sig-poly pairs in Algorithm~\ref{alg:gsb} 
we are no longer able to store the full module element $\gamma$ in $\hdsyz$ at
Line~\ref{alg:gsb:addsyz}. Still one can compute at the same time the initial 
submodule $\hdsyz$ of the syzygy module of $\gen f$. In order to do so, one
needs to exchange Line~\ref{alg:gsb:addsyz} with

\begin{center}
\begin{algorithmic}[1]
\setcounter{ALG@line}{7}
   \State $\hdsyz\gets \hdsyz \cup \set{\sig\gamma}$
\end{algorithmic}
\end{center}

The fact that one can use signature-based \grobner{}
basis algorithms to compute the initial module of the module of corresponding syzygies
was first mentioned in \cite{ggvGGV2010}.
\end{enumerate}
\end{characteristics}

Signature-based \grobner{} basis algorithms like \gsb{} are in the vein of
a bigger class of algorithms computing the image and the kernel of a module
homomorphism at the same time: In our setting the image is the signature
\grobner{} basis \basis{} and the kernel is the syzygy module \hdsyz{}. Other
well-known, \grobner{} basis related algorithms of this type are, for example, the MMM
algorithm by Marinari, M\"oller and Mora (\cite{mmm93}) and the FGLM algorithm by Faug\`ere,
Gianni, Lazard and Mora (\cite{fglmFGLM1993}). Recently, Sun gave a nice
overview on the connections between those algorithms in \cite{sun2013}.

\section{S-Pair Elimination}
\label{sec:spair-elimination}

Until now we have introduced signature \grobner{} bases and their computation
only to receive a \grobner{} basis for some ideal $\langle \gen f \rangle$ and the
initial module of $\syz{\gen f}$. As mentioned in Section~\ref{sec:gsb} \gsb{}
should be understood as a template and common basis for all signature-based
\grobner{} basis algorithms. Thus, it is slow and not at all optimized.
One main bottleneck of \gsb{} is the high number
of \sreduction{}s to zero. As for the classic Buchberger algorithm (see
\cite{bGroebner1965,buchberger2ndCriterion1985}) we are searching
for criteria to discard such useless computations in advance like we have used
known syzygies in \mff{} in Section~\ref{sec:matrix-f5}.

Assume that \gsb{} regular \sreduces{} an S-pair in signature $T$ to $\gamma \in
\module$. Then three different situations can appear:
\begin{enumerate}
\item If $\gamma$ is a syzygy then $\gamma$ is added to $\hdsyz$ in
Line~\ref{alg:gsb:addsyz}.
\item If $\gamma$ is not syzygy but singular top \sreducible{} then by
  Lemma~\ref{lem:singular} $\gamma$ will \sreduce{} to zero. Thus it is
discarded in Line~\ref{alg:gsb:singular}.
\item Otherwise $\gamma$ is used to build new S-pairs with elements in
\basis{}
(Line~\ref{alg:gsb:pairs}) and later on itself added to \basis{}
(Line~\ref{alg:gsb:basis}).
\end{enumerate}

\begin{definition}
For the above three cases $T$ is respectively a \emph{syzygy},
\emph{singular} or \emph{basis} \emph{signature}.
\end{definition}

We are interested in the situations where elements are discarded. In the
following we take a closer look at syzygy and singular signatures.

\subsection{Eliminating S-pairs by known syzygies}
\label{sec:syzygy-criterion}
Clearly, we receive syzygies by \sreductions{} to zero in \gsb{}, but there are
also syzygies immediately known without precomputations as we have already seen
in the example computation of \mff{} in Section~\ref{sec:matrix-f5}.

\begin{definition}
\label{def:koszul}
The \emph{Koszul syzygy} between $\alpha,\beta\in\basis{}$ is
$\koz\alpha\beta\defeq \proj\beta\alpha-\proj\alpha\beta$. If
$\sig{\proj\beta\alpha}\not\simeq\sig{\proj\alpha\beta}$ then the
Koszul syzygy is \emph{regular}. By ``Koszul syzygy'' we always mean
``regular Koszul syzygy''.
\end{definition}

Trivial relations resp. principal syzygies are Koszul syzygies.
Using those and already computed zero reductions we are able to
flag a given signature being predictably syzygy.

\begin{definition}
A signature $T$ is \emph{predictably syzygy} if there exists a syzygy
$\sigma\in\module$ such that $\sig\sigma<T$ and $\sig\sigma|T$.
\end{definition}

Being predictably syzygy gives us a nice characterization when computing
\grobner{} bases.

\begin{lemma}[Syzygy criterion]
\label{lem:syzygy-criterion}
Let $\alpha, \beta \in \basis{}$, $\gamma = \spair\alpha\beta$ with $\sig\gamma$
being predictably syzygy, and let \basis{} be a signature \grobner{} basis up to
$\sig\gamma$. Then $\gamma$ \sreduces{} to zero w.r.t. \basis{}. 
Moreover, if $S$ is a syzygy signature and $S|T$ then $T$ is also a syzygy
signature.
\end{lemma}

\begin{proof}
If $\gamma$ is predictably syzygy then there exists a syzygy $\sigma \in
\module$ such that $\sig\sigma = \sig\gamma$. $\proj{\gamma -\sigma} =
\proj\gamma$ but $\sig{\gamma - \sigma} < \sig\gamma$. By
Definition~\ref{def:sgb} $\gamma - \sigma$ \sreduces{} to zero w.r.t.
\basis{}, thus also $\gamma$ behaves in this way.\qed
\end{proof}

The outcome of Lemma~\ref{lem:syzygy-criterion} is that whenever we
handle an S-pair $\gamma$ in a signature-based \grobner{} basis algorithm like
\gsb{} whose signature is divisible by the signature of a syzygy we can discard
$\gamma$.

\begin{remark}
Restricting Lemma~\ref{lem:syzygy-criterion} to principal syzygies and the
compatible module monomial order used to $\potl$ we get a
statement equivalent to the \ff{} criterion presented in Theorem~1
in~\cite{fF52002Corrected}.
\end{remark}
\subsection{Uniqueness of S-pairs at a given signature}
\label{sec:singular-criterion}
Next we are looking at the situation where the \sreduction{} of an S-pair ends
with a non-syzygy element $\gamma$ that is singular top \sreducible{} w.r.t.
\basis{}. We have already seen in Lemma~\ref{lem:singular} that we can discard
such S-pairs in the computations. The remaining question is how to detect such a
situation.

Being singular top \sreducible{} is a special case of the situation where there
are two or more S-pairs in the same signature $T$.
If so, we only have to regular \sreduce{} one of them as they all regular
\sreduce{} to the same thing by Lemma \ref{lem:p3}. Since
\sreduction{} proceeds by decreasing the lead term, we can for example try
to speed up the process by choosing an S-pair $\gamma$ in signature $T$
whose lead term $\hdp\gamma$ is minimal.
If $\sig{\spair\alpha\beta}=\sig{a\alpha}$, then
we get the same result from regular \sreducing{} $\spair\alpha\beta$
as for regular \sreducing{} $a\alpha$ by
Notation~\ref{conv:spairs}~(\ref{conv:spairs:item}). 

All in all we get the following nice description of the \emph{singular
  criterion}:

\begin{lemma}[Singular criterion]
For any signature $T$ we need to handle exactly one $a\alpha \in \module$ from
\begin{equation}
\rewriters T = \setBuilder{a\alpha}{ \alpha\in\basis\text{, } a\text{ is a monomial
    and } \sig{a\alpha}= T }
\label{eq:rewriterset}
\end{equation}
computing a signature \grobner{} basis.
\label{lem:singular-criterion}
\end{lemma}

\begin{remark}\
\label{rem:singularcriterion}
\begin{enumerate}
\item Note that $\alpha$ might not be involved in any S-pair in signature $T$. In this
situation at signature $T$ no S-pair is computed resp. \sreduced{} at all.
\item Note that when computing signature \grobner{} bases by signature-based
algorithms with an arbitrary pair set order $\pleq$ uniqueness of the elements
in signature $T$ is not guaranteed. A situation as pointed out in
\characteristic~\ref{char:gsbend}~(\ref{char:gsbend:pairsetorders}) might appear and thus
after having already chosen and regular \sreduced{} an element from $\rewriters
T$ the algorithm might come back to signature $T$ and makes a new choice from
$\rewriters T$.
\item Lemma~\ref{lem:singular-criterion} corresponds to rewriting rows in
\mff{} as done in Section~\ref{sec:matrix-f5}. Choosing an element in signature
$T$ mirrors searching already reduced row echelon forms $N_d$ for better
representations of the row labelled by $T$.
\end{enumerate}
\end{remark}

What is now left to do is to make a good choice for $a\alpha$ from $\rewriters
T$. For this we need to introduce the notion of a rewriter in the following.

\section{Rewrite bases}
\label{sec:rewrite-bases}
In Section~\ref{sec:singular-criterion} we have seen that per signature $T$ we
only need to take care of one element. In order to make a choice of such an
element we need to define an order on $\rewriters T$. For this the notion of 
so-called \emph{rewriters} is introduced in the
following. In this section we present a first signature-based \grobner{} basis
algorithm using S-pair elimination as presented in
Section~\ref{sec:spair-elimination}.
This is then the fundamental algorithm we can derive all known, efficient
implementations from.

Similar attempts to achieve such a comprehensive representation of
signature-based \grobner{} basis algorithms are given, for example,
in~\cite{huang2010,sw2011b}. The algorithms presented there, called
\trb{} and \gbgc{} are included in Algorithm~\ref{alg:rba}, called
\rba{}. Note that in~\cite{erF5SB2013} there is already an algorithm called
\rba{}, here we generalized it further.

\subsection{Combining elimination criteria}
\label{sec:combined-criteria}
Before we introduce the concept of rewriter, let us shortly recall the syzygy
criterion: An element $\gamma$ is
discarded if there exists a syzygy $\sigma$ such that $\sig\sigma \mid
\sig\gamma$, or in other words, there exists a monomial $s \in \ring$ such that
$\sig{s\sigma} = \sig\gamma$. Thus we have again two elements of the
same signature and need to decide which one to handle. Of course, by
Remark~\ref{rem:singularcriterion} we take $s\sigma$ since we know that
$\proj{s\sigma} = 0$ already, so no further computations need to be done in
signature $\sig{s\sigma}$. But this is nothing else but a rewording of
Lemma~\ref{lem:syzygy-criterion}, the syzygy criterion. It follows that we can
generalize the set $\rewriters T$ to
\begin{equation}
\rewriters T = \setBuilder{a\alpha}{ \alpha\in\basis\cup\hdsyz\text{, } a\text{ is a monomial
    and } \sig{a\alpha}= T }
\label{eq:generalizedrewriterset}
\end{equation}

The only difference between Equation~\ref{eq:rewriterset}
and~\ref{eq:generalizedrewriterset} is that $\alpha$ is now allowed to be in
$\hdsyz$, too.
With this the two criteria from Section~\ref{sec:syzygy-criterion}
and~\ref{sec:singular-criterion} to find useless S-pairs unite to one single criterion.
Furthermore, with this only one question remains to be answered when
implementing signature-based \grobner{} basis algorithms: How to choose the
single element from $\rewriters T$?

Since we have seen that all elements from $\rewriters T$ ``rewrite'' the same
information for the input ideal $I = \ideal{\gen f}$ at signature $T$ the following
naming conventions are reasonable.

\begin{definition} \
\label{def:rewriter}
\begin{enumerate}
\item\label{def:rewriter:order} A \emph{rewrite order $\rleq$} is a partial order on
$\basis{}\cup\hdsyz$ such that $\rleq$ is a total order on \basis{}

\item An element $\alpha \in \basis$ is a \emph{rewriter in signature $T$} if
$\sig\alpha \mid T$. If for a monomial $a\in\ring$ $\sig{a\alpha} = T$ we also
say for convenience that $a\alpha$ is a rewriter in signature $T$.
The $\preceq$-maximal rewriter in signature $T$ is the \emph{canonical rewriter
  in signature $T$}.
A multiple $a\alpha$ of a basis element $\alpha$ is \emph{rewritable} if
$\alpha$ is not the canonical rewriter in signature $\sig{a\alpha}$.
\end{enumerate}
\end{definition}

\begin{remark}
Of course, the definition of a rewrite order in Definition~\ref{def:rewriter} is
rather generic and not practical. For example, it does not even take care of
the elements in \hdsyz{}. Clearly, for optimized computations one want $s\sigma$
be the canonical rewriter in signature $\sig{s\sigma}$ for
$\sigma\in\hdsyz$. Still, in terms of correctness, one do not need to restrict
Definition~\ref{def:rewriter}~(\ref{def:rewriter:order}) to this. We see in the
following how explicitly defined  rewrite orders can be used to reach efficient
implementations of signature-based criteria to discard useless S-pairs.
\end{remark}

\begin{example}
Looking again at Example~\ref{ex:sreduction} we see that $\sig{x\alpha_5} =
\sig{x^2\alpha_4} = x^2z\mbasis 2$. Defining a rewrite order $\rleq$ by
$\alpha \rleq \beta$ if $\sig\alpha \leq \sig\beta$ we can see
that $x^2\alpha_4$ is rewritable since $\alpha_5$ is the
canonical rewriter in signature $x^2z\mbasis 2$ due to $\sig{\alpha_4} = z
\mbasis 2 < xz \mbasis 2 = \sig{\alpha_5}$.
\end{example}

Definition~\ref{def:rewriter} gives us a choice for $\rewriters T$, namely we
can choose the canonical rewriter in signature $T$ from $\rewriters T$. Of
course, using Equation~\ref{eq:generalizedrewriterset} to find the canonical
rewriter w.r.t. $\rleq$ instead of using the syzygy criterion and the
rewritable criterion independently from each other we need to explain the
following: If a syzygy exists for signature $T$, then all S-pairs in signature
$T$ are removed. It turns out that in the general description of rewrite bases
we are giving here this need not be true at all. Of course it makes sense to
define $\alpha \rleq \beta$ whenever $\beta \in \hdsyz$. We come back to this
fact once we are explicitly defining rewrite orders in
Section~\ref{sec:best-rewrite-order}.

Analogously to Section~3.2 in~\cite{erF5SB2013} we introduce next
the important notion of a rewrite basis. Note that the combination of the
syzygy and the singular criterion lead to a much easier notation.
We see in the following a strong connection to signature \grobner{} bases.

\begin{definition}
\basis{} is a \emph{rewrite basis in signature $T$} if the canonical rewriter
in $T$ is not regular top \sreducible{}.
\basis{} is a \emph{rewrite basis up to signature $T$} if \basis{} is a rewrite
basis in all signatures $S<T$. \basis{} is a \emph{rewrite basis} if
\basis{} is a rewriter basis in all signatures.
\end{definition}

\begin{lemma}
If \basis{} is a rewrite basis up to signature $T$ then \basis{} is also a
signature \grobner{} basis up to $T$.
\label{lem:rewbasis-sigbasis}
\end{lemma}

\begin{proof}
The special case where a rewriter order is a total order on $\basis$
fulfilling $\sig\alpha \mid \sig\beta \Longrightarrow \alpha \rleq \beta$ is
presented in Lemma~8 in~\cite{erF5SB2013}. Generalizing this proof to our
setting is trivial.\qed
\end{proof}

\subsection{An algorithm computing rewrite bases}
\label{sec:rewrite-basis-algorithm}
Next we present an algorithm quite similar to Algorithm~\ref{alg:gsb} that
implements the above mentioned S-pair elimination in the sense that it computes
a rewrite basis. We show that depending on the chosen rewrite order the size of
the rewrite basis varies. 


\begin{algorithm}
\begin{algorithmic}[1]
\Require Ideal $I=\langle \gen f \rangle \subset \ring$, monomial
order $\leq$ on $\ring$ and a compatible extension on $\module$, total order
$\pleq$ on the pairset $\pairset$ of S-pairs, a rewrite order
$\rleq$ on $\basis \cup \hdsyz$
\Ensure Rewrite basis \basis{} for $I$, \grobner{} basis $\hdsyz$ for $\syz{\gen f}$
\State $\basis{}\gets\emptyset$, $\hdsyz \gets \emptyset$
\State $\pairset\gets\set{\mbasis 1,\ldots,\mbasis m}$
\State $\hdsyz\gets \set{f_i \mbasis j - f_j \mbasis i \mid 1\leq i < j \leq
  m}\subseteq \module$\label{alg:rba:initialsyz}
\While{$\pairset\neq\emptyset$}
  \State $\beta\gets \min_\pleq  \pairset$\label{alg:rba:choosespair}
  \State $\pairset\gets \pairset\setminus\set \beta$
  \If {$\notrewritable\beta$} \label{alg:rba:rewritecheck}
    \State $\gamma\gets$ result of regular \sreducing{}
    $\beta$\label{alg:rba:regular}
    \If {$\proj{\gamma}=0$}
      \State $\hdsyz\gets \hdsyz+\set{\gamma}$\label{alg:rba:addsyz}
    \Else \label{alg:rba:addsingularelement}
      \State $\pairset\gets \pairset\cup\setBuilder
        {\spair\alpha\gamma}
        {\alpha\in\basis{}\text{ and $\spair\alpha\gamma$ is
                                regular}}$\label{alg:rba:pairs}
      \State $\basis{}\gets\basis\cup\set\gamma$\label{alg:rba:basis}
    \EndIf
  \EndIf
\EndWhile
\State \textbf{return} $(\basis{}, \hdsyz)$
\end{algorithmic}
\caption{Rewrite basis algorithm \rba.}
\label{alg:rba}
\end{algorithm}

\begin{algorithm}
\begin{algorithmic}[1]
\Require S-pair $a\alpha - b\beta \in\module$, finite subset $\basis\cup\hdsyz \in
\module$, rewrite order $\rleq$ on $\basis\cup\hdsyz$
\Ensure ``True'' if S-pair is rewritable; else ``false''
\If {$a\alpha$ is rewritable}
  \State \textbf{return} true
\EndIf
\If {$b\beta$ is rewritable}
  \State \textbf{return} true
\EndIf
\State \textbf{return} false
\end{algorithmic}
\caption{Rewritability check \rewritablenoarg{} for \rba{}.}
\label{alg:rewritable}
\end{algorithm}

Algorithm~\ref{alg:rba} differs from \gsb{} in three points:
\begin{enumerate}
\item In Line~\ref{alg:rba:initialsyz} \rba{} directly adds the known Koszul
syzygies to \hdsyz{}. This increases the number of possible canonical rewriters
in $\rewriters T$ in a given signature $T$.
\item In Line~\ref{alg:rba:rewritecheck} \rba{} uses Algorithm~\ref{alg:rewritable}
to check if the S-pair $\beta$ is rewritable or not. If so, \rba{} discards
$\beta$ and chooses the next S-pair in $\pairset$. \gsb{} does not provide any
such check.
\item In Lines~\ref{alg:rba:pairs} and~\ref{alg:rba:basis} \rba{} takes the
currently regular \sreduced{} $\gamma$, generates new regular S-pairs with
it and adds $\gamma$ to \basis{}. Whereas \gsb{} handles only not singular top
\sreducible{} $\gamma$, \rba{} runs these steps on all non-syzygy $\gamma$.
\end{enumerate}

Whereas the first two points are optimizations compared to \gsb{}, the
third change seems to be absurd. We have already seen that singular top
\sreducible{} elements are not needed for \basis{}, so why adding them? The
reason is that \rba{} computes rewrite bases, and in order to fulfill the
definition it has to add all these elements to \basis{} nevertheless they are
singular top \sreducible{} or not. Since \rba{} depends on the chosen rewrite
order $\rleq$ we need to store all elements, since they could lead to new
canonical rewriters. We see in Section~\ref{sec:best-rewrite-order} how different
rewrite orders can affect \rba{} quite a lot.

Analogously to Theorem~\ref{thm:term-cor-gsb} we receive the following
statement.

\begin{theorem}
\label{thm:term-cor-rba}
Given $I=\langle \gen f\rangle \subset \ring$, a monomial order
$\leq$ on $\ring$ with a compatible extension on $\module$, $\pleqs$ on
$\pairset$ and a rewrite order
$\rleq$ \rba{} is an algorithm that computes a rewrite basis \basis{} for $I$
and a module $\hdsyz$ generated by a \grobner{} basis
for $\syz{\gen f}$.
\end{theorem}

\begin{proof}
See~\cite{erF5SB2013}: Theorem~7 for correctness and Theorem~20 for termination.\qed
\end{proof}

\begin{characteristics}\
\label{char:rba}
\begin{enumerate}
\item In~\cite{erF5SB2013} algorithm \rba{} is presented for the first time.
Here \rba{} is presented more general in the sense that different pair set orders
are allowed. Moreover, generalizing the idea of rewritability to include the
syzygy criterion is new in the current presentation.
\item\label{char:rba:potl}
If $\potl$ is used then \rba{} computes $\basis$ and $\hdsyz$ incremental
by increasing indices. Thus it makes sense to optimize Algorithm~\ref{alg:rba}
to recompute $\hdsyz$ once the computations in a new index $k$ starts: At this
point we have a \grobner{} basis $\proj\basis=\{\proj{\alpha_1},\ldots,
\proj{\alpha_{k-1}}\} \subset \ring$ for
$\langle f_1,\ldots,f_{i-1}\rangle$. Defining $\alpha_k=\mbasis k$ such that
$\proj{\alpha_k} := f_i$ we can add for $j<k$ $\proj{\alpha_j} \alpha_k
- \proj{\alpha_k} \alpha_j$ to $\hdsyz$.
\item Note that in spite of Theorem~\ref{thm:term-cor-gsb} we have to limit
Theorem~\ref{thm:term-cor-rba} for \rba{}: Whereas one can show that
\gsb{} terminates for any chosen pair set order $\pleq$ we restrict
\rba{} to $\pleqs$. The problem is the interplay between $\pleq$ and $\rleq$: It
is possible to choose both in a way such that \rba{} adds the same sig-poly pair
to \basis{}. This is possible due to the fact that \rba{} does not check for
singular top \sreducibility{} when adding new elements to \basis{} (since this
shall be handled by the more general and flexible rewritability criterion and thus
$\rleq$).

By the ideas of \cite{sw2011b} it is noted in \cite{gvwGVW2011} that
GVW can compute \grobner{} bases by handling S-pairs in any given order.
This coincides with our descriptions of \gsb{} and \rba{}. Moreover, we show
that not only \gvw{} can do so, but all known efficient implementations of
\rba{}, for example, also \ff{}.
\item Note that there is a strong connection between the signature and the
so-called \emph{sugar degree}. It is shown in~\cite{ederInhomog2013} that using
$\pleqs$ combined with a degree compatible monomial order $<$ a signature-based
\grobner{} basis algorithm refines the sugar degree order of critical pairs.
\item Since all known signature-based \grobner{} basis algorithms are special
cases of \rba{} their correctness and termination is clear with
Theorem~\ref{thm:term-cor-rba}. Later on, we discuss the topic of termination
further, especially for \ff{} in Section~\ref{sec:f5-termination}. There we do
not give full proofs, but refer the reader interested in more details on proving
termination to the corresponding papers. A small selection might be already
mentioned here:
\begin{itemize}
\item \cite{galkinTermination2012,panhuwang2012,erF5SB2013,panhuwang2013} for
\ff{} and variants.
\item
\cite{apF5CritRevised2011,gvwGVW2011,panhuwang2012,rs-2012,erF5SB2013,panhuwang2013}
for \gvw{}, \sba{} and variants.
\end{itemize}
\end{enumerate}
\end{characteristics}

Note that due to Lemma~\ref{lem:gsb-with-pleqs},
Corollary~\ref{cor:gsb-minsigbasis} as well as the definition of rewritability
in~\ref{def:rewriter} choosing $\pleqs$ is the best possible choice
for an efficient computation of \basis{} and \hdsyz{}. Thus we restrict
ourselves in Theorem~\ref{thm:term-cor-rba} to this situation.

Moreover, let us agree in the remaining of the paper on the following:

\begin{convention}
If not otherwise stated we assume $\pleq\; =\; \pleqs$.
\end{convention}

If \rba{} can make use of the rewritability checks, is the resulting rewrite
basis, and thus signature \grobner{} basis smaller?

\begin{lemma}
Given $I=\langle \gen f\rangle \subset \ring$ and a monomial order
$\leq$ on $\ring$ with a compatible extension on $\module$ the basis computed by
\gsb{} is always a subset of the one computed by \rba{} up to sig-poly pairs.
\label{lem:gsb-smaller-than-rba}
\end{lemma}

\begin{proof}
Due to $\pleqs$ this follows directly from Corollary~\ref{cor:gsb-minsigbasis}.\qed
\end{proof}

The optimization we achieve when switching from \gsb{} to \rba{} lies in the
fact that \gsb{} regular \sreduces{} many more elements to zero w.r.t.
\basis{}, whereas \rba{} can detect, and thus discard, such an
\sreduction{} in advance.

The following two lemmata are of importance when we compare different rewrite
rules and specific implementations of \rba{}.

\begin{lemma}[Slight variant of Lemma~11 in~\cite{erF5SB2013}]
Let $\alpha \in \module$, let \basis{} be a rewrite basis up to signature
$\sig\alpha$ and let $t$ be a regular \sreducible{} term of~$\proj\alpha$. Then
there exists a regular \sreducer{} $b\beta$ which is
\begin{itemize}
\item not regular top \sreducible{},
\item not rewritable and
\item not syzygy.
\end{itemize}
\label{lem:existence-regular-sreducer}
\end{lemma}
\begin{proof}
Let $M_t$ be the set of all regular \sreducers{} of $t$. Let $c\gamma\in M_t$ of
minimal possible signature $T$, and let $b\beta$ be the canonical rewriter in
signature $T$. By definition, $b\beta$ is not rewritable. Since $\sig{c\gamma}
< \sig\alpha$ $b\beta$ is not regular top \sreducible{}.

Moreover, there cannot exist a $d\delta \in M_t$ such that $d\delta$ regular top
\sreduces{} $c\gamma$ as otherwise $\sig{d\delta} < T$. By Lemma~\ref{lem:p3}
$\hdp{b\beta} = \hdp{c\gamma}$ and thus $b\beta\in M_t$. 

If there exists $\sigma \in \ideal\hdsyz$ such that $\sig\sigma = T$ then
$b\beta - \sigma \in M_t$ since $b\beta \in M_t$, but $\sig{b\beta-\sigma} < T$.
This is a contradiction.\qed
\end{proof}

\begin{lemma}
Let \basis{} be a rewrite basis up to signature $T$, and let $a\alpha$ be the
  canonical rewriter in signature $T$. Then \rba{} \sreduces{} an S-pair in
  signature $T$ if and only if $a\alpha$ is regular top \sreducible{} and $T$ is
  not predictably syzygy.
\label{lem:when-rb-reduces}
\end{lemma}
\begin{proof}
See Lemma~12 in~\cite{erF5SB2013}.\qed
\end{proof}

\begin{remark}
\label{rem:generalized-spairorder-gvw}
In \cite{sw2011b} Sun and Wang explain a generalized criterion for
signature-based \grobner{} basis algorithms which is used in
\cite{gvwGVW2011} by Gao, Volny and Wang to generalize the original description
of the GVW algorithm given in \cite{gvwGVW2010}. For this a partial order on
$\module \times \ring$ is defined. Note that this is included in our combined
crterion described in Section~\ref{sec:combined-criteria}. This is very similar
to the rewrite order we defined in~\ref{def:rewriter}. Still there are some
slight differences: Sun and Wang call a partial order on $\module \times \ring$
\emph{admissible} if for any S-pair $a\alpha - b\beta$ that \sreduced{} to
$\gamma$ with $\sig\gamma = \sig{a\alpha}$ it holds that $\alpha \rleq \gamma$. Clearly, this
is covered by our definition of a rewrite order. Still an admissible partial
order could lead to several chains of ordered elements in \basis{} which are not
connected to each other. This would mean that a possible canonical rewriter in
siganture $T$ in chain $C_i$ cannot be used to discard a useless S-pair which
consists of a generator in chain $C_j$. So for each chain $C_i$ we would receive
an own set of rewriters in signature $T$:
\[
\rewriters{T,C_i} = \setBuilder{a\alpha}{ \alpha\in\basis\cup\hdsyz\text{, } a\text{ is a monomial
    and } \sig{a\alpha}= T, \alpha \text{ is in chain }C_i  }.
\]
Note that correctness and also
termination of \rba{} is not effected by this, but the criterion is not as
efficient as it is using a total order $\rleq$ on \basis{}.

\end{remark}

All in all, the efficiency of \rba{} depends on 
\begin{enumerate}
\item the order in which S-pairs are handled, and
\item the strength of the detection of useless S-pairs.
\end{enumerate}
We know already that $\pleqs$ is the best possible order for
$\pairset$ in terms of the size of the resulting signature \grobner{} basis and
the efficiency of the \sreduction{} steps.
The second point, as well as the size of \basis{} also depend on the chosen
rewrite order. So as a final step on our way understanding signature-based
\grobner{} basis algorithms we have to investigate the overall impact of rewrite
orders.

\subsection{Choosing a rewrite order}
\label{sec:best-rewrite-order}
When thinking about a possible rewrite order to choose we should look again the
set of all possible rewriters in signature $T$:
\[
\rewriters T = \setBuilder{a\alpha}{ \alpha\in\basis\cup\hdsyz\text{, } a\text{ is a monomial
    and } \sig{a\alpha}= T }.
\]
We want to choose the canonical rewriter $a\alpha$ in $T$ for further considerations in
\rba{} and discard all other elements. It is clear that we want to choose
$a\alpha$ in terms of ``being easier to \sreduce{} than the other elements in
$\rewriters T$''. From the point of view of \grobner{} basis computations there
are two canonical selections:
\begin{enumerate}
\item $\alpha$ has been added to \basis{} latest for all $\beta\in\basis$ such
that $b\beta\in\rewriters T$. Here we hope that $\alpha$ is better
\sreduced{} w.r.t. \basis{} and thus $a\alpha$ might be easier to handle in the
following.
\item Let $\hdp{a\alpha} \leq \hdp{b\beta}$ for any $b\beta \in \rewriters T$.
Choosing $a\alpha$ as canonical rewriter in signature $T$ we expect the fewest possible
\sreduction{} steps.
\end{enumerate}

It turns out that all signature-based \grobner{} basis algorithms known until
now choose one of the above options. Thus it makes sense to have a closer look
at those. 

\begin{definition}
\label{def:specific-rewrite-rules}
Let $\alpha,\beta \in \basis \cup \hdsyz$ during a computation of \rba{}.
\begin{enumerate}
\item We say that $\alpha \rleqff \beta$ if
$\beta \in \hdsyz$ or $\alpha$ has been added to \basis{} before $\beta$ is
added to \basis{}. Break ties arbitrarily.
\item\label{def:specific-rewrite-rules:sb}
We say that $\alpha \rleqsb \beta$ if $\sig\alpha \hdp\beta < \sig\beta
\hdp\alpha$ or if $\sig\alpha \hdp\beta = \sig\beta \hdp\alpha$ and $\sig\alpha
< \sig\beta$.
\end{enumerate}
\end{definition}

\begin{remark}\
\begin{enumerate}
\item Using $\pleqs$ in \rba{} $\alpha \rleqff \beta$ for $\alpha,\beta \in \basis$
induces that $\sig\alpha < \sig\beta$.
\item The suffix ``rat'' of $\rleqsb$ refers to the usual notation of this
rewrite order, for example, in~\cite{galkinTermination2012,erF5SB2013}. There
the \emph{ratios} of the signature and the polynomial lead term are compared:
\[\frac{\sig\alpha}{\hdp\alpha} < \frac{\sig\beta}{\hdp\beta}.\]
Multiplying both sides of the inequality by $\hdp\alpha \hdp\beta$ we get
the representation of $\rleqsb$ as in
Definition~\ref{def:specific-rewrite-rules}~(\ref{def:specific-rewrite-rules:sb}). 
We prefer the notation without ratios due to two facts: First of all we do not
need to extend $<$ on the ratios and introduce negative exponents. Secondly, we
can handle $\hdp\alpha = 0$ for elements $\alpha\in\hdsyz$.
\end{enumerate}
\end{remark}

\begin{lemma}
If there exists $\gamma \in \hdsyz$ such that $\gamma\in\rewriters T$ then
all S-pairs in signature $T$ are discarded in \rba{} using either $\rleqff$ or
$\rleqsb$.
\label{lem:combined-criterion-syz}
\end{lemma}

\begin{proof}
If $\gamma \in \hdsyz \cap \rewriters T$ then for all $\alpha$ in $\basis \cap
\rewriters T$ it holds by definition that $\alpha \rleqff \gamma$. Furthermore,
  $\alpha \rleqsb \gamma$ due to
$\sig\alpha \hdp\beta < \sig\beta\hdp\alpha$ where $\hdp\beta = 0$.
Thus no S-pair in signature $T$ is handled by \rba{}.\qed
\end{proof}

\begin{corollary}
If $f_1,\ldots,f_m \in \ring$ form a regular sequence then there is no
\sreduction{} to zero while \rba{} computes a signature \grobner{} basis for
$\ideal{f_1,\ldots,f_m}$ using $\potl$.
\label{cor:regular-sequence}
\end{corollary}

\begin{proof}
The homology of the Koszul complex $K^*$ associated to the regular sequence
$(f_1,\ldots,f_m)$ has the property that $H_\ell(K^*) = 0$ for $\ell>0$. Thus, there
exist only Koszul syzygies of the form $\proj{\alpha_i}\alpha_j -
\proj{\alpha_j} \alpha_i \in
\module$ where $\proj{\basis} = \{\proj{\alpha_1},\ldots,
\proj{\alpha_{k-1}}\}$ is the intermediate
\grobner{} basis for $\langle f_1,\ldots,f_{i-1}\rangle$ and $\alpha_k=\mbasis k \in
\module$ such that $\proj{\alpha_k} = f_i$.
By \characteristic~\ref{char:rba:potl} those syzygies are added in Line~\ref{alg:rba:initialsyz} of
Algorithm~\ref{alg:rba}. It follows that any zero reduction, corresponding to
such a syzygy is detected in advance.\qed
\end{proof}

\begin{corollary}
If $f_1,\ldots,f_m \in \ring$ form a homogeneous regular sequence then there is no
\sreduction{} to zero while \rba{} computes a signature \grobner{} basis for
$\ideal{f_1,\ldots,f_m}$ using $\dpotl$.
\label{cor:homog-regular-sequence}
\end{corollary}

\begin{proof}
This is clear by Corollary~\ref{cor:regular-sequence} and the fact that
\rba{} computes the signature \grobner{} basis for the input ideal by increasing
polynomial degree. Thus at each new degree step $d$ $\proj\basis$ is already a
$d'$-\grobner{} basis for $\ideal{\gen f}$ for all $d'<d$.\qed
\end{proof}

Another question to answer is why \rewritablenoarg{} is allowed to check both
generators of an S-pair and not only the one with higher signature.

\begin{lemma}
Assume \rba{} computing a signature \grobner{} basis for $\ideal{\gen f}$ using
$\rleqff$ or $\rleqsb$.
If \rewritablenoarg{} returns ``true'' for input S-pair $a\alpha-b\beta$,
   $\sig{a\alpha} > \sig{b\beta}$ due to
$b\beta$ being rewritable then \rba{} can discard $a\alpha-b\beta$.
\label{lem:rewritable-2nd-generator}
\end{lemma}

\begin{proof}
If $b\beta$ is rewritable then there exists $\gamma\in\basis\cup\hdsyz$, $\gamma
\neq \beta$ such
that $\gamma$ is the canonical rewriter in $\sig{b\beta}$. Let $\sig{c\gamma} =
\sig{b\beta}$ for some monomial $c$. Since $\beta \rleq \gamma$ and $\rleq$ is
either $\rleqff$ or $\rleqsb$ it follows from
Definition~\ref{def:specific-rewrite-rules} that  $\hdp{b\beta} \geq
\hdp{c\gamma}$. Two situations can happen:
\begin{enumerate}
\item If $\hdp{c\gamma} = \hdp{b\beta}$ then \rba{} handles the S-pair $a\alpha
- c\gamma$.
\item If $\hdp{c\gamma} < \hdp{b\beta}$ then there exists $\delta_i\in\basis$ and
monomials $d_i$ such that $\proj{b\beta}= \sum_{i=1}^\ell \proj{d_i \delta_i} + \proj{c\gamma}$
and $\sig{d_i\delta_i} < \sig{b\beta}$ for all $i \in \{1,\ldots,\ell\}$ since
\basis{} is a signature \grobner{} basis up to $\sig{b\beta}$.
Thus \rba{} handles for some $k \in \{1,\ldots,k\}$ $a\alpha - d_k\delta_k =
\lambda \spair\alpha{\delta_k}$ for some
monomial $\lambda \geq 1$. Note that this case includes $\gamma
\in \hdsyz$.\qed
\end{enumerate}
\end{proof}

Note that whereas we have to handle elements in \hdsyz{} explicitly for
$\rleqff$ there is no need to do so for $\rleqsb$: If $\beta \in \hdsyz$ then
for any $\alpha\in\basis{}$ $\sig\alpha \hdp\beta = 0 \leq \sig\beta\hdp\alpha$.

\begin{lemma}
If \rba{} uses $\rleqsb$ as rewrite order then there exists no singular top
\sreducible{} element in \basis{}.
\label{lem:rewrite-order-no-singular-basis-element}
\end{lemma}

\begin{proof}
\rba{} only regular \sreduces{} an S-pair in a non-syzygy signature $T$ if
\basis{} is not already a rewrite basis in signature $T$ (see
Lemma~\ref{lem:when-rb-reduces}), i.e. only if the
canonical rewriter $a\alpha$ in $T$ is regular top \sreducible{}. Let $b\beta$
be such a regular \sreducer{} of $a\alpha$. Let $a\alpha - b\beta$ regular
\sreduce{} to $\gamma$. Assume there exists $\delta \in \basis$ such that
$\sig{d\delta} = T$ and $\hdp{d\delta} = \hdp\gamma$. Since $a\alpha$ is the
canonical rewriter in signature $T$ w.r.t. $\rleqsb$ it holds that
\[\hdp{d\delta} \geq \hdp{a\alpha} > \hdp\gamma.\]
This contradicts the existence of such an element $\delta \in \basis$.\qed
\end{proof}

\begin{corollary}
Using $\rleqsb$ as rewrite order \rba{} computes a minimal signature
\grobner{} basis.
\end{corollary}

\begin{proof}
Clear by Lemma~\ref{lem:rewrite-order-no-singular-basis-element}. See also
Section~3.3 for more details.\qed
\end{proof}

The question is now if there exist examples where \rba{} using $\rleqff$
computes a  signature \grobner{} basis with more elements than the one achieved
by \rba{} using $\rleqsb$.

\begin{example}
\label{example:difference-rewrite-order}
Let $\field$ be the finite field with $7$ elements and let $\ring =
\field[x,y,z,t]$. Let $<$ be the graded reverse lexicographical monomial order
which we extend to $\potl$ on $\ring^3$. Consider the input ideal $I$ generated
by $f_1 = yz-2t^2$, $f_2 = xy+t^2$, and $f_3 = x^2z+3xt^2-2yt^2$. We present
the calculations done by \rba{} using $\rleqff$ in
Figure~\ref{figure:example:difference-rewrite-order}.

\begin{figure}
\centering
\setlength{\tabcolsep}{0.8em}
\begin{tabular}{cccc}
$\gbasis i\in\basis$ & reduced from & $\hd{\overline{\gbasis i}}$ & $\sig{\gbasis i}$ \\
\hline
$\gbasis 1$ & $\mbasis 1$ & $yz$ & $\mbasis 1$\\
$\gbasis 2$ & $\mbasis 2$ & $xy$ & $\mbasis 2$\\
$\gbasis 3$ & $\spair{\gbasis 2}{\gbasis 1}=z\gbasis 2 - x \gbasis 1$ &
$xt^2$ & $z \mbasis 2$\\
$\gbasis 4$ & $\mbasis 3$ & $x^2z$ & $\mbasis 3$\\
$\gbasis 5$ & $\spair{\gbasis 4}{\gbasis 2}=y \gbasis 4 - xz \gbasis 2$ &
$y^2t^2$ & $y \mbasis 3$\\
$\gbasis 6$ & $\spair{\gbasis 4}{\gbasis 3}=t^2 \gbasis 4 - xz \gbasis 3$ &
$z^3t^2$ & $t^2 \mbasis 3$\\
$\gbasis 7$ & $\spair{\gbasis 6}{\gbasis 1}=y \gbasis 6 - z^2t^2 \gbasis 1$ &
$y^2t^4$ & $yt^2 \mbasis 3$\\
\end{tabular}
\caption{Computations for \rba{} in
  Example~\ref{example:difference-rewrite-order}.}
\label{figure:example:difference-rewrite-order}
\end{figure}

\rba{} with $\rleqsb$ regular \sreduce{} the same S-pairs except the last one:
In signature $yt^2 \mbasis 3$ we have $y \gbasis 6, t^2 \gbasis 5 \in
\rewriters{yt^2 \mbasis 3}$. $\rleqff$ prefers $y \gbasis 6$ over $t_2
\gbasis 5$, thus the S-pair $y \gbasis 6 - z^2 t^2 \gbasis 1$ is handled.
$\rleqsb$ on the other hand has $t^2 \gbasis 5$ as canonical rewriter in
signature $yt^2 \mbasis 3$ as $\hdp{t^2 \gbasis 5} = y^2t^4 < yz^3t^2 = \hdp{y
\gbasis 6}$. With this choice no S-pair in signature $yt^2 \mbasis 3$ is
handled and thus \rba{} terminates.
\end{example}

Note that the canonical rewriter in signature $yt^2 \mbasis 3$ w.r.t. $\rleqsb$
is not regular top \sreducible{}. So by Lemma~\ref{lem:when-rb-reduces}
\rba{} does not reduce any S-pair in this signature. $\rleqff$ chooses its
canonical rewriter $y \gbasis 6$ wrong in the sense that $y \gbasis 6$ can be
further reduced, but only until it reaches $t^2 \gbasis 5$. Whereas this
computation is important for \rba{} in order to compute a rewrite basis w.r.t.
$\rleqff$, it is not needed to achieve a signature \grobner{} basis for $I$.

We conclude this section with the following summary: As we have seen \rba{} is mainly
parametrized by three properties:
\begin{enumerate}
\item the monomial order $<$ and its extension to $\module$,
\item the pair set order $\pleq$ on $\pairset$, and
\item the rewrite order $\rleq$.
\end{enumerate}
We see that even though there are so many different notions of signature-based
\grobner{} algorithms in the literature, all those implementations boil down to
variations of two of the above mentioned three orders: All known algorithms have in
common to use $\pleqs$ on $\pairset$.\footnote{There is some slight difference in the
original presentation of \ff{} in~\cite{fF52002Corrected} which is discussed in
\characteristic~\ref{char:f5}~(\ref{char:f5:pair-set-order}).}

\begin{remark}
Note that such attempts of generalizing the description of signature-based
\grobner{} basis algorithms have already been done, for example,
in~\cite{huang2010,sw2011b,panhuwang2012,panhuwang2013}. As we have already
pointed out in the introduction of this section all of these characterizations are
similar and included in our attempt using \rba{}. The difference in notation are
rather obvious (see also sections~\ref{sec:efficient-implementations-f5}
 --~\ref{sec:efficient-implementations-sb}), thus we relinquish to
give comparisons further than the ones
depicted already in Sections~\ref{sec:spair-elimination} and
~\ref{sec:rewrite-bases}. 
\end{remark}                        

Next we discuss known and efficient implementations of signature-based
\grobner{} basis algorithms as variants of \rba{}.
Note that all algorithms described in the following can be implemented
with any compatible extension to the monomial order. When algorithms were
initially presented with a fixed module monomial order we take care of this.
Still, the only real difference of the implements boils down to the rewrite
orders used.

\section{Faug\`ere's \ff{} algorithm and variants}
\label{sec:efficient-implementations-f5}
In 2002 Faug\`ere presented the \ff{} algorithm (\cite{fF52002Corrected}). This
was the first publication of a signature-based \grobner{} basis algorithm and
introduced the notion of a signature.

In~\cite{erF5SB2013} the connection between \rba{} and \ff{} is already given,
so we give a short review and refer for details to that paper. \ff{}, as
presented in~\cite{fF52002Corrected} uses $\potl$ as extension of the underlying
monomial order $<$.\footnote{\label{fn:potl}Strictly speaking this is not completely true,
  \ff{} as presented in~\cite{fF52002Corrected} uses $\potlp$ defined by
    $a\mbasis i \potlp b\mbasis j$ if and only if $i>j$ or $i=j$ and $a<b$. The
    only difference is to prefer the element of lower index instead of the one
    of higher index. In order to unify notations we assume in the
    following that \ff{} means ``\ff{} uses $\potl$ as module monomial order''.}

\begin{remark}
\label{rem:f5-rewrite-order}
In~\cite{erF5SB2013} it is assumed that \ff{} uses $\rleqff$ as
its rewrite order. Note that this is not true for the initial presentation of
\ff{} in~\cite{fF52002Corrected}:
\ff{} uses $\potl$, so it computes incrementally a \grobner{} basis for
$\ideal{f_1,\ldots,f_i}$ for increasing $i$. For each such index $i$ the
algorithm stores a list of so-called ``rewrite rules'': $\textsc{Rule}_i$.
The S-pairs are first taken by minimal possible degree $d := \deg(\proj{a\alpha})$
for $a\alpha$ being a generator of an S-pair. Once this choice is done this list of
S-pairs, denoted by $\pairset_d$ is handled by subalgorithm \textsc{Spol}.
There S-pairs are checked by the criteria and new rewrite rules are added to the
end of the list $\textsc{Rule}_i$. Once this step is done, the remaining S-pairs
in $\pairset_d$ are handed to the subalgorithm \text{Reduction}. Not until this
point the S-pairs in $\pairset_d$ are sorted by increasing signature. This leads
to the following effects:
\begin{enumerate}
\item If the input is homogeneous, \ff{} reduces S-pairs by increasing
signature, but the rewrite rules are \emph{not} sorted by increasing signature.
\item If the input is inhomogeneous then \ff{} need not even reduce S-pairs by
increasing signatures as it is pointed out in~\cite{ederInhomog2013}. Note that
this behaviour is still covered by \rba{} and using a corresponding pair set
order $\pleq \neq \pleqs$. Still, as discussed in sections~\ref{sec:gsb}
and~\ref{sec:rewrite-bases} the best possible pair set order is $\pleqs$ and it
is shown in~\cite{ederInhomog2013} that \ff{} can easily be equipped with it.
\end{enumerate}
The fact about not handling S-pairs by increasing signatures we describe in more
detail in \characteristic~\ref{char:f5}~(\ref{char:f5:pair-set-order}). The problem of
ordering the rewrite rules is more difficult: As described
in~\cite{fF52002Corrected}, \ff{} might not use $\rleqff$ as rewrite order: For
\ff{} the canonical rewriter in signature $T$ is the element in
$\textsc{Rule}_i$ which was added last. But at the time of concatenation the
S-pairs are not sorted by increasing signature! So the following situation can
happen: Assume we have two S-pairs in degree $d$ with signatures $z \mbasis i$
and $x \mbasis i$. We can assume that in $\pairset_d$ they are ordered like
$[\ldots, x \mbasis i, \ldots, z \mbasis i, \ldots]$.
Let us assume that both S-pairs are not rewritable, so we reduce both. Now,
after $\pairset_d$ is sorted by increasing signature, \ff{} first reduces the
S-pair with signature $z \mbasis i$ to $\alpha$, and later on the one with
signature $x \mbasis i$ to $\beta$. Generating new S-pairs we could have two
S-pairs in $\pairset_{d+2}$ with signature $xyz \mbasis i$: $\spoly \alpha \gamma$
and $\spoly \beta \delta$. In this situation, \ff{} would remove $\spoly \beta
\delta$ and keep $\spoly \alpha \gamma$ since the signature $z \mbasis i$ was added
to $\textsc{Rule}_i$ after $x \mbasis i$ had been added. So in our notation
$\alpha$ is the canonical rewriter in signature $xyz \mbasis i$. Clearly, using
$\rleqff$ $\beta$ is the canonical rewriter in $xyz \mbasis i$. 

Since $\beta$ was computed after $\alpha$ from the algorithm's point of view
$\beta$ might the better element. So it makes sense to optimize \ff{} as presented
in~\cite{fF52002Corrected} to use $\rleqff$.

Moreover, note that in~\cite{erF5SB2013} the authors assume this optimization
already. For a complete proof of termination of \ff{} as presented
in~\cite{fF52002Corrected} we refer the reader to~\cite{galkinTermination2012}.
\end{remark}

In the following we assume that \ff{} uses $\rleqff$ as rewrite order, then
the only difference left from its original description is now
the fact that \ff{} checks the possible \sreducers{} $b\beta$ of an element
$\alpha$ if they are not syzygy and not rewritable.

\begin{lemma}[Lemma~15 in~\cite{erF5SB2013}]
\label{lem:f5-reduction}
Let $\alpha\in\module$, let $t$ be a term of $\proj\alpha$ and let \basis{} be a
rewrite basis up to signature $\sig\alpha$. Then $t$ is regular
\sreducible{} if and only if it is reducible in \ff{}.
\label{lem:comparison-to-f5-reduction}
\end{lemma}

\begin{proof}
Follows also from Lemma~\ref{lem:existence-regular-sreducer}.\qed
\end{proof}

So from Lemma~\ref{lem:f5-reduction} it follows that checking possible reducers
by \rewritablenoarg{} in \rba{} does not change the algorithm's behaviour and
is thus optional. In Section~\ref{sec:f4-f5} we see that the idea of checking
the \sreducer{}s by the criteria comes from a linear algebra point of view.

Let us underline the following characteristics of \ff{}.

\begin{characteristics}\
\label{char:f5}
\begin{enumerate}
\item
\label{char:f5:handling-zero-reductions}
Note that, even so we assume the optimization of \ff{}'s rewrite order as
described in Remark~\ref{rem:f5-rewrite-order}, \ff{} still does not completely
implement $\rleqff$ but a slightly different rewrite order:
The requirement $\alpha \rleq \beta$
whenever $\beta \in \hdsyz$ from $\rleqff$ is relaxed to $\beta = \koz{\mbasis
i}{\mbasis j}$ for $1\leq i < j \leq m$. Thus non-Koszul syzygies in
\hdsyz{} have the same priority as elements in \basis{}. The idea to improve
computations by using zero reductions directly instead was introduced first
in an arXiv preprint of~\cite{apF5CritRevised2011} by Arri and Perry in 2009 
as well as in~\cite{ggvGGV2010} by Gao, Guan and Volny.
\item
\label{char:f5:pair-set-order}
Note that in~\cite{fF52002Corrected} the \ff{} algorithm is described in the
vein of using linear algebra for the reduction steps (see
Section~\ref{sec:f4-f5} for more details). Instead of ordering the
pair set by increasing signatures it is ordered by increasing degree of the
corresponding S-polynomial. A subset $\pairset_d$ of S-pairs at minimal given
degree $d$ is then handled by the \textsc{Reduction} procedure. There, all these
S-pairs (corresponding to degree $d$ polynomials) are sorted by increasing
signature. As already discussed in~\cite{ederInhomog2013}, for homogeneous input
this coincides with using $\pleqs$ since then the degree of the polynomial part
and the degree of the signature are the same. For inhomogeneous input
\ff{}'s attempt might not coincide with $\pleqs$.
In~\cite{galkinTermination2012} Galkin has given a proof for termination of \ff{}
taking care of this situation. Note that in such a situation one might either
prefer to use $\pleqs$ (as pointed out in~\cite{ederInhomog2013}) or saturate
resp. desaturate the elements during the computation of the algorithm.
\item Furthermore, thinking in terms of linear algebra also explains why
in~\cite{fF52002Corrected} higher signature
reductions lead to new S-pairs which are directly added to the \textsc{ToDo}
list in subalgorithm \textsc{TopReduction} and not prolonged to the situation
when a new element is
added \basis{} as it is done in \rba{}: Assuming homogeneous input, in a
Macaulay matrix $M_d$ (see, for example, Section~\ref{sec:matrix-f5}) all
corresponding rows are already stored. Thus a higher signature S-pair (in
\rba{} et al. due to single polynomial \sreduction{} prolonged to a later step)
corresponds to a reduction of a row by some other one below. All possible
S-pairs of degree $d$ are handled at once thus one can directly execute the new
S-pair without generating it later on.
\end{enumerate}
\end{characteristics}

Clearly, the \ff{} criterion and the Rewritten criterion are just special cases
of the syzygy criterion (Lemma~\ref{lem:syzygy-criterion}) and the singular
criterion (Lemma~\ref{lem:singular-criterion}), respectively. For even more details
on how to translate notions like ``canonical rewriter'' to \ff{} equivalents like
``rewrite rules'' we refer to~\cite{erF5SB2013} Section~5.

Moreover, \ff{} implements the \sreduction{} process different to the
description in \rba{}: Instead of prolonging an \sreduction{} $\alpha - b \beta$
with a reducer $b\beta$ of signature $\sig{b\beta} > \sig\alpha$ to the
generation of the S-pair $b\beta -\alpha$ later on, \ff{} directly adds $b\beta
-\alpha$ to the todo list of the current degree in \textsc{Reduction}. Assuming
homogeneous input this makes sense. Again, we see in Section~\ref{sec:f4-f5}
that this is coming from an \ffo{}-style implementation of the
\sreduction{} process.

In the last decade several optimizations and variants of \ff{} where presented.
Using \rba{} we can easily categorize them.

\begin{variants}
In~\cite{fF52002Corrected} three variants of \ff{} are mentioned shortly
without going into detail about their modifications:
\begin{enumerate}
\item \ffpr{} denotes a variant of \ff{} similar to \ffr{} (see
    Section~\ref{sec:f5r}) resp. \ffc{} (see Section~\ref{sec:f5c}): For
inhomogeneous input one can optimize computations by homogenizing the
computations of the intermediate \grobner{} basis $G_i$ for
$\ideal{f_1,\ldots,f_i}$. Before adding the homogenized $f_{i+1}$ one
dehomogenizes $G_i$ and interreduces $G^\textrm{deh}_i$ to $B_i$. This $B_i$ can
then be used for checks with the syzygy criterion as well as for reduction
purposes. We refer to sections~\ref{sec:f5r} and~\ref{sec:f5c} for details on
signature handling in this situation.
\item \ffprpr{} denotes the variant of \ff{} using $\dpotl$ as compatible module
monomial order. Thus, instead of an incremental computation w.r.t. the initial
generators $\gen f$ the algorithm handles elements by increasing degree. Note
that in case of regular input \ffprpr{} computes no zero reduction, whereas this
is possible for $\schl$.
\item The variant \mff{} which uses linear algebra for reduction purposes is described in
Section~\ref{sec:matrix-f5} in detail.
\end{enumerate}
\end{variants}

Note that besides the variants presented in the following there are even more
publications about optimizations and generalizations of the \ff{} algorithm for
computing \grobner{} bases, for
example, see~\cite{faugere-svartz-2012,faugeresvartz-2013,faugeresafeyverron-2013}. Also the main
results in these publications are presented for \ff{}, they do not depend on the
\grobner{} basis algorithm used. Here we are giving a survey especially for
signature-based \grobner{} basis algorithms, thus taking care of not
signature-based tailored research is out of scope of this publication.

Moreover, there are first works in using signature-based criteria for computing
involutive bases (\cite{gerdtHashemiG2V,ghmInvolutiveF5-2013}). 

\subsection{\ffr{} -- Improved lower-index \sreduction{}}
\label{sec:f5r}
In 2005 Stegers reviewed \ff{} in~\cite{stF5Rev2005}. There he introduced a new
variant of \ff{} improving the reduction process. Due to the incremental
structure of \rba{} when using $\potl$ one first computes a signature
\grobner{} basis for $\ideal{f_1,f_2}$, then for $\ideal{f_1,f_2,f_3}$, and so
on. 
Since the intermediate bases need not be minimal Stegers suggested to use in
step $k$ of the algorithm not $\basis_{k-1}$ for reduction pruposes. Instead it
is preferable to reduce the corresponding
\grobner{} basis $G_{k-1} = \set{\proj\alpha \mid \alpha \in \basis_{k-1}}$ to
the reduced \grobner{} basis $B_{k-1}$ for $\ideal{f_1,\ldots,f_{k-1}}$.
Since for all elements handled by \rba{} in iteration step $k$ the
signature has an index $k$ and all elements in $\basis_{k-1}$ have signature
index at most $k-1$ \sreductions{} are always allowed when using $\potl$ and the
signatures need not be checked.

Note that $B_{k-1}$ is only used for the reduction purposes, new S-pairs are
still generated using elements in $\basis_{k-1}$ since otherwise the signatures
would not be correct.

\subsection{\ffc{} -- Improved S-pair generation}
\label{sec:f5c}
Based on Stegers' idea, Eder and Perry presented in 2009 the \ffc{} algorithm
in~\cite{epF5C2009}. Whereas \ffr{} uses the reduced \grobner{} basis $B_{k-1}$
for $\ideal{f_1,\ldots,f_{k-1}}$ only for reduction purposes, \ffc{} extends this
to the generation of new S-pairs in iteration step $k$.

Once \rba{} finishes computing $\basis_{k-1}$ one reduces the corresponding
\grobner{} basis $\proj{\basis_{k-1}}$ to $B_{k-1}$ as described above. Let
$B_{k-1} := \set{g_1,\ldots,g_{m'}}$, then
one introduces $\basis'_{k-1} = \set{\mbasis 1,\ldots, \mbasis{m'}}$. Moreover,
one has to redefine the homomorphism $\alpha \mapsto \proj\alpha$ to go from $\ring^{m'}$
to $\ring$ by sending $\mbasis i$ to $g_i$ for all $i\in\set{1,\ldots,m'}$.

Starting iteration step $k$, \rba{} now computes the signature \grobner{} basis
for $\ideal{g_1,\ldots,g_{m'},f_k}$. Of course, at that point another extension
of the homomorphism $\alpha \mapsto \proj\alpha$ has to be done, since now we
are mapping $\ring^{m'+1} \rightarrow \ring$: We define that
$\proj{\mbasis{m'+1}} := f_k$.

It is shown in Theorem~32 and Corollary~33 of~\cite{epF5C2009} that with this
resetting of the signatures the number of useless \sreductions{} is not
increased, but instead the number of S-pairs generated in step $k$ is decreased.

\begin{variants} \
\begin{enumerate}
\item Due to \characteristic~\ref{char:f5}~(\ref{char:f5:handling-zero-reductions}) one also wants to implement
\ffc{} using $\rleqff$ in order to use zero reductions directly. In 2011, Eder
and Perry denoted this variant \ffa{} in~\cite{epSig2011}.
\item In~\cite{ederImprovedF52013} Eder improves the idea of \ffc{} slightly:
By symbolically generating S-pairs of elements in $\basis'_{k-1}$ (they all
already reduce to zero) signatures useful for discarding S-pairs in iteration
step $k$ can be made available a bit earlier. Thus, in terms of \rba{}, \hdsyz{}
is initialized not only with the signatures of the Koszul syzygies but also with the
signatures of other, already known syzygies. The idea presented there can be used
in any incremental signature-based \grobner{} basis algorithm. The corresponding
variants are denoted, for example, \iffc{} and \iggv{}.
\end{enumerate}
\end{variants}

\subsection{Extended \ff{} criteria}
\label{sec:ext-f5}
In 2010, Ars and Hashemi published~\cite{ahExtF52010} in which they generalized
the \ff{} criterion and the Rewritten criterion in the sense of using different
extensions of the monomial order $<$ on $\module$. These variants are achieved by
using \rba{} not with $\potl$ but one of the following two orders
proposed in~\cite{ahExtF52010}.

\begin{definition}
Let $<$ be a monomial order on $\ring$ and let $a e_i, b e_j$ be two module
monomials in $\module$.
\begin{enumerate}
\item $a e_i <_1 b e_j$ if and only if\footnote{Note that for $<_1$ to be a
total order we need to ensure that $\hdp{\mbasis i} \neq \hdp{\mbasis j}$
whenever $i \neq j$. Having the input ideal $I=\ideal{f_1,\ldots,f_m}$ this can
be achieved by an interreduction of the $f_i$s before entering \rba{}.} either
\begin{center}
$
\begin{array}{ccccccc}
a \hdp{\mbasis i} &<& b \hdp{\mbasis j} &
\text{ or}&&&\\
a \hdp{\mbasis i} &=& b \hdp{\mbasis j}
&\text{ and }& \hdp{\mbasis i} &<& \hdp{\mbasis j}.\\
\end{array}
$
\end{center}
\item $a e_i <_2 b e_j$ if and only if either
\begin{center}
$
\begin{array}{ccccccccccc}
\deg\left(\proj{a \mbasis i}\right) &<& \deg\left(\proj{b \mbasis j}\right) &
\text{ or}&&&&&&&\\
\deg\left(\proj{a \mbasis i}\right) &=& \deg\left(\proj{b \mbasis j}\right)
&\text{ and }& a &<& b &\text{ or}&&&\\
\deg\left(\proj{a \mbasis i}\right) &=& \deg\left(\proj{b \mbasis j}\right)
&\text{ and }& a &=& b &\text{ and }& i &<& j.
\end{array}
$
\end{center}
\end{enumerate}
\end{definition}

Ars and Hashemi implemented the original \ff{} algorithm and their variants of it in the
computer algebra system \magma{} and give timings for several \grobner{} basis
benchmarks. Their variants seem to be more efficient than the original
\ff{} algorithm in most of the examples. Still there exist input, for example
the \textsc{Schrans-Troost} benchmark, for which $\potl$ seems to be more
efficient. Using a framework like \rba{} such behaviour can be tested easily.

\section{Exploiting algebraic structures}
\label{sec:exploit-algebraic-structures}
In this section we present variants of \ff{} that use knowledge of underlying
algebraic structures in order to improve the computations. Note that there exist
more variants doing this besides the $3$ ones we are discussing here, see,
for
example,~\cite{faugere-svartz-2012,faugeresvartz-2013,faugeresafeyverron-2013}
(see also Figure~\ref{fig:decade-in-sig-based-algorithms}). 
The improvements in those variants are not specific to signature-based
\grobner{} basis algorithms, thus we waive to discuss them here.

It is clear that in the future a lot more improvements in this direction can be
expected. Exploiting algebraic structures helps to find more syzygies on the one
hand and to increase the independence of polynomials on the other hand. Both has
a positive influence on the computation of (signature) \grobner{} bases.

\subsection{\fftwo{} -- Improved computations over $\mathbb{F}_2$}
\label{sec:f52}
An easy way to improve \ff{}'s performance over small finite fields is to add
the field equations to \hdsyz{}. When breaking the first hidden field equations
(HFE) challenge in 2003 (\cite{FJ03}) the variant \fftwo{} was used which adds to $\gen f$
the field equations $x_i^2-x_i=0$ in $\mathbb{F}_2$. With this the rewritable
signature criterion is more powerful since Koszul syzygies generated by
those supplementary equations have low signatures. The HFE challenge consists of
$80$ equations in degree $2$. A \grobner{} basis computation of such a system
was intractable beforehand.

\subsection{An \ff{} variant for bihomogeneous ideals generated by polynomials
  of bidegree $(1,1)$}
\label{sec:f5-bihomog}
In 2012 Faug\`ere, Safey El-Din and Spaenlehauer published a variant of
\ff{} dedicated to multihomogeneous, in particular, bihomogeneous systems
generated by bilinear polynomials (\cite{FSS10b}). The main result is to
exploit the algebraic structure of bilinear systems to enlarge \hdsyz{}.

In Corollary~\ref{cor:regular-sequence} we see that \rba{} and thus also
\ff{} computes no reduction to zero if the input sequence is regular.
Whereas a randomly chosen homogeneous polynomial system is regular, this is not
the case for multihomogeneous polynomial systems. Those systems appear, for
example, in cryptography or coding theory. Due to the non-regularity
\ff{} does not remove all zero reductions.

Let $\gen f \in \field[x_0,\ldots,x_{n_x},y_0,\ldots,y_{n_y}]$ be bilinear
polynomials, let $F_i$ denote the sequence $f_1,\ldots,f_i$ and let $I_i$
denote the ideal $\ideal{F_i}$.
The main result is that the kernel of the Jacobian matrices $\jac x {F_i}$ and
$\jac y {F_i}$ w.r.t.
$x_0,\ldots,x_{n_x}$ and $y_0,\ldots,y_{n_y}$, respectively, correspond to those
reductions to zero \ff{} does not detect.
In general, all elements in these kernels are vectors of maximal minors of the
corresponding Jacobian matrices.

Assuming the incremental structure of \ff{} by using $\potl$ it is shown that
the ideal $I_{i-1} : f_i$ is spanned by $I_{i-1}$ and the maximal minors of
$\jac x {F_{i-1}}$ (for $i-1>n_y$) and $\jac y {F_{i-1}}$ (for $i-1>n_x$). The
lead ideal of $I_{i-1} : f_i$ corresponds to the zero reductions associated
to $f_i$. In order to get rid of them one needs to get results for the ideals
generated by the maximal minors of the Jacobian matrices. In~\cite{FSS10b} it is
shown that in general \grobner{} bases for these ideals w.r.t. the graded reverse
lexicographical order are linear combinations of the generators. Thus, once a
\grobner{} basis of $I_{i-1}$ is known (which we can assume due to the
incremental structure of \ff{}) one can efficiently compute a \grobner{} basis
of $I_{i-1}: f_i$. It follows that for generic bilinear systems this variant of
\ff{} does not compute any zero reduction.

It follows that for \rba{} all one has to do is to add the computation of the
maximal minors of the jacobian matrices and add the corresponding syzygies resp.
signatures to \hdsyz{} in Line~\ref{alg:rba:initialsyz} of
Algorithm~\ref{alg:rba}.

\subsection{An \ff{} variant for SAGBI \grobner{} bases}
\label{sec:f5-sagbi}
Faug\`ere and Rahmany presented in 2009 an adjusted variant of \ff{} for
computing so-called SAGBI \grobner{} bases (\cite{FR09}). A SAGBI
\grobner{} basis is the analogon of a \grobner{} basis for ideals in
$\field$-subalgebras. We introduce notation as much as needed to explain the
changes in \ff{}, in particular, \mff{}. For more details on the theory of
SAGBI bases we refer, for example, to~\cite{kapur-madlener-sagbi-1989}.

In this subsection let $G \subset \gl n$ be a subgroup of $n \times n$ invertible matrices over
$\field$. Moreover, we assume that $\field$ has characteristic zero or $p$ such
that $p$ and $\card G$ are coprime.

\begin{definition} \
\begin{enumerate}
\item A polynomial $f\in \ring$ is called \emph{invariant (w.r.t. $G$)} if
$f(Ax) = f(x)$ for all $A \in G$. The set of all polynomials of
$\ring$ invariant w.r.t. $G$ is denoted $R^G$.
\item For $\card G < \infty$ the \emph{Reynolds operator (for $G$)} is the map
$\rey: \ring \rightarrow \invring$ defined by $\rey(f) = \frac{1}{\card G}
\sum_{A \in G} f(Ax)$.
\end{enumerate}
\end{definition}

\begin{proposition}[\cite{cloIdeals2007}]
Let $\rey$ be the Reynolds operator for a finite group $G \subset \gl n$. Then
the following properties hold:
\begin{enumerate}
\item $\rey$ is $\field$-linear.
\item $f \in \ring \Longrightarrow \rey(f) \in \invring$.
\item $f \in \invring \Longrightarrow \rey(f) = f$.
\end{enumerate}
\end{proposition}

Even if $\invring$ might not be finite dimensional as $\field$-vector space,
there exists a decomposition in finite dimensional homogeneous components,
$\invring = \oplus_{d\geq 0}\invring_d$.\footnote{Note that $\ring = \oplus_{d\geq
  0} \ring_d$ and that the action of $G$ preserves the homogeneous components.}
For any term $t \in \ring$ $\rey(t)$ is a homogenous invariant, called
\emph{orbit sum}. Clearly, the set of orbit sums is a vector space basis for
$\invring$.

Here we assume that $f_1,\ldots,f_m$ are homogeneous, invariant polynomials in
$\ring$ and $I$ resp. $I^G$ represent the ideal generated by
$f_1\ldots,f_m$ in $\ring$ resp. $\invring$

\begin{definition} \
\begin{enumerate}
\item A subset $F \subseteq I^G$ is a \emph{SAGBI \grobner{} basis for $I^G$} (up to
degree $d$) if
$\set{\hd f \mid f \in F}$ generates the lead ideal of $I^G$ as an ideal over
the algebra $\ideal{\hd f \mid f \in \invring}$ (up to degree $d$).
\item Let $f,g,p \in \invring$ such that $f\neq0\neq p$. $f$ \emph{SG-reduces}
to $g$ modulo $p$ if there exists a term $t$ of $f$ such that there exists an
$s$ in the set of lead terms of $\invring$ such that $s \hd p = t$ and $g =
f - \frac{\lc t}{\lc p \lc{\rey(s)}} \rey(s) p$.
\end{enumerate}
\end{definition}

Clearly, one can speak of SG-reduction w.r.t. a finite subset $F \subseteq
\invring$. With this a SAGBI \grobner{} basis can be defined similar to a usual
\grobner{} basis:

\begin{proposition}
Let $F$ be a subset of an ideal $I^G \subseteq \invring$. The following are
equivalent:
\begin{enumerate}
\item $F$ is a SAGBI \grobner{} basis for $I^G$.
\item Every $h \in I^G$ SG-reduces to zero w.r.t. $F$.
\end{enumerate}
\end{proposition}

Note that a SAGBI \grobner{} basis might not be finite. 

Instead of using elimination techniques in order to compute a SAGBI
\grobner{} basis for a given ideal $I^G \subseteq \invring$ one can use the
ideas of Thi\'ery who presented in~\cite{thiery-sagbi-2001} a variant of
Buchberger's algorithm.

Faug\`ere and Rahmany use in~\cite{FR09} the \mff{} description of \ff{} to present
the modifications: Let $\gen f \in \invring$ be the homogeneous input elements.
First one defines the so-called \emph{invariant Macaulay matrix}
$\mac d i$ generated by $\rey(t_{j,k}) f_k$ for $1\leq k \leq i$ and terms $t_{j,k}$ such
that $\deg(t_{j,k}) = d -\deg(f_k)$.
Two modifications to the usual Macaulay matrix have to be made:
\begin{enumerate}
\item Instead of labelling the rows of $\mac d i$ by $t_{j,k} \mbasis k$ one uses
$\rey(t_j,k) \mbasis k$. 
\item Instead of labelling the columns by the usual monomials $m_\ell$ they are indexed by $\rey(m_\ell)$.
\end{enumerate}

Besides this no further changes need to be done. The variant of \mff{} presented
here assumes $\potl$ as module monomial order and $\rleqff$ as rewrite order.
One checks for any row labelled by $\rey(t_{j,k}) \mbasis k$ if $f_k$ is the canonical
rewriter in signature $\sig{t_{j,k} \mbasis k}$ and removes the row otherwise. In the 
description of \mff{} this is equivalent to the existence of a row with corresponding
lead term $t_{j,k}$ in a matrix that was previously reduced to row echelon form.

\section{\ff{} and the quest of termination}
\label{sec:f5-termination}
Until 2012 there was still no complete proof of \ff{}'s termination
given. Thus a lot of variants of \ff{} where published in the meantime which have small
adjustments in order to ensure termination.

The main problem with the proof of \ff{}'s termination given
in~\cite{fF52002Corrected} is Theorem~2: It assumes that if the input of
\ff{} is a regular sequence of homogeneous elements then \ff{} does enlarge the
lead ideal after each call of the subalgorithm \textsc{Reduction}. In
Section~8 of~\cite{fF52002Corrected} an example of \ff{} computing a
\grobner{} basis for a regular sequence of three homogeneous elements is given.
In the last call of \textsc{Reduction} only one element, $r_{10}$, is added to
\basis{} with $\hdp{r_{10}} = y^6t^2$. In degree $d=7$ \ff{} has already added
element $r_8$ to \basis{} with $\hdp{r_8} = y^5t^2$. Thus the statement of
Theorem~2, on which the proof of termination of \ff{} in~\cite{fF52002Corrected}
is based on, is not true.

\subsection{Proving \ff{}'s termination}
\label{sec:f5-termination-directly}
At least since Galkin's proof in~\cite{galkinTermination2012} termination of \ff{} is clear.
Several other publications include proofs of \ff{}'s termination, most of them
are only slight variants or simplifications of Galkin's
(see~\cite{panhuwang2012,panhuwang2013}), some are proving termination for
slight variants of \ff{} (see~\cite{erF5SB2013}). The main idea is based on
partitioning \basis{} into sets
\[R_r := \left\{ \alpha_i \;\left\lvert \;\frac{\sig{\alpha_i}}{\hdp{\alpha_i}}
= r \right.\right\}\]
for given ratios $r$. The proof of \ff{}'s termination is then done in two
steps:
\begin{enumerate}
\item One shows that there are only finitely many non-empty sets $R_r$.
\item $\# R_r < \infty$ for any non-empty set $R_r$.
\end{enumerate}
As one can easily see, this attempt can be used for any signature-based
\grobner{} basis algorithm related to \rba{}, thus also termination of
\gvw{} and variants (see Section~\ref{sec:efficient-implementations-sb}) can be
handled in the same way.

In~\cite{panhuwang2012} and~\cite{panhuwang2013} Pan, Hu and Wang present
another attempt in proving \ff{}'s termination. For this they do not only focus
on \ff{} but give generalized algorithms in order to use known termination of
algorithms like \gvw{} (see Section~\ref{sec:gvw}). They give a generalized
\ff{} algorithm called \ffgen{} for which they can easily prove termination in
the vein of Eder and Perry's proof of termination of general signature-based
\grobner{} basis algorithms given in~\cite{epSig2011}. Both publications use the
notation introduced by \ggv{} resp. \gvw{} and then further adopted by Huang
in~\cite{huang2010}. We refer to the corresponding sections (\ref{sec:huang}
and~\ref{sec:gvw}, respectively) for a dictionary translating the notation given
here to theirs. Moreover, note that~\cite{panhuwang2013} takes care of the
problem with the insertion of rewrite rules in the original \ff{} algorithm
discussed in Remark~\ref{rem:f5-rewrite-order}: Instead of using lists
$\textsc{Rule}_i$ for rewrite rules they directly check rewritability by the
order of elements in \basis{} as done in \rba{}, too. \ffgen{} now has a
generalized insertion strategy for new elements in \basis{}, called
\textsc{insert\_F5GEN}. This mirrors the usage of different rewrite orders
$\rleq$ as explained in Section~\ref{sec:rewrite-bases}.
Whereas~\cite{panhuwang2013} focusses on \ff{},~\cite{panhuwang2012} covers also
\gvw{} and variants.

In~\cite{erF5SB2013} Eder and Roune give an easier proof for
\ff{}'s termination assuming that \ff{} uses $\rleqff$ as rewrite order, see
Remark~\ref{rem:f5-rewrite-order}.

\subsection{Variants of \ff{} to ensure termination algorithmically}
\label{sec:f5-termination-algorithmically}
The following variants are still not deprecated, they generate lower degree bounds for
an earlier termination of \ff{}. All the changes presented here can easily be transfered to
\rba{}. Furthermore, note that all the following ideas for modifying
\ff{} to ensure termination assume homogeneous input. The main difference to
proving \ff{}'s termination directly as explained in
Section~\ref{sec:f5-termination-directly} is that the variants presented next
provide algorithmic, termination ensuring modifications to \ff{}.

\subsubsection{\fft{} -- Using the Macaulay bound.}
In~\cite{gashPhD2008} and~\cite{gashF5t2009} Gash presents the variant
\fft{} which makes use of the Macaulay bound $M$ (see, for
    example,~\cite{Macaulay02,laz83,BFS05}) for regular sequences.
Once the degree of the polynomials treated in the algorithm exceed $2M$ redudant
elements (i.e. elements $\alpha$ such that $\hdp\alpha$ is already in the
lead ideal of the current partly computed \grobner{} basis) are added to a
different set $D$. Whenever \ff{} returns such a redundant element $\alpha$,
$\proj\alpha$ is reduced (\emph{not} \sreduced{}!) completely w.r.t. $\basis
\cup D$. All corresponding signatures and rewrite rules are marked to be
invalid. Any newly computed S-pair with one generator out of $D$ is handled
without signature-based criteria checks and just completely reduced (again,
\emph{not} \sreduced{}!) w.r.t. $G \cup D$. Whereas termination and correctness
are ensured in this approach, performance really becomes a problem. Depending on
the input it often introduces an enormous number of zero reductions for elements
generated out of $D$. Moreover, as for \ffb{}, taking care of two different lists
of elements at the same time, is a bottleneck, too.

\subsubsection{Using Buchberger's chain criterion.}
In 2005, Ars defended his PhD thesis (\cite{arsPhd2005}). There a different
variant of \ff{} is presented which was later on denoted by \ffb{} in~\cite{egpF52011}.
In this variant a degree bound of the algorithm is computed with the
help of Buchberger's chain criterion. Besides the usual pair set $\pairset$ a
second set $\pairset^*$ is stored. Whereas $\pairset$ is still used for the
actual computations with \ff{} $\pairset^*$ has only the purpose to find a degree
bound $d$ for the algorithm. Whenever new S-pairs are computed the ones which are
not detected by Buchberger's chain criterion are added to $\pairset^*$.
After updating $\pairset^*$ $d$ is set to the highest degree of any S-pair in
$\pairset^*$. Once the degrees of all S-pairs in $\pairset$ exceed $d$ then by
Buchberger's chain criterion the polynomial part of the computed signature
\grobner{} basis up to degree $d$ is already a \grobner{} basis for the input
ideal.

\begin{characteristics}
\ffb{} uses linear algebra instead of polynomial \sreduction{}. We refer
to Section~\ref{sec:f4-f5} for further details on such an implementation of the
reduction process.
\end{characteristics}

\subsubsection{\ffp{} -- Keeping track of redundancy.}
In 2011, as a last termination dedicated variant
before Galkin's proof in~\cite{galkinTermination2012}, Eder, Gash and Perry present
\ffp{} in~\cite{egpF52011}. The main contribution is the distinction between
so-called ``GB-critical pairs'' and ``F5-critical pairs''. A GB-critical pair
corresponds to an S-pair $a\alpha-b\beta$ whereas $\hdp\alpha$ and $\hdp\beta$
are not already in the lead ideal of the current state of the computed
\grobner{} basis. An F5-critical pair is an S-pair which does not correspond to
a GB-critical pair, i.e. at least one generator is redundant. Whereas
GB-critical pairs are needed to be checked for the resulting \grobner{} basis,
F5-critical pairs seem to be superfluous, but this is not always the case:
Due to the rewritable signature
criterion it might happen that an GB-critical pair is discarded and instead a
corresponding F5-critical pair is \sreduced{} later on. Only since the
F5-critical pair is taken care of the algorithm's correctness is ensured. This
means that even if at a given degree $d$ there is no GB-critical pair left, one
might need to \sreduce{} corresponding F5-critical pairs in this degree.
The idea is now to store all, by \ff{}'s signature-based criteria discarded
GB-critical pairs in a second list $\pairset^*$, and keep all usual critical
pairs (resp. S-pairs) in $\pairset$. As long as the degree of the currently
handled elements in $\pairset$ is smaller or equal to the maximal degree of
elements in $\pairset^*$ the algorithm needs to carry on due to the above
discussion. Once the degree exceeds the maximal degree of an element in
$\pairset^*$ Buchberger's chain criterion is used: If all elements in $\pairset^*$
can be removed by it then the algorithm can terminate. This is due to the fact
that in $\pairset^*$ all for the resulting \grobner{} basis needed, but due to
rewritings discarded GB-critical pairs are stored. Once it is ensured (by
Buchberger's chain criterion) that those reduce to zero, we know that we already
reached a \grobner{} basis of the input.

\begin{characteristics}
\ffp{} starts checking $\pairset^*$ only once the degree of elements in
$\pairset$ exceeds the maximal degree of all
GB-critical pairs removed by \ff{}'s signature-based criteria, not before. Since
\ffb{} does not take care of the connection between F5-critical pairs and
GB-critical pairs, it has to check $\pairset^*$ in each step.

Moreover, \ffp{} stores and checks in $\pairset^*$ only GB-critical pairs that are
also discarded by \ff{}'s signature-based criteria. Only for such a GB-critical
pair a corresponding F5-critical pair might be necessary for the correctness of
the algorithm.
\end{characteristics}

\begin{variants}
For a generic system \ffb{} might find a lower degree bound than
\ffp{}. Moreover, note that both variants are able to terminate the algorithm
once a constant is found: Due to checking $\pairset^*$ by Buchberger's chain
criterion all other S-pairs are removed at this point.
\end{variants}

\section{Signature-based \grobner{} basis algorithms using $\rleqsb$}
\label{sec:efficient-implementations-sb}
Besides \ff{} all other known signature-based \grobner{} basis algorithms use
$\rleqsb$.\footnote{There are some minor exceptions we take care of in the
followng, too.} We can easily see that those instantiations of \rba{}, like
\gvw{} or \sba{}, mostly coincide and just differ in notation.

\subsection{Arri and Perry's work -- \ap{}}
\label{sec:arri-perry}
Aberto Arri released in 2009 a first preprint of his paper with John Perry,
~\cite{apF5CritRevised2011}. There the first mention of $\rleqsb$ can be found.
The paper reviews \ff{}'s criteria given in~\cite{fF52002Corrected} and
presents a signature-based \grobner{} basis algorithm depending on one
criterion only. There it is also called ``\ff{} criterion'' but it is equivalent
to choosing the canonical rewriter in signature $T$ from $\rewriters T$ w.r.t.
$\rleqsb$.

\begin{vocabulary}
The notions ``$\mathcal{S}$-reduction'' and
''$\mathcal{S}$-\grobner{} basis'' coincide with \sreduction{} and signature
\grobner{} basis, respectively.
\end{vocabulary}

\begin{characteristics}\
\begin{enumerate}
\item \ap{} implements \rba{} with $\rleqsb$ and can use any compatible module
monomial order $<$.
\item \ap{} is (nearly simultaneously with \ggv{}, see Section~\ref{sec:g2v}) the first signature
\grobner{} basis algorithm adding signatures of zero reductions directly to
$\hdsyz$.
\item \ap{}'s $\mathcal{S}$-reduction is (also nearly simultaneously with
\ggv{}'s implementation of \sreduction{},
see Section~\ref{sec:g2v}) the first one without checking the reducers
with the signature-based criteria (see also Lemma~\ref{lem:when-rb-reduces}). 
\end{enumerate}
\end{characteristics}

\subsection{The \trb{} algorithm -- top reductional basis}
\label{sec:huang}
Lei Huang was one of the first researchers comparing different signature-based
\grobner{} basis algorithms. In 2010 he presented his \trb{} algorithm
in~\cite{huang2010}, where the name comes from the wording ``top reductional
basis''.

\begin{vocabulary}
A \emph{top reductional prime element} coincides with the notion
``not regular top \sreducible{}'' given in Section~\ref{sec:sreduction} and a
\emph{top reductional basis} is just a signature \grobner{} basis.
\end{vocabulary}

The \trb{} algorithm does not focus on efficiency, but is a more general
algorithmic presentation of signature-based
compuations and included in \rba{}: In~\cite{huang2010} specialzations of \trb{}
are given that coincide with other known algorithms, like \trbff{},
\trbeff{}\footnote{See Section~\ref{sec:ext-f5}.} and \trbgvw{}.\footnote{See
Section~\ref{sec:gvw}.} Moreover, the most optimized variant \trbmj{} is
presented which coincides with \rba{} using $\rleqsb$ and $\pleqs$. Hereby
``MJ'' stands for ``minimal joint multiplied pair'' which corresponds to choose
at a given signature $T$ the canonical rewriter with minimal possible lead term,
that means using $\rleqsb$.

\subsection{The \gbgc{} algorithm -- generalized criteria}
\label{sec:sun-wang}
In 2011, Sun and Wang presented the \gbgc{} algorithm in~\cite{sw2011b}. This
algorithm is also a general one and included in \rba{}. This is, besides
\rba{}, the only signature-based \grobner{} basis algorithm that considers
different pair set orders $\pleq$. As already mentioned in
Remark~\ref{rem:generalized-spairorder-gvw} \gbgc{} is presented to use
partial orders on \basis{} as rewrite orders which is not efficient for
discarding useless S-pairs.

\begin{vocabulary} \
\begin{enumerate}
\item The ``generalized criterion'' the algorithm's name comes from can be directly
translated to choosing the canonical rewriter in signature $T$ in $\rewriters
T$ w.r.t. a given rewrite order $\rleq$.

\item Note that whereas we decide to call the element \emph{maximal} w.r.t. a rewrite
order the canonical rewriter in a given signature, in~\cite{sw2011b} the
\emph{minimum} is chosen. More particular, $\alpha \rleq \beta$ chosen there
coindices with $\frac1\alpha \rleqsb \frac1\beta$. So \gbgc{} still implements
\rba{} with $\rleqsb$, there is just a slight difference in notation.
\end{enumerate}
\end{vocabulary}

\begin{characteristics}
\gbgc{} implements the test for regular \sreduction{} considering the
coefficients of the signatures. Thus, a reduction of a term $t$ of $\proj\alpha$ with
some $b\beta$ such that $\sig\alpha = \sig{b\beta}$ is called
\emph{super-regular} if the coefficients of $\sig{\alpha}$ and $\sig{b\beta}$
differ. This definition comes initially from~\cite{ggvGGV2010}.
By the definitions in Section~\ref{sec:sreduction} we call this a singular top
\sreduction{}.
\end{characteristics}

The following lemma shows that there is no need to consider
coefficients of signatures at all, i.e. there cannot exist a super-regular
top reduction without a regular top \sreduction{}.

\begin{lemma}
\label{lem:super-regular}
In \rba{} there cannot exist a super-regular reduction of a term $t$
without a regular \sreduction{} of $t$.
\end{lemma}

\begin{proof}
See Fact~24 in~\cite{epSig2011}.\qed
\end{proof}

Thus \gbgc{} can be completely described by \rba{}. 

\begin{variants}\
\label{var:sw}
\begin{enumerate}
\item In~\cite{sw2011a} Sun and Wang use a signature \grobner{} basis resulting
from a computation of \rba{} to decide the ideal membership problem for $I$. This is
straightforward since the polynomial part of \basis{} is already a polynomial
\grobner{} basis. The other fact is that signatures can be used for the
representation problem of an element in $I$. Also this is straightforward, since
if you compute with the full module element $\alpha \in \module$ the signature
\grobner{} basis \basis{} stores already the full information. If one is using
\rba{} with sig-poly pairs~\cite{sw2011a} proposes just an algorithm to recover
the full module representation of elements in the \grobner{} basis.
\item\label{var:sw:solvable-algebra}
In 2012, Sun, Wang, Ma and Zhang have presented the \sgb{} algorithm
in~\cite{swmz2012}. \sgb{} is a signature-based \grobner{} basis algorithm for
computations in algebras of solvable type (for example, see~\cite{kw-solvable-type-1990})
like the Weyl algebra or quantum groups.
As a rewrite order $\rleqsb$ is used, which they denote as
``GVW-order''\footnote{More details on the changes in \gvw{}'s rewrite order
over the years can be found in Section~\ref{sec:gvw}.}. Besides adjusting the
polynomial arithmetic for the corresponding algebras no changes with respect
of the signature-based tools have to be made.
\end{enumerate}
\end{variants}

\subsection{The \ggv{} algorithm}
\label{sec:g2v}
The \ggv{} algorithm refers to Gao, Guan and Volny and was first presented in
2010. Its description is published in~\cite{ggvGGV2010}. A high-level
implementation in \singular{} is available under
\begin{center}
\url{http://www.math.clemson.edu/~sgao/code/g2v.sing}.
\end{center}
As mentioned already in
\characteristic~\ref{char:f5}~(\ref{char:f5:handling-zero-reductions}) \ggv{} was, after 
the description in~\cite{apF5CritRevised2011}, the first algorithm who used
non-Koszul syzygies directly in the syzygy criterion. The algorithm is described
in the vein of \ff{}'s description in~\cite{fF52002Corrected} and thus based on
using $\potl$ as module monomial order, which leads to an incremental
\grobner{} basis algorithm.

In~\cite{ggvGGV2010} the authors describe for the first time how \ggv{} and thus
signature-based \grobner{} basis algorithms in general can be used to compute a
\grobner{} basis for the syzygy module, by considering not only the signatures,
  but the full module representations. Still, this leads to the overhead of
  carrying out all computations in $\module$, too.
\begin{characteristics}
\label{char:g2v}
The two most important new features in \ggv{} compared to \ff{} as presented
in~\cite{fF52002Corrected} are
\begin{enumerate}
\item to take coefficients into account for signatures, and
\item to implement no real rewrite rule as done for \ff{}.
\end{enumerate}
\end{characteristics}

Whereas the first point enables so-called ``super-regular reductions'' that
might be not possible in \ff{} it turns out that this not the case: As already
mentioned in Section~\ref{sec:sun-wang} and proven in Fact~24
in~\cite{epSig2011} resp. Lemma~\ref{lem:super-regular},
whenever there exists a super-regular \sreduction{} then there exists also a
regular \sreduction{}. It follows that when it comes to signatures, coefficients
need not be taken into account at all.

In order to discuss the second difference, let us first introduce some
vocabulary.

\begin{vocabulary}
In~\cite{ggvGGV2010} notation is a bit different:
\begin{enumerate}
\item Instead of considering sig-poly pairs $(\sig\alpha,\proj\alpha)$, pairs $(u,v) \in \ring^2$ are
considered. This is possible since \ggv{} is presented only for $\potl$, thus an
incremental computation of $\basis$ is achieved. So any signature $\sig\alpha \in \module$ is always
of the type $\sig\alpha = s \mbasis k$ where $s \in \mon$ and $k$ is the
currently highest index of an element considered. So one can remove $\mbasis k$
without any problem, since all signatures share this module generator for the
current incremental step. So one gets a representation $(s,\proj\alpha)
\in \ring^2$ corresponding to $(u,v)$.
\item Next pairs $(u_1,v_1)$ and $(u_2,v_2)$ are
considered. Let $\lambda = \lcm\left(\hd{v_1},\hd{v_2}\right)$ and define $t_i
:= \frac{\lambda}{\hd{v_i}}$. Then
$\left(t_1\left(u_1,v_1\right),t_2\left(u_2,v_2\right)\right)$ is called the
\emph{J-pair} of $(u_1,v_1)$ and $(u_2,v_2)$. This corresponds to the notion of
our S-pairs. ``J'' denotes ``joint'', thus also parts of the J-pair have special
notation: In the above setting $t_i v_i$ are called \emph{J-polynomials} and
$t_i \hd{u_i}$ are \emph{J-signatures}.
\end{enumerate}
\end{vocabulary}

The other difference to \ff{} mentioned in \characteristic~\ref{char:g2v} is not so obvious at the first
look. Whenever a new $c\gamma$ may be added to $\pairset$ the authors state in the
pseudo code of the algorithm to ``store only one J-pair for each distinct
J-signature''. This
clearly is a rewritable signature criterion, but no explicit statement on which
element shall be kept and which shall be removed. Looking into the
\singular{} code of \ggv{} provided by the authors (see link above) one can see that in the
procedure \textsc{insertPairs} the newly generated element by $c\gamma$ is taken
whereas $a\alpha$, previously added to $\pairset$, is removed
if $\sig{a\alpha} = \sig{c\gamma}$.
Thus \ggv{}
implements $\rleqff$ as rewrite rule and not $\rleqsb$. The reason we keep
\ggv{} in this section is that it is the historical predecessor of
\gvw{} which uses (in its current version) $\rleqsb$ (see
below).

One difference left is the fact that in the provided code for \ggv{} only one
generator of an S-pair resp. J-pair is stored in $\pairset$. Thus an
S-pair reduction $a\alpha - b\beta$ might not take place, but instead there might
exist a better reducer $c\gamma$ instead of $b\beta$. This is an implicit statement
of the rewritable criterion on the second generator of the S-pair.

\begin{lemma}
After adding $a\alpha$ from the S-pair $a\alpha -b\beta$ to
$\pairset$ in \ggv{}, if there exists another regular top \sreducer{}
$c\gamma$ of $a\alpha$ which is not rewritable then $b\beta$ is rewritable.
\label{lem:ggv-rewritable}
\end{lemma}

\begin{proof}
If there exists another regular \sreducer{} $\gamma \in \basis$ which is,
at the moment $a\alpha$ is started to be regular \sreduced{}, not rewritable,
then instead of $a\alpha - b\beta$ the regular \sreduction{} $a\alpha - c\gamma$
takes place for some monomial $c$.
Since $\hdp{b\beta} = \hdp{c\gamma}$ and $\sig{b\beta}$, $\sig{c\gamma} <
\sig{a\alpha}$ three situations may happen:
\begin{enumerate}
\item If $\sig{c\gamma} = \sig{b\beta}$ then we can assume w.l.o.g. that $\gamma$ is the
canonical rewriter in signature $\sig{b\beta}$. Thus $b\beta$ is rewritable.
\item If $\sig{c\gamma} > \sig{b\beta}$ then the S-pair $c\gamma - b\beta$ has
been already \sreduced{} to an element $\delta \in \basis$. Since $\sig\delta =
\sig{c\gamma}$ and \ggv{} uses $\rleqff$ $\delta$ is the canonical rewriter in
signature $\sig{c\gamma}$ and thus $c\gamma$ is rewritable, a contradiction to
our assumption.
\item If $\sig{b\beta} > \sig{c\gamma}$ then the S-pair $b\beta-c\gamma$ has
been already reduced to an element $\delta \in \basis$, this time $\sig\delta =
\sig{b\beta}$. By the same argument as above $b\beta$ is rewritable.\qed
\end{enumerate}
\end{proof}

All in all, \ggv{} (as presented in~\cite{ggvGGV2010})implements \rba{}
with $\potl$ and $\rleqff$.

\begin{characteristics}
Lemma~\ref{lem:ggv-rewritable} might suggest that \ggv{} as presented
in~\cite{ggvGGV2010} makes use of the rewritability. Looking at the
\singular{} code provided it turns out that this is not the fact: In procedure
\textsc{findReductor} a reducer of the same index is searched for in \basis{}.
This search starts from the initially added element of current index to \basis{}.
Thus the first possible regular top \sreducer{} found might not be a ``better''
choice, where ``better'' is meant in terms of the rewrite order $\rleqff$.
\end{characteristics}

\subsection{The \gvw{} algorithm}
\label{sec:gvw}
Later in 2010, Gao, Volny and Wang published~\cite{gvwGVW2010} in which they
describe the algorithm \gvw{}.\footnote{Please note that there are different
versions of the \gvw{} paper which refer to~\cite{gvwGVW2010}
~\cite{gvwGVW2011} and~\cite{gvwGVW2013} respectively.} In this first presentation \gvw{}
generalizes \ggv{} in the sense that compatible module monomial orders can be
used freely instead of restricting to only $\potl$. Still, $\rleqff$ is used as
rewrite order in this version of \gvw{}. 

The work of Huang (see Section~\ref{sec:huang}) and Sun and Wang (see
Section~\ref{sec:sun-wang}) resulted in an algorithm denoted \gvwhs{} in Volny's
PhD thesis (\cite{volnyGVW2011}) in 2011. \gvwhs{} uses $\rleqsb$ as rewrite
order, besides this fact it coincides with \gvw{}.

In 2011 and later, the initial \gvw{} paper~\cite{gvwGVW2010} has been updated
to~\cite{gvwGVW2011}. There \gvw{} already uses $\rleqsb$ as rewrite order.
\begin{vocabulary} \
\label{voc:gvw-notations}
\begin{enumerate}
\item In the current state\footnote{March 2014} of the \gvw{} paper defining
the canonical rewriter w.r.t. $\rleqsb$ is called \emph{eventually super
top-reducible}
resp. \emph{covered by \basis{}}.
Moreover, note that a \emph{strong \grobner{} basis} in the setting of the \gvw{}
paper\footnote{Note that usually the term \emph{strong \grobner{} basis} denotes
special \grobner{} bases in polynomial rings over Euclidean domains like $\Z$.}
coincides with the union $\basis \cup \hdsyz$ here.

\item With the above definition of a strong \grobner{} bases, speaking of detecting
all useless S-pairs resp. J-pairs
the ``uselessness'' needs to be understood in terms of $\basis \cup \hdsyz$: Clearly,
a zero reduction of an S-pair is not useless in these terms since it leads to a
new syzygy that is not a multiple of an element of \hdsyz{} already. Thus one
needs to be careful and not mix this up with the uselessness of an S-pair w.r.t.
a usual polynomial \grobner{} basis resp. a signature \grobner{} basis \basis{}.
\end{enumerate}
\end{vocabulary}

Let us sum up the historic development of \gvw{}: \ggv{} implements
\rba{} with $\potl$ and $\rleqff$. \gvw{} is introduced as \ggv{} with the
option to use different compatible module monomial orders, but still
implementing $\rleqff$. Due to the work of Huang (\cite{huang2010}) and
Sun and Wang (\cite{sw2011b}), \gvw{} nowadays is understood as the algorthim
Volny denotes in his PhD thesis as \gvwhs{}: \rba{} with no restriction on the
compatible module monomial order and $\rleqsb$ as rewrite order.

Note that in~\cite{gvwGVW2013} the 2013 revision of \gvw{} a new step in
considering more principal syzygies is added. We discuss this in
Section~\ref{sec:buch}.

\subsection{The \sba{} algorithm}
\label{sec:sb}
Roune and Stillman presented the \sba{} algorithm at ISSAC'12
(\cite{rs-2012}). As Eder and Roune have already pointed out
in~\cite{erF5SB2013} \sba{} is \rba{} implemented with $\rleqsb$ as rewrite
order. The \rba{} algorithm presented here is only a slight generalization
of the one given in~\cite{erF5SB2013}, allowing different pair set orders and
reduce the syzygy criterion to a special case of the rewritable signature criterion.

\begin{remark}
Note that Roune and Stillman lay an emphasis on implementational aspects and data
structures. For this purpose an extended version of their ISSAC'12 paper is
available (\cite{rs-ext-2012}) in which different data representations
are compared and discussed extensively.
\end{remark}

\subsection{The \ssg{} algorithm}
\label{sec:ssg}
In 2012 Galkin described in~\cite{galkinSimple2012} the \ssg{} algorithm, where
``ssg'' stands for ``simple signature-based''. Comparing \ssg{} to
\rba{} both coincide once we choose $\rleqsb$ as rewrite order.
In~\cite{galkinSimple2012} Galkin defines a partial order $<_H$ on
sig-poly pairs ($H$ denotes the set of all sig-poly pairs) in the following way:
\begin{center}
$
\begin{array}{rcl}
\spp\alpha \;<_H \;\spp\beta &\Longleftrightarrow &\sig\beta \hdp\alpha \;<\;
\sig\alpha \hdp\beta.
\end{array}
$
\end{center}
 Moreover, syzygies are treated to be smaller
w.r.t. $<_H$ then any non-syzygy. From this it follows that $\spp\alpha <_H \spp\beta$
coincides with $\frac1\alpha \rleqsb \frac1\beta$. In part 4~(b) of the pseudo code of
the \ssg{} algorithm the rewritable signature criterion is then implemented
in the following way (adjusted to our notation):
\[
\pairset \gets \pairset \setminus \set{\alpha \in \pairset \mid \exists \beta
  \in \basis \text{ such that } \frac1\beta \rleqsb \frac1\alpha \text{ and
  }\sig\beta \mid \sig\alpha}
\]
With the above described connection between $<_H$ and $\rleqsb$ one directly
sees that this is just \rba{}'s rewrite procedure using $\rleqsb$.

\section{Using Buchberger's criteria in signature-based \grobner{} basis algorithms}
\label{sec:buch}
A natural question coming to one's mind is how \rba{}'s rewrite criterion is related
to Buchberger's Product and Chain
criterion,~\cite{bGroebner1965,kollBuchberger1978,bGroebnerCriterion1979}. 
Both predict useless computations
in advance, but how do both attempts relate to each other? Does one include the other,
or are there cases where one side is not able to cover the other side completely?
It turns out that one can easily combine both classes of criteria, even more one
can show that the rewrite criterion includes Buchberger's criteria ``most of the
time''. It is more or less a question about how much overhead one wants to add
to \rba{} in order to track principal syzygies on the go. For a detailed
discussion on the algebraic nature of this relation we refer
to~\cite{eder-predicting-zero-reductions-2014}.

In 2008~\cite{gashPhD2008} Gash presented a first discussion on using Buchberger's Product
and Chain criterion in signature-based algorithms. Moreover, Gerdt and Hashemi
presented an improved variant of \ggv{} in~\cite{gerdtHashemiG2V} making use of
these criteria. In 2013, Gao, Volny and Wang presented a revised version of
\gvw{} in~\cite{gvwGVW2013} that adds another step to store more principal
syzygies. We shortly cover these variants in the following.

\subsection{Buchberger's criteria}
Let us give a short review of Buchberger's Product and Chain criterion:

\begin{lemma}[Product criterion~\cite{bGroebner1965,bGroebnerCriterion1979}]
Let $f,g \in \ring$ with $\lcm\left(\hd f,\hd g\right) = \hd f \hd g$.
Then $\spoly f g$ reduces to zero w.r.t.
$\left\{f,g\right\}$.
\label{lem:prod-crit}
\end{lemma}
In the above situation we also say that the S-polynomial \emph{$\spoly f g$ fulfills the
Product criterion}.

\begin{lemma}[Chain criterion~\cite{kollBuchberger1978,bGroebnerCriterion1979}]
Let $f,g,h \in \ring$, and let $G \subset \ring$ be a finite subset. If it holds
that $\hd h \mid \lcm\left(\hd f,
\hd g\right)$, and if $\spoly f h$ and $\spoly h g$ have a standard representation
w.r.t. $G$ resp., then $\spoly f g$ has a standard representation w.r.t. $G$.
\label{lem:chain-crit}
\end{lemma}

The question is now how do those criteria relate to the rewrite criterion in
signature-based \grobner{} basis algorithms. Gash gave a first proof that the
Product criterion can be used in a signature-based algorithm without any
problem due to the fact that the reductions w.r.t. $\left\{\alpha,\beta\right\}$ are
regular \sreductions{} when considering $\proj\alpha = f$ and $\proj\beta=g$ in
Lemma~\ref{lem:prod-crit}.
Furthermore Gash proved that a version of Lemma~\ref{lem:chain-crit} where the
signatures corresponding to $f$, $g$, and $h$ are restricted can be used in
\rba{}.

In~2014 Eder presented in~\cite{eder-predicting-zero-reductions-2014} a proof
that the Chain criterion is completely included in the rewrite criterion of
\rba{}, without any further restrictions. Moreover, the problem of being not
able to predict all zero reductions
that are found by the Product criterion is explained there in detail. A
small counterexample for \rba{} using $\schl$ is given. Furthermore, it is still
an open question whether \rba{} using $\potl$ completely covers the Product
criterion. So it seems that the relation between Buchberger's criteria and
signature-based ones are depending on the chosen module monomial order.

Also the question of using Buchberger's criteria in \rba{} is answered, there
are two possible implementations of a combination of the criteria:
The first one explicitly, the second one more subtle.

\subsection{\gggv{} -- a Gebauer-M\"oller-like \ggv{}}
In~\cite{gerdtHashemiG2V} Gerdt and Hashemi present \gggv{}, a variant of
\ggv{}. In their variant they add $3$ new conditions to be checked which
coincides with the three steps in Gebauer-M\"oller's implementation of
Buchberger's algorithm, see~\cite{gmInstallation1988}. Moreover, they show that
adding these conditions can be done without corrupting signatures, thus
\gggv{} is still a correct and terminating signature-based \grobner{} basis
algorithm with the rewrite criterion implemented as usual (see Theorem~4.1
in~\cite{gerdtHashemiG2V}). 

Note that due to the results in~\cite{eder-predicting-zero-reductions-2014} it
is not needed to check the Chain criterion explicitly since it is completely
covered by the rewrite criterion. 

\subsection{\gvw{}'s 2013 revision}
In 2013 Gao, Volny and Wang revised \gvw{} again, with the current status being
presented in~\cite{gvwGVW2013}. In this version of \gvw{} a new step is
inserted, namely an additional computation of principal syzygies even though the
regular \sreduced{} $\gamma$ might not fulfill $\proj\gamma = 0$.
In~\cite{gvwGVW2013} this is Step $4b~(b1)$ of Figure~3.1. In our
notation this would be after Line~\ref{alg:rba:addsingularelement} of
Algorithm~\ref{alg:rba}. Even though $\proj\gamma$ is not zero, all new possible
principal syzygies are generated and added to \hdsyz{}. Afterwards
\hdsyz{} is interreduced. This has two impacts:
\begin{enumerate}
\item On the one hand new syzygies might be added such that more useless
computations can be predicted and removed in advance. Clearly, with this attempt
also all useless computations predicted by Buchberger's Product criterion
(representing exactly some of these principal syzygies) are detected, too.
\item On the other hand a lot of these new principal syzygies added to
\hdsyz{} may have signatures that are just multiples of signatures already
available in \hdsyz{}. Thus the overhead might be rather high compared to the
benefits.
\end{enumerate}

Clearly, \gvw{}'s attempt adding all possible principal syzygies does not give
more information to the rewrite criterion than testing for Buchberger's Product
criterion directly and adding the corresponding signature to \hdsyz{}
accordingly. In terms of efficient implementations it seems that checking the
Product criterion explicitly introduces less overhead than generating new
principal syzygies whenever a new element $\gamma$ is added to \basis{}.

Whereas the first variant adds $1$ syzygy resp. signature to \hdsyz{} when it is
needed, the second one always tries to recover all such relations and afterwards
checks, which ones can be removed from \hdsyz{} being just multiples of each
other.

\section{\sreductions{} using linear algebra}
\label{sec:f4-f5}
As already pointed out in Section~\ref{sec:efficient-implementations-f5} \ff{}
is presented in~\cite{fF52002Corrected} in the vein of implemeting the
\sreduction{} process using linear algebra. \mff{}, presented in
Section~\ref{sec:matrix-f5}, is efficient once the system of polynomial
equations is dense. Clearly, this is not always the case, and thus, selecting
S-pairs to be reduced is more convenient compared to building full Macaulay
matrices at a given degree $d$. The first presentation of such an S-pair
generating algorithm using linear algebra for reduction purposes is the
\ffo{} algorithm (\cite{fF41999}).  Here we present a variant of \ffo{} that
uses signature-based criteria to detect reductions to zero resp. rows reducing
to zero in advance. This leads to smaller changes in the implementation of some
subalgorithms of \ffo{} corresponding to the switch from usual polynomial
reduction to \sreduction{}. Albrecht and Perry describe a possible
implementation of this, called \ffs{} in~\cite{apF452010}.

\begin{characteristics}
Note that the variant \ffs{} described in~\cite{apF452010} differs from \ff{}
by more than replacing the polynomial \sreduction{} by linear algebra:
\begin{enumerate}
\item Instead of incrementally computing the \grobner{} basis for $\ideal{\gen
  f}$ computations are done by increasing degrees: Whereas \ff{} proceeds by
  index first, \ffs{} prefers the degree of the polynomials over the index. This
  corresponds to switching from $\potl$ to $\dpotl$.
\item Instead of sorting the generators by decreasing index, they are ordered by
increasing index (see also Footnote~\ref{fn:potl}).
\item Due to the switch from $\potl$ to $\dpotl$ the rewrite rules
$\textsc{Rule}_i$ might not be sorted by increasing degree when only appending
new rules as done in~\cite{fF52002Corrected}. Thus the subalgorithm \textsc{Add
  Rule} takes care of sorting $\textsc{Rule}_i$ by increasing degree. Note that
  as mentioned in Remark~\ref{rem:f5-rewrite-order} this still need not ensure a
  sorting of $\textsc{Rule}_i$ by increasing signature.
\end{enumerate}
\end{characteristics}

Giving a full description of the ideas behind the \ffo{} algorithm out of scope
of this survey, we refer the readers interested to~\cite{fF41999}. Here we explain in
detail an \ffo{}-style variant of \rba{}. With this any known implementation of
signature-based \grobner{} basis algorithms described in
sections~\ref{sec:efficient-implementations-f5}
and~\ref{sec:efficient-implementations-sb} can be modified in the same way to
use linear algebra for reduction purposes.


\begin{algorithm}
\begin{algorithmic}[1]
\Require Ideal $I=\langle \gen f \rangle \subset \ring$, monomial
order $\leq$ on $\ring$ and a compatible extension on $\module$, total order
$\pleq$ on the pairset $\pairset$ of S-pairs, a rewrite order
$\rleq$ on $\basis \cup \hdsyz$
\Ensure Rewrite basis \basis{} for $I$, \grobner{} basis $\hdsyz$ for $\syz{\gen f}$
\State $\basis{}\gets\emptyset$, $\hdsyz \gets \emptyset$, $d \gets 0$
\State $\pairset\gets\set{\mbasis 1,\ldots,\mbasis m}$
\State $\hdsyz\gets \set{f_i \mbasis j - f_j \mbasis i \mid 1\leq i < j \leq
  m}\subseteq \module$\label{alg:ffrba:initialsyz}
\While{$\pairset\neq\emptyset$}
  \State $d\gets d+1$
  \State $\pairset_d\gets \select\pairset$\label{alg:ffrba:choosespair}
  \State $\pairset\gets \pairset\setminus\pairset_d$
  \State $\mathcal L_d \gets \left\{a\alpha,b\beta \mid a\alpha - b\beta \in
  \pairset_d\right\}$\label{alg:ffrba:ld}
  \State $\mathcal L_d \gets \symbolicpreprocessing{}(\mathcal L_d,\basis)$
  \State $M_d \gets \text{matrix gen. by rows corr. to $\proj{a\alpha}$ for
    $a\alpha \in \mathcal L_d$ (sorted by signatures)}$
  \State $N_d \gets $ row echelon form of $M_d$ computed without row swapping
  \label{alg:ffrba:row-echelon-form}
  \State $\basis_d \gets \left\{\gamma \mid \proj{\gamma} \text{ corresponding
    to a row in } N_d \right\}$
  \State $\basis^+_d \gets \left\{\gamma \in \basis_d \mid \hdp\gamma \neq
  \hdp{a\alpha} \text{ for } a\alpha \in \mathcal L_d, \sig\gamma =
  \sig{a\alpha}\right\}$ \label{alg:ffrba:new-elements}
  \While{$\basis^+_d \neq \emptyset$} \label{alg:ffrba:choose}
    \State $\gamma \gets \min_<\basis^+_d$
    \State $\basis^+_d \gets \basis^+_d \setminus \set \gamma$
    \If {$\proj{\gamma}=0$}
      \State $\hdsyz\gets \hdsyz+\set{\gamma}$\label{alg:ffrba:addsyz}
    \Else \label{alg:ffrba:addsingularelement}
      \State $\pairset\gets \pairset\cup\setBuilder
        {\spair\alpha\gamma}
        {\alpha\in\basis{}\text{ and $\spair\alpha\gamma$ is
                                regular}}$\label{alg:ffrba:pairs}
      \State $\basis{}\gets\basis\cup\set\gamma$\label{alg:ffrba:basis}
    \EndIf
  \EndWhile
\EndWhile
\State \textbf{return} $(\basis{}, \hdsyz)$
\end{algorithmic}
\caption{Rewrite basis algorithm using linear algebra \ffrba{}.}
\label{alg:ffrba}
\end{algorithm}

The main difference between \rba{} and \ffrba{} is the usage of linear algebra
for the reduction process in the later one. Instead of fulfilling
\sreductions{} on each new S-pair, \ffrba{} implements a variant of \ffo{}'s reduction
process: In Line~\ref{alg:ffrba:choosespair} we no longer need to choose only
one single S-pair as done in \rba{} but a subset of $\pairset$ can be taken at
once. The generators of those symbolic S-pairs are then stored in $\mathcal L_d$
(Line~\ref{alg:ffrba:ld}). Subalgorithm \symbolicpreprocessing{} is then
precomputing all possible reducers of the elements in $\mathcal L_d$. Due to the
additional structure of the signatures one has to change this part slightly
compared to an implementation in the \ffo{} Algorithm. This is discussed in
\characteristic{}~\ref{char:differences-f4-f5}. After all elements needed to
execute in the $d$th reduction step of the algorithm are stored in $\mathcal L_d$ a
corresponding matrix $M_d$ w.r.t. $<$ is constructed: The rows of $M_d$
represent the elements $\proj{a\alpha}$ for $a\alpha \in \mathcal L_d$, the
columns represent the corresponding monomials in $\ring$ ordered w.r.t. $<$. As in the
\mff{} Algorithm each row has a signature, namely $\sig{a\alpha}$. As mentioned
already in Section~\ref{sec:sreduction} \sreductions{} on the polynomial side correspond
to fixing an order on the rows in $M_d$. Thus the computation
of the row echelon form of $M_d$ in Line~\ref{alg:ffrba:row-echelon-form} is
done without row swapping.

\begin{variants}\
\begin{enumerate}
\item As already mentioned in \characteristic~\ref{char:gsbend}~(\ref{char:gsbend:ssp}) for
an efficient implementation one would use $\spp\alpha$ instead of $\alpha$ in
\ffrba{}. Algorithm~\ref{alg:ffrba} as presented here works with full module
elements, that means when computing the row echelon form one needs to keep track
of all corresponding module operations in $a\alpha$ for each such row in $M_d$.
Focussing on the computation of a \grobner{} basis and using only $\spp{a\alpha}$
this overhead disappears completely due to the fact that row swappings are not
allowed and thus the signatures corresponding to rows in $M_d$ do not change
throughout the whole process.
\item In \ffo{} all polynomials corresponding to rows in $N_d$ are added to the
\grobner{} basis which lead term is not already included in the lead ideal.
Signatures lead to \sreductions{}. We have seen already in
Section~\ref{sec:rewrite-bases} that elements $\gamma$ might be added to
\basis{} even so there exists some $\alpha \in \basis$ such that $\hdp\alpha
\mid \hdp\gamma$. Thus we cannot discard those elements. In
Line~\ref{alg:ffrba:new-elements} we choose the elements $\gamma$ that need to be added
to \basis{} (or \hdsyz{} if $\proj\gamma = 0$): If the polynomial lead term
corresponding to a signature $\sig{a\alpha}$ has not changed during the
computation of the row echelon form $N_d$ of $M_d$ then we do not need to add
this element to \basis{}. In any other case, we do so.
\end{enumerate}
\end{variants}

In Algorithm~\ref{alg:symbolic-preprocessing} we state the pseudo code of a
signature respecting variant of \textbf{Symbolic Preprocessing} from~\cite{fF41999}.


\begin{algorithm}
\begin{algorithmic}[1]
\Require a finite subset $\mathcal{U}$ of $\module$, a finite subset $\basis$ of
$\module$
\Ensure a finite subset $\mathcal{U}$ of $\module$

\State $D \gets \left\{\hdp\beta \mid \beta \in \mathcal{U}\right\}$
\State $C \gets \left\{\text{monomials of }\proj\beta \mid \beta \in
\mathcal{U}\right\}$
\While{$C \neq D$}
  \State $m \gets \max_< \left(C \setminus D\right)$
  \State $D \gets D \cup \{m\}$
  \State $\mathcal{V} \gets \emptyset$
  \For {$\gamma \in \basis$}\label{alg:symbolic-preprocessing:for-start}
    \If {$\exists c\in\mon$ such that $m = \hdp{c\gamma}$ and
      $\notrewritable{c\gamma}$}\label{alg:symbolic-preprocessing:choice}
      \State $\mathcal{V} \gets \mathcal{V} \cup \left\{c\gamma\right\}$
  \EndIf
  \EndFor\label{alg:symbolic-preprocessing:for-stop}
  \State $e\varepsilon \gets$ element of minimal signature in $\mathcal{V}$
  \label{alg:symbolic-preprocessing:min-sig}
  \State $\mathcal{U} \gets \mathcal{U} \cup
  \left\{e\varepsilon\right\}$
  \State $C \gets C \cup \left\{\text{monomials of } \proj{e\varepsilon}\right\}$
\EndWhile
\State \textbf{return} $\mathcal{U}$
\end{algorithmic}
\caption{\symbolicpreprocessing{} respecting signatures.}
\label{alg:symbolic-preprocessing}
\end{algorithm}

\begin{characteristics}
\label{char:differences-f4-f5}
Algorithm~\ref{alg:symbolic-preprocessing} has undergone several small
changes compared to the version presented in~\cite{fF41999}:
\begin{enumerate}
\item From lines~\ref{alg:symbolic-preprocessing:for-start}
to~\ref{alg:symbolic-preprocessing:for-stop} the algorithm loops over all
elements $\gamma \in \basis$ searching for a possible, not rewritable reducer of the
monomial $m$. If successful we add the multiplied reducer to an intermediate
set $\mathcal{V}$. Instead to the original \symbolicpreprocessing{} algorithm
we do not stop after finding a first possible reducer her. The idea is to take
in Line~\ref{alg:symbolic-preprocessing:min-sig} the single reducer
$e\varepsilon$ of minimal signature from $\mathcal V$. The smaller the signature
of $e\varepsilon$ the bigger is the probbility that $e\varepsilon$ might be an
allowed reducer of some other row in $M_d$ for term $\hdp{e\varepsilon}$.
\item Let $a\alpha \in \mathcal{U}$ such that $m$ is a monomial in
$\proj{a\alpha}$ and $a\alpha$ is of maximal signature for all such elements in
$\mathcal{U}$. Note that it is still possible that $\sig{e\varepsilon} > \sig{a\alpha}$.
If $m = \hdp{a\alpha}$ this corresponds to the creation of a new S-pair
$\spair{\varepsilon}{\alpha} = e\varepsilon - a\alpha$. Note that in Algorithm~\ref{alg:rba}
the generation of this S-pair is postponed: There only regular \sreductions{} are
computed in Line~\ref{alg:rba:regular}, $\spair\epsilon\alpha$ is generated in
Line~\ref{alg:rba:pairs} first. Moreover, note that there does not exist another
reducer $e'\varepsilon'$ such that $m-\proj{e'\varepsilon'}$ corresponds to a
regular \sreduction{} since $e\varepsilon$ is chosen to be minimal w.r.t. its
signature.
\item\label{char:differences-f4-f5:unique-reducer}
Due to Lines~\ref{alg:symbolic-preprocessing:choice}
and~\ref{alg:symbolic-preprocessing:min-sig} the reducer for $m$ is
uniquely defined. This choice depends on the chosen rewrite order $\rleq$ as
well as the module monomial order $<$. Furthermore, one can exchange
Line~\ref{alg:symbolic-preprocessing:min-sig} by another choice, for example,
the element in $\mathcal V$ which is most sparse or the one which has the lowest
coefficient bound. Thus using the ideas of~\cite{bSlimGB2005} is possible.
Note that such changes may put a penalty on the efficiency of
the algorithm due to introducing many more S-pairs as the chosen reducer might not
be of minimal possible signature. Still, correctness and
termination are not affected.
\end{enumerate}
\end{characteristics}

An optimization of \ffo{} given in~\cite{fF41999} is the usage of the
\simplify{} subalgorithm: \simplify{} tries to exchange generators of
S-polynomials and found reducers in \symbolicpreprocessing{} with ``better
ones'': Polynomial products $u f \in \ring$ are tried to be
exchanged by elements $\frac u t g$ where $\hd g = \hd{tf}$ for a divisor
$t$ of $u$. In~\cite{fF41999} the normal strategy for choosing critical
pairs is used, that means, computations are done by increasinig polynomial
degree and thus $g$ can be found in a previously constructed matrix $M_d$in degree
$d := \deg(tf)$. $g$ might not be added to the intermediate \grobner{} basis as
$\hd f \mid \hd g$. Still, $g$ might be further reduced than $f$ and thus one
can prevent the algorithm in degree $\deg(uf)$ from redoing reduction steps
already performed in degree $d$ by exchanging $uf$ by $\frac u t g$.

Due to the signatures this is not so easy in our setting: What if a
simplification of $a\alpha$ by $\frac a b \beta$ leads to $\sig{\frac a b \beta}
> \sig{a\alpha}$? In
\characteristic~\ref{char:differences-f4-f5}~(\ref{char:differences-f4-f5:unique-reducer})
we have seen that the rewrite order $\rleq$ as well as the module monomial order
$<$ uniquely define the reducer of a monomial $m$. This definition incorporates
the ideas of \simplify{} in the signature-based world.

\begin{variants}
Let us finish with the following notes on the idea of simplification in \ffo{}-like
signature-based \grobner{} basis algorithms.
\begin{enumerate}
\item Besides the way \simplify{} is presented in~\cite{fF41999} other
ways of choosing a better reducer are possible. In~\cite{bSlimGB2005}
Brickenstein gives various choices. In the signature-based world this is
reflected by the different implementations of the rewrite order $\rleq$ and the
module monomial order $<$.
\item If we assume $\potl$ as module monomial order then we can make use of the
incremental behaviour of the computations: Assume that we are computing the
\grobner{} basis for $\ideal{f_1,\ldots,f_i}$ having already computed one for
$\ideal{f_1,\ldots,f_{i-1}}$, say $\proj{\basis_{i-1}}$. Now we can implement \ffo{}'s \simplify{}
routine without any changes for elements in $\proj{\basis_{i-1}}$: All reducers
from $\proj{\basis_{i-1}}$
have a lower signature due to
its index $<i$. Thus, as already described in~\cite{fF52002Corrected} we do not need to
check them by any criterion. Moreover, simplifying any such reducer by another
element from a computation during a previous iteration step the corresponding
signature still has index $<i$. Furthermore, assuming \ffc{} (see
Section~\ref{sec:f5c}) we can assume $\proj{\basis_{i-1}}$ to be reduced to $B_{i-1}$ which
optimizes the choice of reducers even more. Since adding
\simplify{} to \ffrba{} respectively \symbolicpreprocessing{} is straight
forward in this situation we do not give explicit pseudo code for this.
\item Moreover, exchanging $\basis^+_d$ with $\basis_d$ in the argument of the
\textbf{while} loop in Line~\ref{alg:ffrba:choose} of Algorithm~\ref{alg:ffrba}
one can trigger a \simplify{}-like process: Since all non-zero elements are
added to $\basis$, only the S-pairs generated by the best reduced elements
are not rewritten. Of course this feature is paid dearly for by generating
all the useless S-pairs in first place due to the redundant elements in \basis{}.
\end{enumerate}
\end{variants}

\section{Experimental results}
\label{sec:available-implementations}
In the following we present experimental results of \grobner{} basis benchmarks
and random systems. All systems are computed over a field of characteristic $32003$, with graded
reverse lexicographical monomial order. The random systems
are defined by $3$ parameters on the input generators:
\begin{center}
$
\begin{array}{rl}
\hrandom\left(\text{\# vars=\# equations},\text{minimal degree},\text{maximal
    degree}\right) & \text{(homogeneous)}\\
\random\left(\text{\# vars=\# equations},\text{minimal degree},\text{maximal
    degree}\right) & \text{(affine)}\\
\end{array}
$
\end{center}

Polynomials are random dense in the corresponding number of variables. The
systems are available under
\begin{center}
\url{https://github.com/ederc/singular-benchmarks}.
\end{center}

The implementation is done in the computer algebra system \singular{}
(\cite{singular400}). Signature-based \grobner{} basis
algorithms are officially available in \singular{} starting version
4-0-0\footnote{All examples in this survey are computed with the commit
5d25c42ce5a7cfe24a13632fa0f7cc6b85961ccb available under
\url{https://github.com/Singular/Sources}.}.

We do not add timings since we do not want to start a fastest implementation
contest. We are interested in presenting the size of the basis, the number of
syzygies found and used, the number of reductions as well as the complete number of
operations, that means, multiplications. Those are the numbers that are unique to
the different variants of signature-based \grobner{} basis implementations. Any
real new variant might compute numbers different to those presented in the following.

All algorithms using $\potl$ are implemented with the ideas of \ffpr{} resp. \ffc{}, that
means, inbetween the incremental steps of computing the signature
\grobner{} basis the intermediate bases are reduced and new signatures are
generated (see Section~\ref{sec:f5c}). This leads to three facts:
\begin{enumerate}
\item The number of elements in \hdsyz{} increases. The number is usually much
higher than the ones for the computation w.r.t. $\schl$ or $\dpotl$.
\item The difference in the size of the resulting signature \grobner{} basis
between using $\rleqff$ and $\rleqsb$ diminishes: Since both computations are
starting the last iteration step with the same number of elements (using the reduced
\grobner{} basis) only differences during the last incremental step are
captured. Thus mostly the differences in the size of \basis{} are much bigger for the
computations w.r.t. $\schl$.
\item When counting the number of reduction steps as well as the number of
overall operations, one needs to distinguish between the
\sreductions{} done by \rba{} and the number of usual reductions done inbetween
two incremental steps when interreducing the intermediate \grobner{} bases.
In the tables below we give for computations w.r.t. $\potl$ the values for
\sreductions{} as well as the values for all reductions including the
interreduction steps. Clearly, for $\schl$ and $\dpotl$ there is no
interreduction due to non-incremental execution.
\end{enumerate}

\begin{remark}\
\begin{enumerate}
\item Note that the behaviour for computations w.r.t. $\dtopl$ is not optimal.
Choosing this module monomial order leads to very long running times in most of the
cases. Thus we do not include the corresponding results.
\item The differences when adding Buchberger's Product and Chain criterion to
\rba{} as described in Section~\ref{sec:buch} are subtle and do not change the
overall behaviour of \rba{}. In order not to overload our tables with even more
variants of \rba{} we do not cover those differences here. With the information
and discussions given in this survey the reader is able to understand the differences to
experimental results given in~\cite{gerdtHashemiG2V,gvwGVW2013,
eder-predicting-zero-reductions-2014} which focus on this setting.  
\end{enumerate}
\end{remark}

We have to distinguish different ways of computation in the following:
\begin{enumerate}
\item \rba{} can fulfill only top \sreductions{} or full \sreductions{}
(including tail \sreductions{}).
\item The examples can be affine or homogeneous.
\end{enumerate}

Note that the differences between only top \sreductions{} and full
\sreductions{} are only found in the number of \sreductions{} and the number of
operations. Therefore the other tables do not include a differentiation between
those two. Next we present the results for homogeneous respectively affine
input. These values have two different ways of being used:
\begin{enumerate}
\item The reader new to signature-based \grobner{} basis algorithms can get a
feeling for the behaviour of this kind of algorithm. One can easily
compare the results presented here with the outcome of \singular{}'s
Gebauer-M\"oller implementation.
\item For researchers trying to improve signature-based \grobner{} basis
algorithms those numbers are good reference points in order to see what kind of
optimizations are achieved.
\end{enumerate}

The corresponding figures after the corresponding tables give a graphical
overview of the behaviour of the different variants for the random systems
w.r.t. increasing number of generators. 

Moreover note that we stopped the computations for affine random systems at $12$
resp. for homogeneous random systems at $13$ generators for variants using only
top \sreductions{} since running time was too long. For full \sreductions{}
we could go on until $14$ generators.
\vfill

\subsection{Experimental results for homogeneous systems}

\setlength{\tabcolsep}{3pt}
\renewcommand{\arraystretch}{1.2}

\begin{table}[H]
\begin{minipage}{.50\linewidth}
\resizebox{0.90\textwidth}{!}{
\begin{tabu} to \linewidth {@{}c|S[table-format=5.0]
  S[table-format=5.0]|S[table-format=5.0]
  S[table-format=5.0]|S[table-format=5.0]
  S[table-format=5.0] @{}}
  \multirow{2}{*}{\bf Benchmark} & \multicolumn{2}{c|}{$\potl$} &
  \multicolumn{2}{c|}{$\schl$} & \multicolumn{2}{c}{$\dpotl$}\\ 
  & \multicolumn{1}{c}{$\rleqff$} & \multicolumn{1}{c|}{$\rleqsb$}
  & \multicolumn{1}{c}{$\rleqff$} & \multicolumn{1}{c|}{$\rleqsb$}
  & \multicolumn{1}{c}{$\rleqff$} & \multicolumn{1}{c}{$\rleqsb$}\\ 
\hline
cyclic-7 & 36 & 36 & 145 & 145 & 36 & 36\\ 
cyclic-8 & 244 & 244 & 672 & 672 & 244 & 244\\ 
eco-10 & 247 & 247 & 367 & 367 & 247 & 247\\ 
eco-11 & 502 & 502 & 749 & 749 & 502 & 502\\ 
f-633 & 3 & 3 & 9 & 9 & 3 & 3\\ 
f-744 & 190 & 190 & 259 & 259 & 190 & 190\\ 
katsura-11 & 0 & 0 & 353 & 353 & 0 & 0\\ 
katsura-12 & 0 & 0 & 640 & 640 & 0 & 0\\ 
noon-8 & 0 & 0 & 294 & 294 & 0 & 0\\ 
noon-9 & 0 & 0 & 682 & 682 & 0 & 0\\ 
$\hrandom(6,2,2)$ & 0 & 0 & 26 & 26 & 0 & 0\\ 
$\hrandom(7,2,2)$ & 0 & 0 & 49 & 49 & 0 & 0\\ 
$\hrandom(7,2,4)$ & 0 & 0 & 80 & 80 & 0 & 0\\ 
$\hrandom(7,2,6)$ & 0 & 0 & 635 & 635 & 0 & 0\\ 
$\hrandom(8,2,2)$ & 0 & 0 & 102 & 102 & 0 & 0\\ 
$\hrandom(8,2,4)$ & 0 & 0 & 345 & 345 & 0 & 0\\ 
$\hrandom(9,2,2)$ & 0 & 0 & 181 & 181 & 0 & 0\\ 
$\hrandom(10,2,2)$ & 0 & 0 & 339 & 339 & 0 & 0\\ 
$\hrandom(11,2,2)$ & 0 & 0 & 590 & 590 & 0 & 0\\ 
$\hrandom(12,2,2)$ & 0 & 0 & 1083 & 1083 & 0 & 0\\ 
$\hrandom(13,2,2)$ & 0 & 0 & 1867 & 1867 & 0 & 0\\ 
$\hrandom(14,2,2)$ & 0 & 0 & 3403 & 3403 & 0 & 0\\ 
\hline
\end{tabu}
}
\vspace*{3mm}
\caption{\# zero reductions (homogeneous)}
\end{minipage}
\qquad
\begin{minipage}{.50\linewidth}
\resizebox{\textwidth}{!}{
\begin{tabu} to \linewidth {@{}c|S[table-format=6.0]
  S[table-format=6.0]|S[table-format=6.0]
  S[table-format=6.0]|S[table-format=6.0]
  S[table-format=6.0] @{}}
  \multirow{2}{*}{\bf Benchmark} & \multicolumn{2}{c|}{$\potl$} &
  \multicolumn{2}{c|}{$\schl$} & \multicolumn{2}{c}{$\dpotl$}\\ 
  & \multicolumn{1}{c}{$\rleqff$} & \multicolumn{1}{c|}{$\rleqsb$}
  & \multicolumn{1}{c}{$\rleqff$} & \multicolumn{1}{c|}{$\rleqsb$}
  & \multicolumn{1}{c}{$\rleqff$} & \multicolumn{1}{c}{$\rleqsb$}\\ 
\hline
cyclic-7 & 758 & 658 & 871 & 848 & 949 & 751\\ 
cyclic-8 & 3402 & 2614 & 4074 & 3658 & 5534 & 3884\\ 
eco-10 & 677 & 508 & 541 & 478 & 934 & 567\\ 
eco-11 & 1423 & 1016 & 1092 & 965 & 2372 & 1168\\ 
f-633 & 60 & 58 & 61 & 59 & 61 & 57\\ 
f-744 & 616 & 465 & 377 & 348 & 745 & 573\\ 
katsura-11 & 700 & 700 & 553 & 553 & 2188 & 2161\\ 
katsura-12 & 1383 & 1384 & 1076 & 1076 & 6020 & 6020\\ 
noon-8 & 1384 & 1390 & 1384 & 1389 & 1384 & 1389\\ 
noon-9 & 3743 & 3750 & 3743 & 3749 & 3743 & 3749\\ 
$\hrandom(6,2,2)$ & 52 & 52 & 39 & 39 & 62 & 62\\ 
$\hrandom(7,2,2)$ & 101 & 101 & 67 & 67 & 124 & 124\\ 
$\hrandom(7,2,4)$ & 333 & 333 & 249 & 249 & 349 & 349\\ 
$\hrandom(7,2,6)$ & 4066 & 4066 & 2928 & 2928 & 4247 & 4247\\ 
$\hrandom(8,2,2)$ & 185 & 185 & 128 & 128 & 242 & 242\\ 
$\hrandom(8,2,4)$ & 1397 & 1397 & 997 & 997 & 1507 & 1507\\ 
$\hrandom(9,2,2)$ & 365 & 365 & 223 & 223 & 479 & 479\\ 
$\hrandom(10,2,2)$ & 676 & 676 & 426 & 426 & 932 & 932\\ 
$\hrandom(11,2,2)$ & 1326 & 1326 & 767 & 767 & 1832 & 1832\\ 
$\hrandom(12,2,2)$ & 2492 & 2492 & 1463 & 1463 & 3557 & 3557\\ 
$\hrandom(13,2,2)$ & 4879 & 4879 & 2708 & 2708 & 6973 & 6973\\ 
$\hrandom(14,2,2)$ & 9259 & 9259 & 5142 & 5142 & 13524 & 13524\\ 
\hline
\end{tabu}
}
\vspace*{2mm}
\caption{Size of \basis{} (homogeneous)}
\end{minipage}
\label{table:size-of-basis-homogeneous-systems}
\end{table}

\begin{table}[h]
\centering
\begin{minipage}{.62\linewidth}
\resizebox{\textwidth}{!}{
\begin{tabu} {@{}c|S[table-format=8.0]
  S[table-format=8.0]|S[table-format=8.0]
  S[table-format=8.0]|S[table-format=8.0]
  S[table-format=8.0] @{}}
  \multirow{2}{*}{\bf Benchmark} & \multicolumn{2}{c|}{$\potl$} &
  \multicolumn{2}{c|}{$\schl$} & \multicolumn{2}{c}{$\dpotl$}\\ 
  & \multicolumn{1}{c}{$\rleqff$} & \multicolumn{1}{c|}{$\rleqsb$}
  & \multicolumn{1}{c}{$\rleqff$} & \multicolumn{1}{c|}{$\rleqsb$}
  & \multicolumn{1}{c}{$\rleqff$} & \multicolumn{1}{c}{$\rleqsb$}\\ 
\hline
cyclic-7 & 7260 & 7260 & 187 & 187 & 439 & 338\\ 
cyclic-8 & 103285 & 103285 & 761 & 761 & 3691 & 2599\\ 
eco-10 & 30508 & 30508 & 412 & 412 & 1111 & 848\\ 
eco-11 & 118110 & 118110 & 804 & 804 & 2978 & 1750\\ 
f-633 & 122 & 122 & 37 & 37 & 40 & 40\\ 
f-744 & 16616 & 16616 & 316 & 316 & 654 & 641\\ 
katsura-11 & 24976 & 24976 & 408 & 408 & 2728 & 2670\\ 
katsura-12 & 92235 & 92235 & 706 & 706 & 9065 & 9148\\ 
noon-8 & 406 & 406 & 322 & 322 & 84 & 84\\ 
noon-9 & 666 & 666 & 718 & 718 & 120 & 120\\ 
$\hrandom(6,2,2)$ & 231 & 231 & 41 & 41 & 57 & 57\\ 
$\hrandom(7,2,2)$ & 780 & 780 & 70 & 70 & 119 & 119\\ 
$\hrandom(7,2,4)$ & 3160 & 3160 & 123 & 123 & 152 & 152\\ 
$\hrandom(7,2,6)$ & 162735 & 162735 & 681 & 681 & 950 & 950\\ 
$\hrandom(8,2,2)$ & 2278 & 2278 & 130 & 130 & 243 & 243\\ 
$\hrandom(8,2,4)$ & 41328 & 41328 & 402 & 402 & 573 & 573\\ 
$\hrandom(9,2,2)$ & 8256 & 8256 & 217 & 217 & 485 & 485\\ 
$\hrandom(10,2,2)$ & 24976 & 24976 & 384 & 384 & 964 & 964\\ 
$\hrandom(11,2,2)$ & 90951 & 90951 & 645 & 645 & 1896 & 1896\\ 
$\hrandom(12,2,2)$ & 294528 & 294528 & 1149 & 1149 & 3728 & 3728\\ 
$\hrandom(13,2,2)$ & 1070916 & 1070916 & 1945 & 1945 & 7285 & 7285\\ 
$\hrandom(14,2,2)$ & 3667986 & 3667986 & 3494 & 3494 & 14258 & 14258\\ 
\hline
\end{tabu}
}
\vspace*{2mm}
\caption{Size of \hdsyz{} (homogeneous)}
\label{table:size-of-hdsyz-homogeneous-systems}
\end{minipage}
\end{table}

\begin{center}
\begin{table}[H]
\resizebox{0.90\textwidth}{!}{
\begin{tabu} to \linewidth {@{}c|S[table-format=7.0]
  S[table-format=7.0]|S[table-format=7.0]
  S[table-format=7.0]|S[table-format=7.0]
  S[table-format=7.0]|S[table-format=7.0]
  S[table-format=7.0]|S[table-format=7.0]
  S[table-format=7.0]|S[table-format=7.0]
  S[table-format=7.0] @{}}
  \multirow{3}{*}{\bf Benchmark} & \multicolumn{6}{c|}{only top \sreductions{}} 
  & \multicolumn{6}{c}{full \sreductions{}}\\ 
  & \multicolumn{2}{c|}{$\potl$} &
  \multicolumn{2}{c|}{$\schl$} & \multicolumn{2}{c|}{$\dpotl$} 
  & \multicolumn{2}{c|}{$\potl$} &
  \multicolumn{2}{c|}{$\schl$} & \multicolumn{2}{c}{$\dpotl$}\\ 
  & \multicolumn{1}{c}{$\rleqff$} & \multicolumn{1}{c|}{$\rleqsb$}
  & \multicolumn{1}{c}{$\rleqff$} & \multicolumn{1}{c|}{$\rleqsb$}
  & \multicolumn{1}{c}{$\rleqff$} & \multicolumn{1}{c|}{$\rleqsb$}
  & \multicolumn{1}{c}{$\rleqff$} & \multicolumn{1}{c|}{$\rleqsb$}
  & \multicolumn{1}{c}{$\rleqff$} & \multicolumn{1}{c|}{$\rleqsb$}
  & \multicolumn{1}{c}{$\rleqff$} & \multicolumn{1}{c}{$\rleqsb$}\\ 
\hline
cyclic-7 & \num{1e17.161} & \num{1e16.798} & \num{1e18.257} & \num{1e18.127} & \num{1e18.094} & \num{1e17.630} & \num{1e16.844} & \num{1e16.492} & \num{1e17.267} & \num{1e17.134} & \num{1e17.347} & \num{1e16.939}\\ 
(incl. interred) & \num{1e17.455} & \num{1e17.209} &  &  &  &  & \num{1e16.936} & \num{1e16.619} &  &  &  & \\ 
\tabucline[0.1pt on 1pt off 4pt]{-}
cyclic-8 & \num{1e22.575} & \num{1e21.346} & \num{1e22.755} & \num{1e22.331} & \num{1e23.638} & \num{1e22.325} & \num{1e22.529} & \num{1e21.303} & \num{1e21.916} & \num{1e21.529} & \num{1e22.981} & \num{1e21.696}\\ 
(incl. interred) & \num{1e22.967} & \num{1e22.156} &  &  &  &  & \num{1e22.730} & \num{1e21.727} &  &  &  & \\ 
\tabucline[0.1pt on 1pt off 4pt]{-}
eco-10 & \num{1e18.901} & \num{1e18.565} & \num{1e19.569} & \num{1e19.507} & \num{1e19.232} & \num{1e18.859} & \num{1e19.938} & \num{1e18.986} & \num{1e18.564} & \num{1e18.503} & \num{1e19.946} & \num{1e19.075}\\ 
(incl. interred) & \num{1e19.038} & \num{1e18.726} &  &  &  &  & \num{1e19.972} & \num{1e19.053} &  &  &  & \\ 
\tabucline[0.1pt on 1pt off 4pt]{-}
eco-11 & \num{1e21.519} & \num{1e20.976} & \num{1e21.909} & \num{1e21.851} & \num{1e21.735} & \num{1e21.123} & \num{1e22.996} & \num{1e21.290} & \num{1e21.126} & \num{1e21.034} & \num{1e22.772} & \num{1e21.531}\\ 
(incl. interred) & \num{1e21.644} & \num{1e21.149} &  &  &  &  & \num{1e23.014} & \num{1e21.352} &  &  &  & \\ 
\tabucline[0.1pt on 1pt off 4pt]{-}
f-633 & \num{1e9.767} & \num{1e9.604} & \num{1e10.321} & \num{1e10.236} & \num{1e9.658} & \num{1e9.397} & \num{1e9.484} & \num{1e9.522} & \num{1e10.086} & \num{1e9.801} & \num{1e9.433} & \num{1e9.044}\\ 
(incl. interred) & \num{1e9.864} & \num{1e9.713} &  &  &  &  &  &  &  &  &  & \\ 
\tabucline[0.1pt on 1pt off 4pt]{-}
f-744 & \num{1e16.895} & \num{1e16.857} & \num{1e17.256} & \num{1e17.150} & \num{1e17.248} & \num{1e17.157} & \num{1e17.175} & \num{1e16.763} & \num{1e17.055} & \num{1e16.936} & \num{1e17.347} & \num{1e17.162}\\ 
(incl. interred) & \num{1e17.126} & \num{1e17.121} &  &  &  &  & \num{1e17.208} & \num{1e16.811} &  &  &  & \\ 
\tabucline[0.1pt on 1pt off 4pt]{-}
katsura-11 & \num{1e18.953} & \num{1e18.708} & \num{1e22.403} & \num{1e22.384} & \num{1e20.638} & \num{1e20.527} & \num{1e22.393} & \num{1e22.257} & \num{1e22.018} & \num{1e22.067} & \num{1e22.040} & \num{1e21.985}\\ 
(incl. interred) & \num{1e22.556} & \num{1e22.514} &  &  &  &  & \num{1e22.815} & \num{1e22.712} &  &  &  & \\ 
\tabucline[0.1pt on 1pt off 4pt]{-}
katsura-12 & \num{1e21.496} & \num{1e21.110} & \num{1e24.661} & \num{1e24.596} & \num{1e23.183} & \num{1e23.063} & \num{1e25.507} & \num{1e25.319} & \num{1e24.257} & \num{1e24.287} & \num{1e24.521} & \num{1e24.455}\\ 
(incl. interred) & \num{1e25.692} & \num{1e25.654} &  &  &  &  & \num{1e25.977} & \num{1e25.840} &  &  &  & \\ 
\tabucline[0.1pt on 1pt off 4pt]{-}
noon-8 & \num{1e15.745} & \num{1e14.634} & \num{1e16.310} & \num{1e15.662} & \num{1e15.745} & \num{1e14.634} & \num{1e18.166} & \num{1e18.023} & \num{1e18.230} & \num{1e18.094} & \num{1e18.165} & \num{1e18.022}\\ 
(incl. interred) & \num{1e15.747} & \num{1e14.638} &  &  &  &  &  &  &  &  &  & \\ 
\tabucline[0.1pt on 1pt off 4pt]{-}
noon-9 & \num{1e18.303} & \num{1e16.820} & \num{1e18.756} & \num{1e17.843} & \num{1e18.303} & \num{1e16.820} & \num{1e20.787} & \num{1e20.606} & \num{1e20.835} & \num{1e20.660} & \num{1e20.787} & \num{1e20.606}\\ 
(incl. interred) &  & \num{1e16.822} &  &  &  &  &  &  &  &  &  & \\ 
\tabucline[0.1pt on 1pt off 4pt]{-}
$\hrandom(6,2,2)$ & \num{1e10.039} & \num{1e10.229} & \num{1e11.126} & \num{1e11.126} & \num{1e10.764} & \num{1e10.966} & \num{1e10.137} & \num{1e10.364} & \num{1e10.537} & \num{1e10.537} & \num{1e10.707} & \num{1e10.880}\\ 
(incl. interred) & \num{1e10.982} & \num{1e11.094} &  &  &  &  & \num{1e10.408} & \num{1e10.599} &  &  &  & \\ 
\tabucline[0.1pt on 1pt off 4pt]{-}
$\hrandom(7,2,2)$ & \num{1e12.189} & \num{1e12.308} & \num{1e13.279} & \num{1e13.279} & \num{1e13.046} & \num{1e13.209} & \num{1e12.103} & \num{1e12.294} & \num{1e12.298} & \num{1e12.298} & \num{1e13.015} & \num{1e13.122}\\ 
(incl. interred) & \num{1e13.227} & \num{1e13.327} &  &  &  &  & \num{1e12.435} & \num{1e12.589} &  &  &  & \\ 
\tabucline[0.1pt on 1pt off 4pt]{-}
$\hrandom(7,2,4)$ & \num{1e16.664} & \num{1e16.701} & \num{1e16.782} & \num{1e16.782} & \num{1e17.093} & \num{1e17.202} & \num{1e15.190} & \num{1e15.298} & \num{1e14.911} & \num{1e14.911} & \num{1e16.557} & \num{1e16.796}\\ 
(incl. interred) & \num{1e17.314} & \num{1e17.365} &  &  &  &  & \num{1e15.268} & \num{1e15.370} &  &  &  & \\ 
\tabucline[0.1pt on 1pt off 4pt]{-}
$\hrandom(7,2,6)$ & \num{1e23.748} & \num{1e23.763} & \num{1e23.614} & \num{1e23.614} & \num{1e24.113} & \num{1e24.175} & \num{1e22.136} & \num{1e22.130} & \num{1e21.400} & \num{1e21.400} & \num{1e23.783} & \num{1e23.906}\\ 
(incl. interred) & \num{1e24.421} & \num{1e24.458} &  &  &  &  & \num{1e22.217} & \num{1e22.211} &  &  &  & \\ 
\tabucline[0.1pt on 1pt off 4pt]{-}
$\hrandom(8,2,2)$ & \num{1e14.104} & \num{1e14.290} & \num{1e15.300} & \num{1e15.300} & \num{1e15.342} & \num{1e15.462} & \num{1e14.073} & \num{1e14.232} & \num{1e14.127} & \num{1e14.127} & \num{1e15.306} & \num{1e15.408}\\ 
(incl. interred) & \num{1e15.431} & \num{1e15.538} &  &  &  &  & \num{1e14.534} & \num{1e14.652} &  &  &  & \\ 
\tabucline[0.1pt on 1pt off 4pt]{-}
$\hrandom(8,2,4)$ & \num{1e20.898} & \num{1e20.923} & \num{1e21.097} & \num{1e21.097} & \num{1e21.574} & \num{1e21.644} & \num{1e19.426} & \num{1e19.468} & \num{1e18.918} & \num{1e18.918} & \num{1e21.209} & \num{1e21.402}\\ 
(incl. interred) & \num{1e21.660} & \num{1e21.694} &  &  &  &  & \num{1e19.593} & \num{1e19.631} &  &  &  & \\ 
\tabucline[0.1pt on 1pt off 4pt]{-}
$\hrandom(9,2,2)$ & \num{1e16.135} & \num{1e16.239} & \num{1e17.331} & \num{1e17.331} & \num{1e17.719} & \num{1e17.804} & \num{1e16.200} & \num{1e16.315} & \num{1e15.758} & \num{1e15.758} & \num{1e17.685} & \num{1e17.745}\\ 
(incl. interred) & \num{1e17.586} & \num{1e17.679} &  &  &  &  & \num{1e16.703} & \num{1e16.785} &  &  &  & \\ 
\tabucline[0.1pt on 1pt off 4pt]{-}
$\hrandom(10,2,2)$ & \num{1e18.291} & \num{1e18.369} & \num{1e19.271} & \num{1e19.271} & \num{1e20.099} & \num{1e20.156} & \num{1e18.249} & \num{1e18.339} & \num{1e17.491} & \num{1e17.491} & \num{1e20.072} & \num{1e20.120}\\ 
(incl. interred) & \num{1e19.893} & \num{1e19.949} &  &  &  &  & \num{1e18.889} & \num{1e18.947} &  &  &  & \\ 
\tabucline[0.1pt on 1pt off 4pt]{-}
$\hrandom(11,2,2)$ & \num{1e20.217} & \num{1e20.299} & \num{1e21.169} & \num{1e21.169} & \num{1e22.524} & \num{1e22.563} & \num{1e20.448} & \num{1e20.510} & \num{1e19.105} & \num{1e19.105} & \num{1e22.490} & \num{1e22.507}\\ 
(incl. interred) & \num{1e22.013} & \num{1e22.057} &  &  &  &  & \num{1e21.078} & \num{1e21.118} &  &  &  & \\ 
\tabucline[0.1pt on 1pt off 4pt]{-}
$\hrandom(12,2,2)$ & \num{1e22.508} & \num{1e22.552} & \num{1e23.200} & \num{1e23.200} & \num{1e24.955} & \num{1e24.980} & \num{1e22.636} & \num{1e22.679} & \num{1e21.024} & \num{1e21.024} & \num{1e24.928} & \num{1e24.949}\\ 
(incl. interred) & \num{1e24.358} & \num{1e24.380} &  &  &  &  & \num{1e23.367} & \num{1e23.394} &  &  &  & \\ 
\tabucline[0.1pt on 1pt off 4pt]{-}
$\hrandom(13,2,2)$ & \num{1e24.684} & \num{1e24.805} & \num{1e25.310} & \num{1e25.310} & \num{1e27.409} & \num{1e27.427} & \num{1e24.937} & \num{1e24.966} & \num{1e22.506} & \num{1e22.506} & \num{1e27.370} & \num{1e27.373}\\ 
(incl. interred) & \num{1e26.562} & \num{1e26.606} &  &  &  &  & \num{1e25.651} & \num{1e25.669} &  &  &  & \\ 
\tabucline[0.1pt on 1pt off 4pt]{-}
$\hrandom(14,2,2)$ &  &  &  &  &  &  & \num{1e26.989} & \num{1e27.011} & \num{1e24.273} & \num{1e24.273} & \num{1e29.792} & \num{1e29.801}\\ 
(incl. interred) &  &  &  &  &  &  & \num{1e27.844} & \num{1e27.856} &  &  &  & \\ 
\hline
\end{tabu}}
\vspace{2mm}
\caption{\# \sreductions{} (incl. interreductions) (homogeneous)}
\label{table:number-of-sreductions-incl-interreductions-homogeneous-systems}
\end{table}
\end{center}

\begin{center}
\begin{table}[H]
\resizebox{0.92\textwidth}{!}{
\begin{tabu} to \linewidth {@{}c|S[table-format=7.0]
  S[table-format=7.0]|S[table-format=7.0]
  S[table-format=7.0]|S[table-format=7.0]
  S[table-format=7.0]|S[table-format=7.0]
  S[table-format=7.0]|S[table-format=7.0]
  S[table-format=7.0]|S[table-format=7.0]
  S[table-format=7.0] @{}}
  \multirow{3}{*}{\bf Benchmark} & \multicolumn{6}{c|}{only top \sreductions{}} 
  & \multicolumn{6}{c}{full \sreductions{}}\\ 
  & \multicolumn{2}{c|}{$\potl$} &
  \multicolumn{2}{c|}{$\schl$} & \multicolumn{2}{c|}{$\dpotl$} 
  & \multicolumn{2}{c|}{$\potl$} &
  \multicolumn{2}{c|}{$\schl$} & \multicolumn{2}{c}{$\dpotl$}\\ 
  & \multicolumn{1}{c}{$\rleqff$} & \multicolumn{1}{c|}{$\rleqsb$}
  & \multicolumn{1}{c}{$\rleqff$} & \multicolumn{1}{c|}{$\rleqsb$}
  & \multicolumn{1}{c}{$\rleqff$} & \multicolumn{1}{c|}{$\rleqsb$}
  & \multicolumn{1}{c}{$\rleqff$} & \multicolumn{1}{c|}{$\rleqsb$}
  & \multicolumn{1}{c}{$\rleqff$} & \multicolumn{1}{c|}{$\rleqsb$}
  & \multicolumn{1}{c}{$\rleqff$} & \multicolumn{1}{c}{$\rleqsb$}\\ 
\hline
cyclic-7 & \num{1e24.276} & \num{1e23.871} & \num{1e24.024} & \num{1e23.996} & \num{1e23.851} & \num{1e23.467} & \num{1e24.050} & \num{1e23.598} & \num{1e24.257} & \num{1e24.133} & \num{1e24.319} & \num{1e23.866}\\ 
(incl. interred) & \num{1e24.429} & \num{1e24.076} &  &  &  &  & \num{1e24.130} & \num{1e23.709} &  &  &  & \\ 
\tabucline[0.1pt on 1pt off 4pt]{-}
cyclic-8 & \num{1e31.011} & \num{1e29.800} & \num{1e29.796} & \num{1e29.543} & \num{1e30.427} & \num{1e29.361} & \num{1e30.890} & \num{1e29.641} & \num{1e30.162} & \num{1e29.738} & \num{1e31.363} & \num{1e30.127}\\ 
(incl. interred) & \num{1e31.287} & \num{1e30.374} &  &  &  &  & \num{1e31.086} & \num{1e30.066} &  &  &  & \\ 
\tabucline[0.1pt on 1pt off 4pt]{-}
eco-10 & \num{1e24.459} & \num{1e24.148} & \num{1e24.120} & \num{1e24.102} & \num{1e23.676} & \num{1e23.394} & \num{1e25.484} & \num{1e24.560} & \num{1e24.046} & \num{1e23.975} & \num{1e25.292} & \num{1e24.480}\\ 
(incl. interred) & \num{1e24.582} & \num{1e24.294} &  &  &  &  & \num{1e25.514} & \num{1e24.619} &  &  &  & \\ 
\tabucline[0.1pt on 1pt off 4pt]{-}
eco-11 & \num{1e27.765} & \num{1e27.229} & \num{1e26.904} & \num{1e26.895} & \num{1e26.556} & \num{1e26.111} & \num{1e29.186} & \num{1e27.564} & \num{1e27.216} & \num{1e27.098} & \num{1e28.630} & \num{1e27.505}\\ 
(incl. interred) & \num{1e27.881} & \num{1e27.392} &  &  &  &  & \num{1e29.203} & \num{1e27.619} &  &  &  & \\ 
\tabucline[0.1pt on 1pt off 4pt]{-}
f-633 & \num{1e12.149} & \num{1e12.024} & \num{1e12.776} & \num{1e12.744} & \num{1e12.029} & \num{1e11.764} & \num{1e12.069} & \num{1e12.166} & \num{1e12.716} & \num{1e12.607} & \num{1e12.120} & \num{1e11.809}\\ 
(incl. interred) & \num{1e12.228} & \num{1e12.109} &  &  &  &  &  &  &  &  &  & \\ 
\tabucline[0.1pt on 1pt off 4pt]{-}
f-744 & \num{1e21.443} & \num{1e21.480} & \num{1e21.888} & \num{1e21.799} & \num{1e21.654} & \num{1e21.656} & \num{1e21.767} & \num{1e21.414} & \num{1e21.798} & \num{1e21.695} & \num{1e21.842} & \num{1e21.683}\\ 
(incl. interred) & \num{1e21.613} & \num{1e21.664} &  &  &  &  & \num{1e21.795} & \num{1e21.452} &  &  &  & \\ 
\tabucline[0.1pt on 1pt off 4pt]{-}
katsura-11 & \num{1e27.790} & \num{1e27.539} & \num{1e29.423} & \num{1e29.366} & \num{1e28.251} & \num{1e28.072} & \num{1e29.937} & \num{1e29.809} & \num{1e29.268} & \num{1e29.314} & \num{1e29.290} & \num{1e29.208}\\ 
(incl. interred) & \num{1e30.118} & \num{1e30.064} &  &  &  &  & \num{1e30.408} & \num{1e30.315} &  &  &  & \\ 
\tabucline[0.1pt on 1pt off 4pt]{-}
katsura-12 & \num{1e30.965} & \num{1e30.650} & \num{1e32.378} & \num{1e32.238} & \num{1e31.579} & \num{1e31.352} & \num{1e33.522} & \num{1e33.337} & \num{1e32.214} & \num{1e32.240} & \num{1e32.537} & \num{1e32.421}\\ 
(incl. interred) & \num{1e33.729} & \num{1e33.679} &  &  &  &  & \num{1e34.073} & \num{1e33.947} &  &  &  & \\ 
\tabucline[0.1pt on 1pt off 4pt]{-}
noon-8 & \num{1e20.100} & \num{1e19.657} & \num{1e20.365} & \num{1e19.983} & \num{1e20.142} & \num{1e19.681} & \num{1e22.212} & \num{1e22.292} & \num{1e22.249} & \num{1e22.327} & \num{1e22.212} & \num{1e22.292}\\ 
(incl. interred) &  &  &  &  &  &  &  &  &  &  &  & \\ 
\tabucline[0.1pt on 1pt off 4pt]{-}
noon-9 & \num{1e22.901} & \num{1e22.234} & \num{1e23.042} & \num{1e22.455} & \num{1e22.901} & \num{1e22.234} & \num{1e25.142} & \num{1e25.212} & \num{1e25.165} & \num{1e25.234} & \num{1e25.142} & \num{1e25.212}\\ 
(incl. interred) &  &  &  &  &  &  &  &  &  &  &  & \\ 
\tabucline[0.1pt on 1pt off 4pt]{-}
$\hrandom(6,2,2)$ & \num{1e14.628} & \num{1e14.812} & \num{1e15.453} & \num{1e15.453} & \num{1e15.403} & \num{1e15.585} & \num{1e14.626} & \num{1e14.827} & \num{1e14.614} & \num{1e14.614} & \num{1e15.180} & \num{1e15.340}\\ 
(incl. interred) & \num{1e15.502} & \num{1e15.615} &  &  &  &  & \num{1e14.932} & \num{1e15.096} &  &  &  & \\ 
\tabucline[0.1pt on 1pt off 4pt]{-}
$\hrandom(7,2,2)$ & \num{1e17.592} & \num{1e17.729} & \num{1e18.409} & \num{1e18.409} & \num{1e18.430} & \num{1e18.566} & \num{1e17.448} & \num{1e17.602} & \num{1e17.275} & \num{1e17.275} & \num{1e18.217} & \num{1e18.316}\\ 
(incl. interred) & \num{1e18.505} & \num{1e18.608} &  &  &  &  & \num{1e17.806} & \num{1e17.927} &  &  &  & \\ 
\tabucline[0.1pt on 1pt off 4pt]{-}
$\hrandom(7,2,4)$ & \num{1e23.387} & \num{1e23.432} & \num{1e23.509} & \num{1e23.509} & \num{1e23.906} & \num{1e24.003} & \num{1e22.172} & \num{1e22.241} & \num{1e21.820} & \num{1e21.820} & \num{1e22.992} & \num{1e23.136}\\ 
(incl. interred) & \num{1e23.763} & \num{1e23.808} &  &  &  &  & \num{1e22.246} & \num{1e22.312} &  &  &  & \\ 
\tabucline[0.1pt on 1pt off 4pt]{-}
$\hrandom(7,2,6)$ & \num{1e33.815} & \num{1e33.837} & \num{1e33.676} & \num{1e33.676} & \num{1e34.197} & \num{1e34.256} & \num{1e32.492} & \num{1e32.495} & \num{1e31.905} & \num{1e31.905} & \num{1e33.185} & \num{1e33.239}\\ 
(incl. interred) & \num{1e34.120} & \num{1e34.142} &  &  &  &  & \num{1e32.541} & \num{1e32.544} &  &  &  & \\ 
\tabucline[0.1pt on 1pt off 4pt]{-}
$\hrandom(8,2,2)$ & \num{1e20.337} & \num{1e20.506} & \num{1e21.167} & \num{1e21.167} & \num{1e21.455} & \num{1e21.544} & \num{1e20.288} & \num{1e20.398} & \num{1e19.965} & \num{1e19.965} & \num{1e21.233} & \num{1e21.315}\\ 
(incl. interred) & \num{1e21.430} & \num{1e21.532} &  &  &  &  & \num{1e20.723} & \num{1e20.806} &  &  &  & \\ 
\tabucline[0.1pt on 1pt off 4pt]{-}
$\hrandom(8,2,4)$ & \num{1e29.458} & \num{1e29.498} & \num{1e29.617} & \num{1e29.617} & \num{1e30.086} & \num{1e30.171} & \num{1e28.285} & \num{1e28.307} & \num{1e27.830} & \num{1e27.830} & \num{1e29.250} & \num{1e29.361}\\ 
(incl. interred) & \num{1e29.869} & \num{1e29.904} &  &  &  &  & \num{1e28.397} & \num{1e28.418} &  &  &  & \\ 
\tabucline[0.1pt on 1pt off 4pt]{-}
$\hrandom(9,2,2)$ & \num{1e23.220} & \num{1e23.348} & \num{1e23.913} & \num{1e23.913} & \num{1e24.515} & \num{1e24.572} & \num{1e23.209} & \num{1e23.282} & \num{1e22.474} & \num{1e22.474} & \num{1e24.250} & \num{1e24.295}\\ 
(incl. interred) & \num{1e24.351} & \num{1e24.436} &  &  &  &  & \num{1e23.656} & \num{1e23.710} &  &  &  & \\ 
\tabucline[0.1pt on 1pt off 4pt]{-}
$\hrandom(10,2,2)$ & \num{1e26.107} & \num{1e26.187} & \num{1e26.596} & \num{1e26.596} & \num{1e27.576} & \num{1e27.609} & \num{1e26.082} & \num{1e26.133} & \num{1e25.097} & \num{1e25.097} & \num{1e27.301} & \num{1e27.332}\\ 
(incl. interred) & \num{1e27.328} & \num{1e27.377} &  &  &  &  & \num{1e26.589} & \num{1e26.625} &  &  &  & \\ 
\tabucline[0.1pt on 1pt off 4pt]{-}
$\hrandom(11,2,2)$ & \num{1e28.985} & \num{1e29.069} & \num{1e29.278} & \num{1e29.278} & \num{1e30.684} & \num{1e30.703} & \num{1e29.008} & \num{1e29.041} & \num{1e27.607} & \num{1e27.607} & \num{1e30.351} & \num{1e30.362}\\ 
(incl. interred) & \num{1e30.234} & \num{1e30.277} &  &  &  &  & \num{1e29.512} & \num{1e29.535} &  &  &  & \\ 
\tabucline[0.1pt on 1pt off 4pt]{-}
$\hrandom(12,2,2)$ & \num{1e31.888} & \num{1e31.953} & \num{1e32.059} & \num{1e32.059} & \num{1e33.791} & \num{1e33.802} & \num{1e31.914} & \num{1e31.936} & \num{1e30.326} & \num{1e30.326} & \num{1e33.462} & \num{1e33.472}\\ 
(incl. interred) & \num{1e33.203} & \num{1e33.232} &  &  &  &  & \num{1e32.468} & \num{1e32.483} &  &  &  & \\ 
\tabucline[0.1pt on 1pt off 4pt]{-}
$\hrandom(13,2,2)$ & \num{1e34.837} & \num{1e34.975} & \num{1e34.881} & \num{1e34.881} & \num{1e36.959} & \num{1e36.966} & \num{1e34.856} & \num{1e34.870} & \num{1e32.855} & \num{1e32.855} & \num{1e36.581} & \num{1e36.583}\\ 
(incl. interred) & \num{1e36.137} & \num{1e36.198} &  &  &  &  & \num{1e35.405} & \num{1e35.415} &  &  &  & \\ 
\tabucline[0.1pt on 1pt off 4pt]{-}
$\hrandom(14,2,2)$ &  &  &  &  &  &  & \num{1e37.720} & \num{1e37.729} & \num{1e35.562} & \num{1e35.562} & \num{1e39.738} & \num{1e39.741}\\ 
(incl. interred) &  &  &  &  &  &  & \num{1e38.315} & \num{1e38.321} &  &  &  & \\ 
\hline
\end{tabu}}
\vspace{2mm}
\caption{\# multiplications (incl. interreductions) (homogeneous)}
\label{table:number-of-multiplications-incl-interreductions-homogeneous-systems}
\end{table}
\end{center}

\subsection{Experimental results for affine systems}

\begin{center}
\begin{table}[H]
\begin{minipage}{.50\linewidth}
\resizebox{0.90\textwidth}{!}{
\begin{tabu} to \linewidth {@{}c|S[table-format=5.0]
  S[table-format=5.0]|S[table-format=5.0]
  S[table-format=5.0]|S[table-format=5.0]
  S[table-format=5.0] @{}}
  \multirow{2}{*}{\bf Benchmark} & \multicolumn{2}{c|}{$\potl$} &
  \multicolumn{2}{c|}{$\schl$} & \multicolumn{2}{c}{$\dpotl$}\\ 
  & \multicolumn{1}{c}{$\rleqff$} & \multicolumn{1}{c|}{$\rleqsb$}
  & \multicolumn{1}{c}{$\rleqff$} & \multicolumn{1}{c|}{$\rleqsb$}
  & \multicolumn{1}{c}{$\rleqff$} & \multicolumn{1}{c}{$\rleqsb$}\\ 
\hline
cyclic-7 & 36 & 36 & 145 & 145 & 36 & 36\\ 
cyclic-8 & 244 & 244 & 672 & 672 & 244 & 244\\ 
eco-10 & 0 & 0 & 367 & 367 & 367 & 367\\ 
eco-11 & 0 & 0 & 749 & 749 & 749 & 749\\ 
f-633 & 1 & 1 & 9 & 9 & 3 & 3\\ 
f-744 & 1 & 1 & 259 & 259 & 188 & 188\\ 
katsura-11 & 0 & 0 & 353 & 353 & 0 & 0\\ 
katsura-12 & 0 & 0 & 640 & 640 & 0 & 0\\ 
noon-8 & 0 & 0 & 294 & 294 & 0 & 0\\ 
noon-9 & 0 & 0 & 682 & 682 & 0 & 0\\ 
$\random(6,2,2)$ & 0 & 0 & 26 & 26 & 0 & 0\\ 
$\random(7,2,2)$ & 0 & 0 & 49 & 49 & 0 & 0\\ 
$\random(8,2,2)$ & 0 & 0 & 102 & 102 & 0 & 0\\ 
$\random(9,2,2)$ & 0 & 0 & 181 & 181 & 0 & 0\\ 
$\random(10,2,2)$ & 0 & 0 & 339 & 339 & 0 & 0\\ 
$\random(11,2,2)$ & 0 & 0 & 590 & 590 & 0 & 0\\ 
$\random(12,2,2)$ & 0 & 0 & 1083 & 1083 & 0 & 0\\ 
$\random(13,2,2)$ & 0 & 0 & 1867 & 1867 & 0 & 0\\ 
$\random(14,2,2)$ & 0 & 0 & 3403 & 3403 & 0 & 0\\ 
\hline
\end{tabu}}
\vspace{3mm}
\caption{\# zero reductions (affine)}
\label{table:number-of-zero-reductions-affine-systems}
\end{minipage}
\qquad
\begin{minipage}{.50\linewidth}
\resizebox{\textwidth}{!}{
\begin{tabu} to \linewidth {@{}c|S[table-format=6.0]
  S[table-format=6.0]|S[table-format=6.0]
  S[table-format=6.0]|S[table-format=6.0]
  S[table-format=6.0] @{}}
  \multirow{2}{*}{\bf Benchmark} & \multicolumn{2}{c|}{$\potl$} &
  \multicolumn{2}{c|}{$\schl$} & \multicolumn{2}{c}{$\dpotl$}\\ 
  & \multicolumn{1}{c}{$\rleqff$} & \multicolumn{1}{c|}{$\rleqsb$}
  & \multicolumn{1}{c}{$\rleqff$} & \multicolumn{1}{c|}{$\rleqsb$}
  & \multicolumn{1}{c}{$\rleqff$} & \multicolumn{1}{c}{$\rleqsb$}\\ 
\hline
cyclic-7 & 779 & 679 & 871 & 848 & 949 & 751\\ 
cyclic-8 & 3559 & 2775 & 4074 & 3658 & 5534 & 3884\\ 
eco-10 & 522 & 405 & 541 & 478 & 782 & 671\\ 
eco-11 & 1055 & 774 & 1092 & 965 & 1717 & 1415\\ 
f-633 & 60 & 58 & 61 & 59 & 61 & 57\\ 
f-744 & 468 & 442 & 377 & 348 & 502 & 464\\ 
katsura-11 & 762 & 743 & 553 & 553 & 2188 & 2161\\ 
katsura-12 & 1473 & 1474 & 1076 & 1076 & 6020 & 6020\\ 
noon-8 & 1384 & 1390 & 1384 & 1389 & 1384 & 1389\\ 
noon-9 & 3743 & 3750 & 3743 & 3749 & 3743 & 3749\\ 
$\random(6,2,2)$ & 58 & 58 & 39 & 39 & 62 & 62\\ 
$\random(7,2,2)$ & 113 & 113 & 67 & 67 & 124 & 124\\ 
$\random(8,2,2)$ & 212 & 212 & 128 & 128 & 242 & 242\\ 
$\random(9,2,2)$ & 365 & 366 & 223 & 223 & 479 & 479\\ 
$\random(10,2,2)$ & 677 & 677 & 426 & 426 & 932 & 932\\ 
$\random(11,2,2)$ & 1327 & 1357 & 767 & 767 & 1832 & 1832\\ 
$\random(12,2,2)$ & 2502 & 2537 & 1463 & 1463 & 3557 & 3557\\ 
$\random(13,2,2)$ & 4879 & 4879 & 2708 & 2708 & 6973 & 6973\\ 
$\random(14,2,2)$ & 9259 & 9259 & 5142 & 5142 & 13924 & 13934\\ 
\hline
\end{tabu}}
\vspace{2mm}
\caption{Size of \basis{} (affine)}
\label{table:size-of-basis-affine-systems}
\end{minipage}
\end{table}
\end{center}

\begin{center}
\begin{table}[H]
\centering
\begin{minipage}{.62\linewidth}
\resizebox{\textwidth}{!}{
\begin{tabu} to \linewidth {@{}c|S[table-format=8.0]
  S[table-format=8.0]|S[table-format=8.0]
  S[table-format=8.0]|S[table-format=8.0]
  S[table-format=8.0] @{}}
  \multirow{2}{*}{\bf Benchmark} & \multicolumn{2}{c|}{$\potl$} &
  \multicolumn{2}{c|}{$\schl$} & \multicolumn{2}{c}{$\dpotl$}\\ 
  & \multicolumn{1}{c}{$\rleqff$} & \multicolumn{1}{c|}{$\rleqsb$}
  & \multicolumn{1}{c}{$\rleqff$} & \multicolumn{1}{c|}{$\rleqsb$}
  & \multicolumn{1}{c}{$\rleqff$} & \multicolumn{1}{c}{$\rleqsb$}\\ 
\hline
cyclic-7 & 10011 & 10011 & 187 & 187 & 439 & 338\\ 
cyclic-8 & 186966 & 189420 & 761 & 761 & 3691 & 2599\\ 
eco-10 & 21528 & 21528 & 412 & 412 & 2809 & 2531\\ 
eco-11 & 73153 & 71253 & 804 & 804 & 6869 & 5914\\ 
f-633 & 120 & 120 & 37 & 37 & 40 & 40\\ 
f-744 & 20503 & 19701 & 316 & 316 & 641 & 628\\ 
katsura-11 & 40755 & 35511 & 408 & 408 & 2728 & 2670\\ 
katsura-12 & 134940 & 134940 & 706 & 706 & 9065 & 9148\\ 
noon-8 & 406 & 406 & 322 & 322 & 84 & 84\\ 
noon-9 & 666 & 666 & 718 & 718 & 120 & 120\\ 
$\random(6,2,2)$ & 378 & 378 & 41 & 41 & 57 & 57\\ 
$\random(7,2,2)$ & 1326 & 1326 & 70 & 70 & 119 & 119\\ 
$\random(8,2,2)$ & 4465 & 4465 & 130 & 130 & 243 & 243\\ 
$\random(9,2,2)$ & 8256 & 8385 & 217 & 217 & 485 & 485\\ 
$\random(10,2,2)$ & 25200 & 25200 & 384 & 384 & 964 & 964\\ 
$\random(11,2,2)$ & 91378 & 104653 & 645 & 645 & 1896 & 1896\\ 
$\random(12,2,2)$ & 302253 & 330078 & 1149 & 1149 & 3728 & 3728\\ 
$\random(13,2,2)$ & 1070916 & 1070916 & 1945 & 1945 & 7285 & 7285\\ 
$\random(14,2,2)$ & 3667986 & 3667986 & 3494 & 3494 & 15122 & 15122\\ 
\hline
\end{tabu}}
\vspace{2mm}
\caption{Size of \hdsyz{} (affine)}
\label{table:size-of-hdsyz-affine-systems}
\end{minipage}
\end{table}
\end{center}

\begin{center}
\begin{table}[H]
\resizebox{0.92\textwidth}{!}{
\begin{tabu} to \linewidth {@{}c|S[table-format=7.0]
  S[table-format=7.0]|S[table-format=7.0]
  S[table-format=7.0]|S[table-format=7.0]
  S[table-format=7.0]|S[table-format=7.0]
  S[table-format=7.0]|S[table-format=7.0]
  S[table-format=7.0]|S[table-format=7.0]
  S[table-format=7.0] @{}}
  \multirow{3}{*}{\bf Benchmark} & \multicolumn{6}{c|}{only top \sreductions{}} 
  & \multicolumn{6}{c}{full \sreductions{}}\\ 
  & \multicolumn{2}{c|}{$\potl$} &
  \multicolumn{2}{c|}{$\schl$} & \multicolumn{2}{c|}{$\dpotl$} 
  & \multicolumn{2}{c|}{$\potl$} &
  \multicolumn{2}{c|}{$\schl$} & \multicolumn{2}{c}{$\dpotl$}\\ 
  & \multicolumn{1}{c}{$\rleqff$} & \multicolumn{1}{c|}{$\rleqsb$}
  & \multicolumn{1}{c}{$\rleqff$} & \multicolumn{1}{c|}{$\rleqsb$}
  & \multicolumn{1}{c}{$\rleqff$} & \multicolumn{1}{c|}{$\rleqsb$}
  & \multicolumn{1}{c}{$\rleqff$} & \multicolumn{1}{c|}{$\rleqsb$}
  & \multicolumn{1}{c}{$\rleqff$} & \multicolumn{1}{c|}{$\rleqsb$}
  & \multicolumn{1}{c}{$\rleqff$} & \multicolumn{1}{c}{$\rleqsb$}\\ 
\hline
cyclic-7 & \num{1e17.136} & \num{1e16.774} & \num{1e18.247} & \num{1e18.117} & \num{1e18.093} & \num{1e17.630} & \num{1e16.792} & \num{1e16.492} & \num{1e16.977} & \num{1e16.881} & \num{1e16.984} & \num{1e16.829}\\ 
(incl. interred) & \num{1e17.383} & \num{1e17.133} &  &  &  &  & \num{1e16.854} & \num{1e16.569} &  &  &  & \\ 
\tabucline[0.1pt on 1pt off 4pt]{-}
cyclic-8 & \num{1e22.533} & \num{1e21.343} & \num{1e22.742} & \num{1e22.312} & \num{1e23.598} & \num{1e22.286} & \num{1e22.454} & \num{1e21.273} & \num{1e21.810} & \num{1e21.439} & \num{1e22.693} & \num{1e21.504}\\ 
(incl. interred) & \num{1e22.890} & \num{1e22.093} &  &  &  &  & \num{1e22.669} & \num{1e21.730} &  &  &  & \\ 
\tabucline[0.1pt on 1pt off 4pt]{-}
eco-10 & \num{1e17.827} & \num{1e16.397} & \num{1e20.148} & \num{1e20.124} & \num{1e20.251} & \num{1e20.114} & \num{1e19.858} & \num{1e17.368} & \num{1e19.964} & \num{1e19.807} & \num{1e20.017} & \num{1e19.773}\\ 
(incl. interred) & \num{1e18.181} & \num{1e17.170} &  &  &  &  & \num{1e19.911} & \num{1e17.633} &  &  &  & \\ 
\tabucline[0.1pt on 1pt off 4pt]{-}
eco-11 & \num{1e20.872} & \num{1e18.652} & \num{1e22.605} & \num{1e22.567} & \num{1e22.720} & \num{1e22.563} & \num{1e23.620} & \num{1e19.822} & \num{1e22.497} & \num{1e22.307} & \num{1e22.511} & \num{1e22.290}\\ 
(incl. interred) & \num{1e21.208} & \num{1e19.755} &  &  &  &  & \num{1e23.651} & \num{1e20.177} &  &  &  & \\ 
\tabucline[0.1pt on 1pt off 4pt]{-}
f-633 & \num{1e9.644} & \num{1e9.426} & \num{1e10.828} & \num{1e10.768} & \num{1e10.131} & \num{1e9.833} & \num{1e9.401} & \num{1e9.202} & \num{1e9.977} & \num{1e9.713} & \num{1e9.386} & \num{1e9.031}\\ 
(incl. interred) & \num{1e9.730} & \num{1e9.526} &  &  &  &  &  &  &  &  &  & \\ 
\tabucline[0.1pt on 1pt off 4pt]{-}
f-744 & \num{1e15.037} & \num{1e14.422} & \num{1e17.500} & \num{1e17.526} & \num{1e17.500} & \num{1e17.579} & \num{1e15.305} & \num{1e14.829} & \num{1e16.996} & \num{1e16.872} & \num{1e17.123} & \num{1e17.089}\\ 
(incl. interred) & \num{1e15.535} & \num{1e15.361} &  &  &  &  & \num{1e15.601} & \num{1e15.100} &  &  &  & \\ 
\tabucline[0.1pt on 1pt off 4pt]{-}
katsura-11 & \num{1e18.856} & \num{1e18.627} & \num{1e22.411} & \num{1e22.395} & \num{1e20.837} & \num{1e20.730} & \num{1e22.191} & \num{1e22.055} & \num{1e21.081} & \num{1e21.280} & \num{1e21.994} & \num{1e21.944}\\ 
(incl. interred) & \num{1e22.478} & \num{1e22.438} &  &  &  &  & \num{1e22.620} & \num{1e22.537} &  &  &  & \\ 
\tabucline[0.1pt on 1pt off 4pt]{-}
katsura-12 & \num{1e21.302} & \num{1e21.254} & \num{1e24.540} & \num{1e24.528} & \num{1e23.412} & \num{1e23.288} & \num{1e24.980} & \num{1e24.867} & \num{1e23.689} & \num{1e23.741} & \num{1e24.488} & \num{1e24.437}\\ 
(incl. interred) & \num{1e25.391} & \num{1e25.372} &  &  &  &  & \num{1e25.648} & \num{1e25.575} &  &  &  & \\ 
\tabucline[0.1pt on 1pt off 4pt]{-}
noon-8 & \num{1e15.693} & \num{1e14.519} & \num{1e16.097} & \num{1e15.308} & \num{1e15.691} & \num{1e14.529} & \num{1e17.602} & \num{1e17.408} & \num{1e17.758} & \num{1e17.582} & \num{1e17.667} & \num{1e17.479}\\ 
(incl. interred) & \num{1e15.694} & \num{1e14.523} &  &  &  &  &  &  &  &  &  & \\ 
\tabucline[0.1pt on 1pt off 4pt]{-}
noon-9 & \num{1e18.197} & \num{1e16.651} & \num{1e18.523} & \num{1e17.422} & \num{1e18.202} & \num{1e16.650} & \num{1e20.203} & \num{1e19.962} & \num{1e20.312} & \num{1e20.084} & \num{1e20.243} & \num{1e20.003}\\ 
(incl. interred) & \num{1e18.198} & \num{1e16.652} &  &  &  &  &  &  &  &  &  & \\ 
\tabucline[0.1pt on 1pt off 4pt]{-}
$\random(6,2,2)$ & \num{1e10.008} & \num{1e10.189} & \num{1e13.068} & \num{1e13.068} & \num{1e10.752} & \num{1e10.946} & \num{1e11.008} & \num{1e11.274} & \num{1e11.696} & \num{1e11.696} & \num{1e11.659} & \num{1e11.814}\\ 
(incl. interred) & \num{1e11.342} & \num{1e11.434} &  &  &  &  & \num{1e11.218} & \num{1e11.451} &  &  &  & \\ 
\tabucline[0.1pt on 1pt off 4pt]{-}
$\random(7,2,2)$ & \num{1e12.137} & \num{1e12.267} & \num{1e15.076} & \num{1e15.076} & \num{1e13.018} & \num{1e13.169} & \num{1e13.110} & \num{1e13.322} & \num{1e13.393} & \num{1e13.393} & \num{1e14.010} & \num{1e14.110}\\ 
(incl. interred) & \num{1e13.632} & \num{1e13.752} &  &  &  &  & \num{1e13.382} & \num{1e13.559} &  &  &  & \\ 
\tabucline[0.1pt on 1pt off 4pt]{-}
$\random(8,2,2)$ & \num{1e13.984} & \num{1e14.212} & \num{1e17.269} & \num{1e17.269} & \num{1e15.297} & \num{1e15.406} & \num{1e15.168} & \num{1e15.345} & \num{1e15.285} & \num{1e15.285} & \num{1e16.353} & \num{1e16.435}\\ 
(incl. interred) & \num{1e15.911} & \num{1e16.021} &  &  &  &  & \num{1e15.561} & \num{1e15.698} &  &  &  & \\ 
\tabucline[0.1pt on 1pt off 4pt]{-}
$\random(9,2,2)$ & \num{1e16.106} & \num{1e16.210} & \num{1e19.294} & \num{1e19.294} & \num{1e17.656} & \num{1e17.729} & \num{1e17.438} & \num{1e17.562} & \num{1e16.906} & \num{1e16.906} & \num{1e18.755} & \num{1e18.805}\\ 
(incl. interred) & \num{1e18.019} & \num{1e18.225} &  &  &  &  & \num{1e17.876} & \num{1e17.968} &  &  &  & \\ 
\tabucline[0.1pt on 1pt off 4pt]{-}
$\random(10,2,2)$ & \num{1e18.013} & \num{1e18.096} & \num{1e21.381} & \num{1e21.381} & \num{1e20.021} & \num{1e20.070} & \num{1e19.532} & \num{1e19.629} & \num{1e18.703} & \num{1e18.703} & \num{1e21.156} & \num{1e21.193}\\ 
(incl. interred) & \num{1e20.089} & \num{1e20.170} &  &  &  &  & \num{1e20.071} & \num{1e20.139} &  &  &  & \\ 
\tabucline[0.1pt on 1pt off 4pt]{-}
$\random(11,2,2)$ & \num{1e20.117} & \num{1e20.165} & \num{1e23.397} & \num{1e23.397} & \num{1e22.418} & \num{1e22.450} & \num{1e21.859} & \num{1e21.923} & \num{1e20.377} & \num{1e20.377} & \num{1e23.591} & \num{1e23.605}\\ 
(incl. interred) & \num{1e22.409} & \num{1e22.476} &  &  &  &  & \num{1e22.391} & \num{1e22.436} &  &  &  & \\ 
\tabucline[0.1pt on 1pt off 4pt]{-}
$\random(12,2,2)$ & \num{1e22.039} & \num{1e22.087} & \num{1e25.447} & \num{1e25.447} & \num{1e24.833} & \num{1e24.854} & \num{1e24.089} & \num{1e24.134} & \num{1e22.169} & \num{1e22.169} & \num{1e26.028} & \num{1e26.044}\\ 
(incl. interred) & \num{1e24.512} & \num{1e24.597} &  &  &  &  & \num{1e24.737} & \num{1e24.766} &  &  &  & \\ 
\tabucline[0.1pt on 1pt off 4pt]{-}
$\random(13,2,2)$ &  &  &  &  &  &  & \num{1e26.472} & \num{1e26.501} & \num{1e23.917} & \num{1e23.917} & \num{1e28.463} & \num{1e28.465}\\ 
(incl. interred) &  &  &  &  &  &  & \num{1e27.079} & \num{1e27.097} &  &  &  & \\ 
\tabucline[0.1pt on 1pt off 4pt]{-}
$\random(14,2,2)$ &  &  &  &  &  &  & \num{1e28.496} & \num{1e28.517} & \num{1e25.697} & \num{1e25.697} & \num{1e30.897} & \num{1e30.886}\\ 
(incl. interred) &  &  &  &  &  &  & \num{1e29.255} & \num{1e29.268} &  &  &  & \\ 
\hline
\end{tabu}}
\vspace{2mm}
\caption{\# \sreductions{} (incl. interreductions) (affine)}
\label{table:number-of-sreductions-incl-interreductions-affine-systems}
\end{table}
\end{center}

\begin{center}
\begin{table}[H]
\resizebox{0.92\textwidth}{!}{
\begin{tabu} to \linewidth {@{}c|S[table-format=7.0]
  S[table-format=7.0]|S[table-format=7.0]
  S[table-format=7.0]|S[table-format=7.0]
  S[table-format=7.0]|S[table-format=7.0]
  S[table-format=7.0]|S[table-format=7.0]
  S[table-format=7.0]|S[table-format=7.0]
  S[table-format=7.0] @{}}
  \multirow{3}{*}{\bf Benchmark} & \multicolumn{6}{c|}{only top \sreductions{}} 
  & \multicolumn{6}{c}{full \sreductions{}}\\ 
  & \multicolumn{2}{c|}{$\potl$} &
  \multicolumn{2}{c|}{$\schl$} & \multicolumn{2}{c|}{$\dpotl$} 
  & \multicolumn{2}{c|}{$\potl$} &
  \multicolumn{2}{c|}{$\schl$} & \multicolumn{2}{c}{$\dpotl$}\\ 
  & \multicolumn{1}{c}{$\rleqff$} & \multicolumn{1}{c|}{$\rleqsb$}
  & \multicolumn{1}{c}{$\rleqff$} & \multicolumn{1}{c|}{$\rleqsb$}
  & \multicolumn{1}{c}{$\rleqff$} & \multicolumn{1}{c|}{$\rleqsb$}
  & \multicolumn{1}{c}{$\rleqff$} & \multicolumn{1}{c|}{$\rleqsb$}
  & \multicolumn{1}{c}{$\rleqff$} & \multicolumn{1}{c|}{$\rleqsb$}
  & \multicolumn{1}{c}{$\rleqff$} & \multicolumn{1}{c}{$\rleqsb$}\\ 
\hline
cyclic-7 & \num{1e24.253} & \num{1e23.830} & \num{1e24.008} & \num{1e23.989} & \num{1e23.866} & \num{1e23.482} & \num{1e24.065} & \num{1e23.612} & \num{1e24.220} & \num{1e24.105} & \num{1e24.229} & \num{1e23.832}\\ 
(incl. interred) & \num{1e24.366} & \num{1e23.986} &  &  &  &  & \num{1e24.122} & \num{1e23.679} &  &  &  & \\ 
\tabucline[0.1pt on 1pt off 4pt]{-}
cyclic-8 & \num{1e30.906} & \num{1e29.738} & \num{1e29.788} & \num{1e29.534} & \num{1e30.435} & \num{1e29.367} & \num{1e30.885} & \num{1e29.649} & \num{1e30.261} & \num{1e29.837} & \num{1e31.273} & \num{1e30.076}\\ 
(incl. interred) & \num{1e31.169} & \num{1e30.276} &  &  &  &  & \num{1e31.091} & \num{1e30.092} &  &  &  & \\ 
\tabucline[0.1pt on 1pt off 4pt]{-}
eco-10 & \num{1e22.815} & \num{1e21.728} & \num{1e24.428} & \num{1e24.509} & \num{1e24.608} & \num{1e24.492} & \num{1e25.262} & \num{1e22.866} & \num{1e25.001} & \num{1e24.889} & \num{1e24.996} & \num{1e24.758}\\ 
(incl. interred) & \num{1e23.141} & \num{1e22.336} &  &  &  &  & \num{1e25.302} & \num{1e23.062} &  &  &  & \\ 
\tabucline[0.1pt on 1pt off 4pt]{-}
eco-11 & \num{1e26.502} & \num{1e24.807} & \num{1e27.357} & \num{1e27.421} & \num{1e27.574} & \num{1e27.433} & \num{1e29.367} & \num{1e25.900} & \num{1e28.161} & \num{1e27.961} & \num{1e28.087} & \num{1e27.867}\\ 
(incl. interred) & \num{1e26.823} & \num{1e25.662} &  &  &  &  & \num{1e29.398} & \num{1e26.194} &  &  &  & \\ 
\tabucline[0.1pt on 1pt off 4pt]{-}
f-633 & \num{1e11.817} & \num{1e11.694} & \num{1e13.360} & \num{1e13.352} & \num{1e12.530} & \num{1e12.196} & \num{1e11.772} & \num{1e11.676} & \num{1e12.636} & \num{1e12.487} & \num{1e11.998} & \num{1e11.737}\\ 
(incl. interred) & \num{1e11.870} & \num{1e11.751} &  &  &  &  &  &  &  &  &  & \\ 
\tabucline[0.1pt on 1pt off 4pt]{-}
f-744 & \num{1e19.292} & \num{1e18.824} & \num{1e22.257} & \num{1e22.351} & \num{1e22.040} & \num{1e22.231} & \num{1e19.776} & \num{1e19.329} & \num{1e21.737} & \num{1e21.626} & \num{1e21.839} & \num{1e21.774}\\ 
(incl. interred) & \num{1e19.701} & \num{1e19.571} &  &  &  &  & \num{1e20.018} & \num{1e19.574} &  &  &  & \\ 
\tabucline[0.1pt on 1pt off 4pt]{-}
katsura-11 & \num{1e27.943} & \num{1e27.683} & \num{1e29.565} & \num{1e29.518} & \num{1e28.714} & \num{1e28.531} & \num{1e29.884} & \num{1e29.747} & \num{1e28.906} & \num{1e29.074} & \num{1e29.449} & \num{1e29.381}\\ 
(incl. interred) & \num{1e30.080} & \num{1e30.010} &  &  &  &  & \num{1e30.353} & \num{1e30.264} &  &  &  & \\ 
\tabucline[0.1pt on 1pt off 4pt]{-}
katsura-12 & \num{1e31.121} & \num{1e30.819} & \num{1e32.451} & \num{1e32.365} & \num{1e32.092} & \num{1e31.867} & \num{1e33.503} & \num{1e33.263} & \num{1e32.052} & \num{1e32.173} & \num{1e32.712} & \num{1e32.620}\\ 
(incl. interred) & \num{1e33.642} & \num{1e33.586} &  &  &  &  & \num{1e34.102} & \num{1e33.946} &  &  &  & \\ 
\tabucline[0.1pt on 1pt off 4pt]{-}
noon-8 & \num{1e20.056} & \num{1e19.653} & \num{1e20.264} & \num{1e19.887} & \num{1e20.114} & \num{1e19.689} & \num{1e21.588} & \num{1e21.714} & \num{1e21.678} & \num{1e21.798} & \num{1e21.623} & \num{1e21.747}\\ 
(incl. interred) & \num{1e20.057} & \num{1e19.654} &  &  &  &  &  &  &  &  &  & \\ 
\tabucline[0.1pt on 1pt off 4pt]{-}
noon-9 & \num{1e22.844} & \num{1e22.216} & \num{1e22.945} & \num{1e22.356} & \num{1e22.854} & \num{1e22.218} & \num{1e24.490} & \num{1e24.615} & \num{1e24.546} & \num{1e24.667} & \num{1e24.511} & \num{1e24.634}\\ 
(incl. interred) &  &  &  &  &  &  &  &  &  &  &  & \\ 
\tabucline[0.1pt on 1pt off 4pt]{-}
$\random(6,2,2)$ & \num{1e15.986} & \num{1e16.176} & \num{1e18.567} & \num{1e18.567} & \num{1e16.548} & \num{1e16.683} & \num{1e16.617} & \num{1e16.799} & \num{1e17.303} & \num{1e17.303} & \num{1e16.988} & \num{1e17.122}\\ 
(incl. interred) & \num{1e16.883} & \num{1e17.001} &  &  &  &  & \num{1e16.823} & \num{1e16.982} &  &  &  & \\ 
\tabucline[0.1pt on 1pt off 4pt]{-}
$\random(7,2,2)$ & \num{1e18.938} & \num{1e19.090} & \num{1e21.285} & \num{1e21.285} & \num{1e19.618} & \num{1e19.708} & \num{1e19.548} & \num{1e19.674} & \num{1e19.885} & \num{1e19.885} & \num{1e19.970} & \num{1e20.057}\\ 
(incl. interred) & \num{1e19.896} & \num{1e20.017} &  &  &  &  & \num{1e19.792} & \num{1e19.900} &  &  &  & \\ 
\tabucline[0.1pt on 1pt off 4pt]{-}
$\random(8,2,2)$ & \num{1e21.745} & \num{1e21.956} & \num{1e24.219} & \num{1e24.219} & \num{1e22.720} & \num{1e22.777} & \num{1e22.463} & \num{1e22.556} & \num{1e22.706} & \num{1e22.706} & \num{1e22.990} & \num{1e23.056}\\ 
(incl. interred) & \num{1e22.865} & \num{1e22.990} &  &  &  &  & \num{1e22.758} & \num{1e22.834} &  &  &  & \\ 
\tabucline[0.1pt on 1pt off 4pt]{-}
$\random(9,2,2)$ & \num{1e24.760} & \num{1e24.885} & \num{1e26.994} & \num{1e26.994} & \num{1e25.853} & \num{1e25.887} & \num{1e25.443} & \num{1e25.505} & \num{1e25.278} & \num{1e25.278} & \num{1e25.952} & \num{1e25.990}\\ 
(incl. interred) & \num{1e25.888} & \num{1e26.020} &  &  &  &  & \num{1e25.756} & \num{1e25.807} &  &  &  & \\ 
\tabucline[0.1pt on 1pt off 4pt]{-}
$\random(10,2,2)$ & \num{1e27.602} & \num{1e27.690} & \num{1e29.876} & \num{1e29.876} & \num{1e29.022} & \num{1e29.043} & \num{1e28.360} & \num{1e28.404} & \num{1e28.060} & \num{1e28.060} & \num{1e28.986} & \num{1e29.012}\\ 
(incl. interred) & \num{1e28.820} & \num{1e28.889} &  &  &  &  & \num{1e28.709} & \num{1e28.743} &  &  &  & \\ 
\tabucline[0.1pt on 1pt off 4pt]{-}
$\random(11,2,2)$ & \num{1e30.564} & \num{1e30.645} & \num{1e32.708} & \num{1e32.708} & \num{1e32.214} & \num{1e32.226} & \num{1e31.334} & \num{1e31.362} & \num{1e30.725} & \num{1e30.725} & \num{1e31.979} & \num{1e31.989}\\ 
(incl. interred) & \num{1e31.836} & \num{1e31.880} &  &  &  &  & \num{1e31.692} & \num{1e31.715} &  &  &  & \\ 
\tabucline[0.1pt on 1pt off 4pt]{-}
$\random(12,2,2)$ & \num{1e33.432} & \num{1e33.515} & \num{1e35.602} & \num{1e35.602} & \num{1e35.437} & \num{1e35.444} & \num{1e34.272} & \num{1e34.291} & \num{1e33.523} & \num{1e33.523} & \num{1e35.028} & \num{1e35.037}\\ 
(incl. interred) & \num{1e34.771} & \num{1e34.821} &  &  &  &  & \num{1e34.674} & \num{1e34.688} &  &  &  & \\ 
\tabucline[0.1pt on 1pt off 4pt]{-}
$\random(13,2,2)$ &  &  &  &  &  &  & \num{1e37.253} & \num{1e37.266} & \num{1e36.278} & \num{1e36.278} & \num{1e38.084} & \num{1e38.085}\\ 
(incl. interred) &  &  &  &  &  &  & \num{1e37.651} & \num{1e37.661} &  &  &  & \\ 
\tabucline[0.1pt on 1pt off 4pt]{-}
$\random(14,2,2)$ &  &  &  &  &  &  & \num{1e40.141} & \num{1e40.149} & \num{1e39.080} & \num{1e39.080} & \num{1e41.140} & \num{1e41.132}\\ 
(incl. interred) &  &  &  &  &  &  & \num{1e40.582} & \num{1e40.588} &  &  &  & \\ 
\hline
\end{tabu}}
\vspace{2mm}
\caption{\# multiplications (incl. interreductions) (affine)}
\label{table:number-of-multiplications-incl-interreductions-affine-systems}
\end{table}
\end{center}

\subsection{Observations}
\begin{figure}
\begin{tikzpicture}[domain=6:14,scale=0.6,transform shape]
    \draw[very thin,dotted,draw=ggrey, xstep=2.0cm] (0,0) grid (16.9,30.9);
    \fill[fill=ggrey, fill opacity=0.1] (0,0) rectangle (16.9,30.9);
    \draw[->] (-0.2,0) -- (17.2,0) node[right] {$x$};
    \draw[->] (0pt,-2pt) -- (0,31.2) node[above] {\# operations (log $2$) for $\hrandom(x,2,2)$};

    \foreach \x/\xtext in {6,...,14}
      \draw[shift={(2*\x-12,0)}] (0pt,2pt) -- (0pt,-2pt) node[below] {$\xtext$};

    \foreach \y/\ytext in {12,14,...,42}
      \draw[shift={(0,(\y-12)/2)}] (2pt,0pt) -- (-2pt,0pt) node[left] {$\ytext$};
    \draw [myblued!80!ggrey] plot [] coordinates {
    (0,2.9318)
    (2,5.8056)
    (4,8.7234)
    (6,11.6565)
    (8,14.5893)
    (10,17.5121)
    (12,20.4684)
    (14,23.4051)
    (16,26.3148)
    (16.6,27.187710000000003)
    };
    \draw [myblued!40!ggrey] plot [] coordinates {
    (0,3.0961)
    (2,5.9265)
    (4,8.8057)
    (6,11.7104)
    (8,14.6249)
    (10,17.5354)
    (12,20.4834)
    (14,23.4150)
    (16,26.3213)
    (16.6,27.193190000000005)
    };
    \draw [mygreend!80!ggrey] plot [] coordinates {
    (0,2.6137)
    (2,5.2753)
    (4,7.9654)
    (6,10.4744)
    (8,13.0967)
    (10,15.6073)
    (12,18.3257)
    (14,20.8545)
    (16,23.5622)
    (16.6,24.37451)
    };
    \draw [mygreend!40!ggrey] plot [] coordinates {
    (0,2.6137)
    (2,5.2753)
    (4,7.9654)
    (6,10.4744)
    (8,13.0967)
    (10,15.6073)
    (12,18.3257)
    (14,20.8545)
    (16,23.5622)
    (16.6,24.37451)
    };
    \draw [myyellowd!30!ggrey] plot [] coordinates {
    (0,3.1796)
    (2,6.2172)
    (4,9.2331)
    (6,12.2498)
    (8,15.3010)
    (10,18.3507)
    (12,21.4617)
    (14,24.5814)
    (16,27.7378)
    (16.6,28.68472)
    };
    \draw [myyellowd!10!ggrey] plot [] coordinates {
    (0,3.3400)
    (2,6.3164)
    (4,9.3146)
    (6,12.2951)
    (8,15.3323)
    (10,18.3620)
    (12,21.4725)
    (14,24.5829)
    (16,27.7411)
    (16.6,28.688560000000003)
    };
    \draw [myredd!90!myyellowd] plot [] coordinates {
    (0,3.5025)
    (2,6.5045)
    (4,9.4295)
    (6,12.3509)
    (8,15.3281)
    (10,18.2337)
    (12,21.2027)
    (14,24.1372)
    (14.6,25.01755)
    };
    \draw [myredd!60!myyellowd] plot [] coordinates {
    (0,3.6151)
    (2,6.6083)
    (4,9.5323)
    (6,12.4357)
    (8,15.3767)
    (10,18.2771)
    (12,21.2322)
    (14,24.1983)
    (14.6,25.08813)
    };
    \draw [mypinkd!80!ggrey] plot [] coordinates {
    (0,3.4527)
    (2,6.4088)
    (4,9.1674)
    (6,11.9131)
    (8,14.5957)
    (10,17.2779)
    (12,20.0595)
    (14,22.8812)
    (14.6,23.727710000000002)
    };
    \draw [mypinkd!50!ggrey] plot [] coordinates {
    (0,3.4527)
    (2,6.4088)
    (4,9.1674)
    (6,11.9131)
    (8,14.5957)
    (10,17.2779)
    (12,20.0595)
    (14,22.8812)
    (14.6,23.727710000000002)
    };
    \draw [orange!80!ggrey] plot [] coordinates {
    (0,3.4032)
    (2,6.4296)
    (4,9.4546)
    (6,12.5152)
    (8,15.5760)
    (10,18.6835)
    (12,21.7913)
    (14,24.9592)
    (14.6,25.90957)
    };
    \draw [orange!50!ggrey] plot [] coordinates {
    (0,3.5851)
    (2,6.5658)
    (4,9.5436)
    (6,12.5718)
    (8,15.6090)
    (10,18.7033)
    (12,21.8024)
    (14,24.9659)
    (14.6,25.914950000000005)
    };
    \node[draw=ggrey,fill=ggrey,fill opacity=0.2,text opacity=1,
       align=left,inner sep=10pt,rounded corners=3pt] at (12,8){
    \color{myredd!90!myyellowd}{($\potl$ | $\rleqff$ | top)}\\ 
    \color{myredd!60!myyellowd}{($\potl$ | $\rleqsb$ | top)}\\ 
    \color{mypinkd!80!ggrey}{($\schl$ | $\rleqff$ | top)}\\ 
    \color{mypinkd!50!ggrey}{($\schl$ | $\rleqsb$ | top)}\\ 
    \color{orange!80!ggrey}{($\dpotl$ | $\rleqff$ | top)}\\ 
    \color{orange!50!ggrey}{($\dpotl$ | $\rleqsb$ | top)}\\ 
    \color{myblued!80!ggrey}{($\potl$ | $\rleqff$ | full)}\\ 
    \color{myblued!40!ggrey}{($\potl$ | $\rleqsb$ | full)}\\ 
    \color{mygreend!80!ggrey}{($\schl$ | $\rleqff$ | full)}\\ 
    \color{mygreend!40!ggrey}{($\schl$ | $\rleqsb$ | full)}\\ 
    \color{myyellowd!30!ggrey}{($\dpotl$ | $\rleqff$ | full)}\\ 
    \color{myyellowd!10!ggrey}{($\dpotl$ | $\rleqsb$ | full)}
    };
\end{tikzpicture}
\caption{Number of multiplications for homogeneous random examples by increasing number of generators, all of degree $2$}
\label{fig:number-of-multiplications-for-homogeneous-random-examples-by-increasing-number-of-generators,-all-of-degree-$2$}
\end{figure}
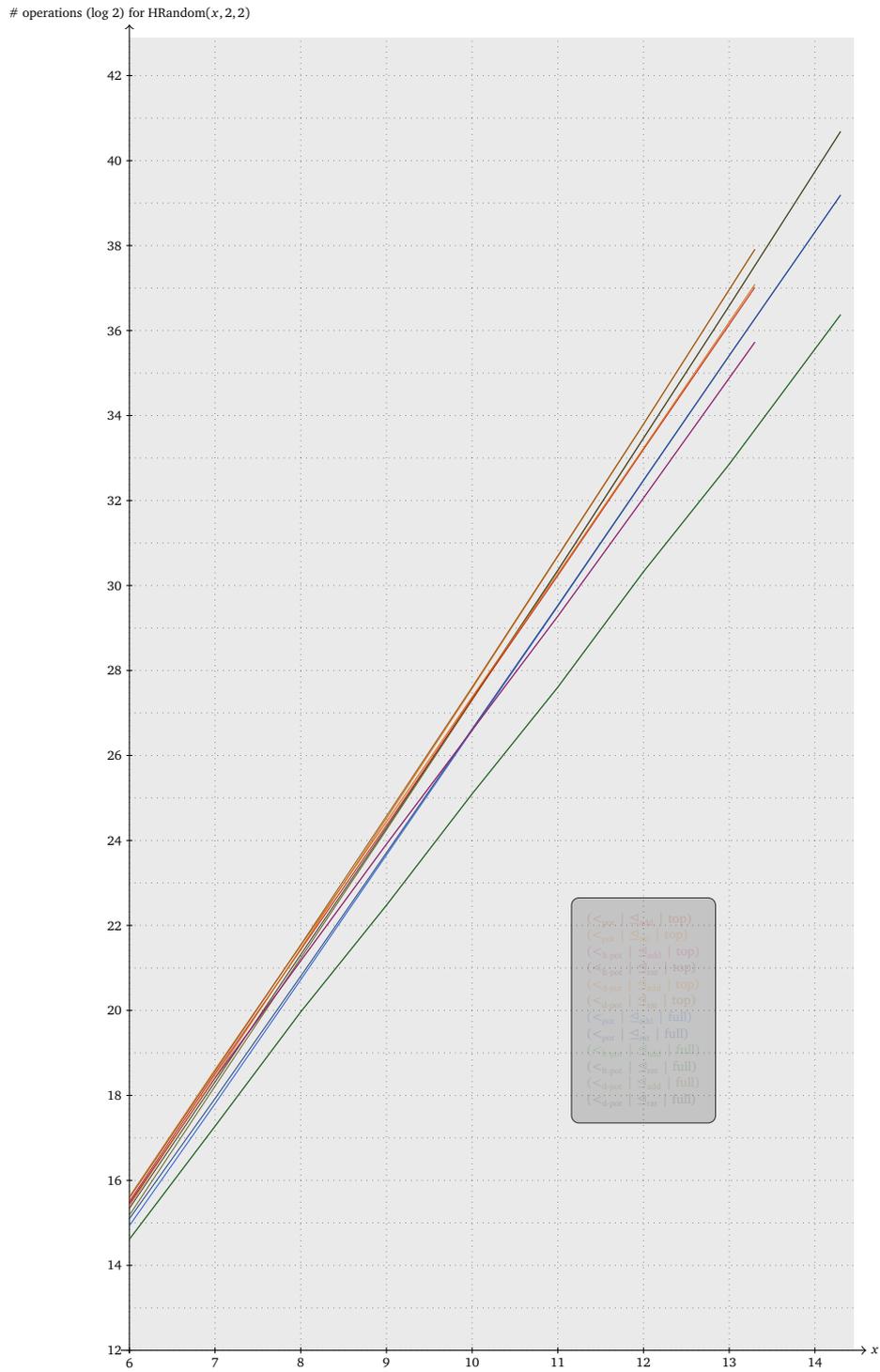
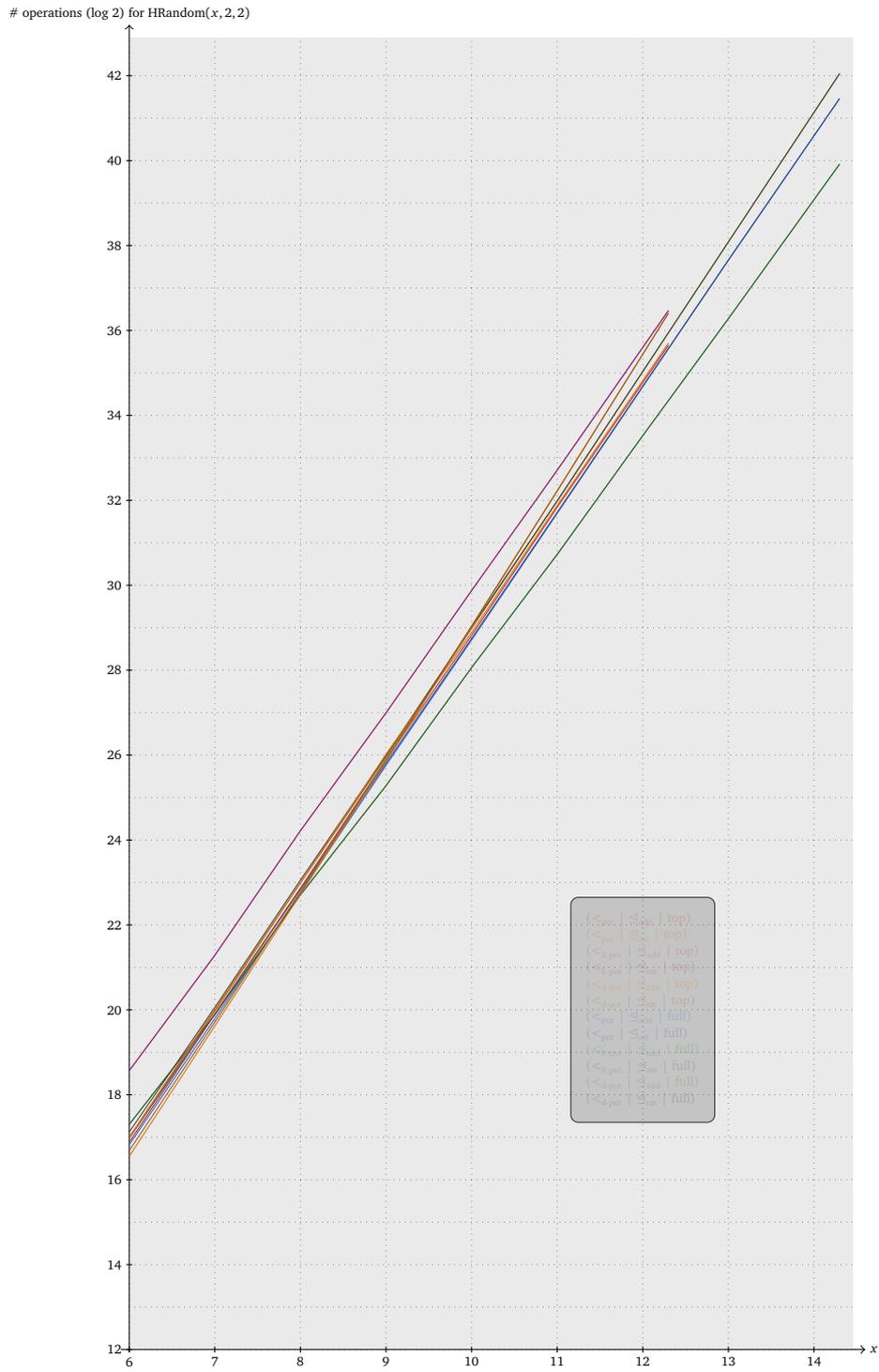
\begin{figure}
\begin{tikzpicture}[domain=6:14,scale=0.6,transform shape]
    \draw[very thin,dotted,draw=ggrey, xstep=2.0cm] (0,0) grid (16.9,30.9);
    \fill[fill=ggrey, fill opacity=0.1] (0,0) rectangle (16.9,30.9);
    \draw[->] (-0.2,0) -- (17.2,0) node[right] {$x$};
    \draw[->] (0pt,-2pt) -- (0,31.2) node[above] {\# operations (log $2$) for $\hrandom(x,2,2)$};

    \foreach \x/\xtext in {6,...,14}
      \draw[shift={(2*\x-12,0)}] (0pt,2pt) -- (0pt,-2pt) node[below] {$\xtext$};

    \foreach \y/\ytext in {12,14,...,42}
      \draw[shift={(0,(\y-12)/2)}] (2pt,0pt) -- (-2pt,0pt) node[left] {$\ytext$};
    \draw [myblued!80!ggrey] plot [] coordinates {
    (0,4.8229)
    (2,7.7925)
    (4,10.7578)
    (6,13.7564)
    (8,16.7091)
    (10,19.6922)
    (12,22.6738)
    (14,25.6513)
    (16,28.5821)
    (16.6,29.461340000000003)
    };
    \draw [myblued!40!ggrey] plot [] coordinates {
    (0,4.9817)
    (2,7.9000)
    (4,10.8344)
    (6,13.8066)
    (8,16.7435)
    (10,19.7147)
    (12,22.6882)
    (14,25.6607)
    (16,28.5881)
    (16.6,29.46632)
    };
    \draw [mygreend!80!ggrey] plot [] coordinates {
    (0,5.3027)
    (2,7.8846)
    (4,10.7062)
    (6,13.2785)
    (8,16.0596)
    (10,18.7253)
    (12,21.5234)
    (14,24.2778)
    (16,27.0799)
    (16.6,27.920530000000003)
    };
    \draw [mygreend!40!ggrey] plot [] coordinates {
    (0,5.3027)
    (2,7.8846)
    (4,10.7062)
    (6,13.2785)
    (8,16.0596)
    (10,18.7253)
    (12,21.5234)
    (14,24.2778)
    (16,27.0799)
    (16.6,27.920530000000003)
    };
    \draw [myyellowd!30!ggrey] plot [] coordinates {
    (0,4.9882)
    (2,7.9700)
    (4,10.9904)
    (6,13.9518)
    (8,16.9858)
    (10,19.9785)
    (12,23.0278)
    (14,26.0837)
    (16,29.1395)
    (16.6,30.056240000000006)
    };
    \draw [myyellowd!10!ggrey] plot [] coordinates {
    (0,5.1223)
    (2,8.0569)
    (4,11.0559)
    (6,13.9902)
    (8,17.0124)
    (10,19.9890)
    (12,23.0374)
    (14,26.0850)
    (16,29.1325)
    (16.6,30.046750000000003)
    };
    \draw [myredd!90!myyellowd] plot [] coordinates {
    (0,4.8826)
    (2,7.8963)
    (4,10.8653)
    (6,13.8884)
    (8,16.8196)
    (10,19.8356)
    (12,22.7709)
    (12.6,23.651490000000003)
    };
    \draw [myredd!60!myyellowd] plot [] coordinates {
    (0,5.0005)
    (2,8.0166)
    (4,10.9904)
    (6,14.0195)
    (8,16.8893)
    (10,19.8797)
    (12,22.8209)
    (12.6,23.703260000000004)
    };
    \draw [mypinkd!80!ggrey] plot [] coordinates {
    (0,6.5670)
    (2,9.2847)
    (4,12.2188)
    (6,14.9939)
    (8,17.8756)
    (10,20.7084)
    (12,23.6023)
    (12.6,24.47047)
    };
    \draw [mypinkd!50!ggrey] plot [] coordinates {
    (0,6.5670)
    (2,9.2847)
    (4,12.2188)
    (6,14.9939)
    (8,17.8756)
    (10,20.7084)
    (12,23.6023)
    (12.6,24.47047)
    };
    \draw [orange!80!ggrey] plot [] coordinates {
    (0,4.5479)
    (2,7.6183)
    (4,10.7202)
    (6,13.8531)
    (8,17.0222)
    (10,20.2140)
    (12,23.4367)
    (12.6,24.40351)
    };
    \draw [orange!50!ggrey] plot [] coordinates {
    (0,4.6832)
    (2,7.7077)
    (4,10.7774)
    (6,13.8869)
    (8,17.0426)
    (10,20.2257)
    (12,23.4435)
    (12.6,24.408839999999998)
    };
    \node[draw=ggrey,fill=ggrey,fill opacity=0.2,text opacity=1,
       align=left,inner sep=10pt,rounded corners=3pt] at (12,8){
    \color{myredd!90!myyellowd}{($\potl$ | $\rleqff$ | top)}\\ 
    \color{myredd!60!myyellowd}{($\potl$ | $\rleqsb$ | top)}\\ 
    \color{mypinkd!80!ggrey}{($\schl$ | $\rleqff$ | top)}\\ 
    \color{mypinkd!50!ggrey}{($\schl$ | $\rleqsb$ | top)}\\ 
    \color{orange!80!ggrey}{($\dpotl$ | $\rleqff$ | top)}\\ 
    \color{orange!50!ggrey}{($\dpotl$ | $\rleqsb$ | top)}\\ 
    \color{myblued!80!ggrey}{($\potl$ | $\rleqff$ | full)}\\ 
    \color{myblued!40!ggrey}{($\potl$ | $\rleqsb$ | full)}\\ 
    \color{mygreend!80!ggrey}{($\schl$ | $\rleqff$ | full)}\\ 
    \color{mygreend!40!ggrey}{($\schl$ | $\rleqsb$ | full)}\\ 
    \color{myyellowd!30!ggrey}{($\dpotl$ | $\rleqff$ | full)}\\ 
    \color{myyellowd!10!ggrey}{($\dpotl$ | $\rleqsb$ | full)}
    };
\end{tikzpicture}
\caption{Number of multiplications for affine random examples by increasing number of generators, all of degree $2$}
\label{fig:number-of-multiplications-for-affine-random-examples-by-increasing-number-of-generators,-all-of-degree-$2$}
\end{figure}

From the experimental results stated here one can make several observations when
it comes to signature-based \grobner{} basis algorithms:
\begin{enumerate}
\item For homogeneous input systems the number of zero reductions computed by
\rba{} using $\potl$ and $\dpotl$ is the very same: This is clear due to the
fact that \rba{} computes for each new degree $d$ $d$-\grobner{} bases step by
step. Clearly, those numbers mostly differ for affine input systems, see
Table~\ref{table:number-of-zero-reductions-affine-systems}. Even though the
number of zero reductions for \rba{} using $\schl$ is mostly higher than those
numbers for $\potl$ resp. $\dpotl$, due to $\schl$ \rba{} can handle S-pairs
more freely (not incremental, not degree-wise) and thus performs better with
almost always less \sreductions{} and operations overall.
\item Full \sreductions{} mostly behave better than top \sreductions{}.
Performing tail \sreductions{} in earlier stages of the algorithm leads to fewer
reduction steps overall. In
figures~\ref{fig:number-of-multiplications-for-homogeneous-random-examples-by-increasing-number-of-generators,-all-of-degree-$2$}
and~\ref{fig:number-of-multiplications-for-affine-random-examples-by-increasing-number-of-generators,-all-of-degree-$2$}
one can see this behaviour quite good.
\item As we can see also in
figures~\ref{fig:number-of-multiplications-for-homogeneous-random-examples-by-increasing-number-of-generators,-all-of-degree-$2$}
and~\ref{fig:number-of-multiplications-for-affine-random-examples-by-increasing-number-of-generators,-all-of-degree-$2$}
the differences between using $\rleqff$ and $\rleqsb$ vanish for bigger
computations: In each figure $12$ variants are presented, but we can only see $6$
lines: both rewrite order behave the same. Even though we can see in the above
tables that $\rleqsb$ usually leads to smaller bases and generates less S-pairs, the
choice of reducers w.r.t. $\rleqff$ is better in terms of sparsity.
\item Overall one can see the $\schl$ seems to be the best choice as module
monomial order when it comes to the number \sreductions{} resp. the number of
operations executed: Its numbers are always smaller than the ones for the
corresponding setting of $\rleq$ with other module monomial orders.
\end{enumerate}

Thus the chosen rewrite order $\rleq$ seems to be not as important as generally
accepted. The main differences lay in the module monomial order.

\section{Concluding remarks}
In this survey we covered all known variants of signature-based \grobner{} basis
algorithms. We gave a complete classification based on a generic algorithmic
framework called \rba{} which can be implemented in various different ways. The
variations are based on $3$ different orders:
\begin{enumerate}
\item $<$ denotes the monomial order as well as the compatible module monomial
order. We have seen in Section~\ref{sec:available-implementations} that this
order has the biggest impact.
\item $\rleq$ denotes the rewrite order. If \rba{} handles various elements of
the same signature, only one needs to be further \sreduced{}. The rewrite order
give a unique choice which element is chosen and which are removed. In
Section~\ref{sec:available-implementations} we have seen that the outcomes of
using different implementations of $\rleq$, namely $\rleqff$ and $\rleqsb$ are
nearly equivalent when it comes to the number of operations.
\item $\pleq$ denotes the order in which S-pairs are handled in \rba{}. Nearly
all known efficient implementations use $\pleqs$, so S-pairs are handled by
increasing signature.
\end{enumerate}

Thus any known algorithm, like \ff{} or \gvw{} can be implemented with any of
the above $3$ choices, so the difference are rather small. Even so some of those
algorithms are presented in a restricted setting, for example \ggv{} for $\potl$
only, they all can be seen as different, specialized implementatons of \rba{}
and thus are just slight variants of each other and not complete new algorithms
as possibly assumed. We covered all variants known and gave a dictionary for
translating different notations used in the corresponding publications. Thus
this survey can also be used as a reference for researcher interested in this
topic.

Important aspects when optimizing \rba{} and further open questions are the following:
\begin{enumerate}
\item Ensuring termination algorithmically as presented in
Section~\ref{sec:f5-termination-algorithmically} can lead to earlier termination
and thus improved behaviour of the algorithm by using different techniques to
detect the completeness of \basis{}.
\item Exploiting algebraic structures is an area of high research at the moment
(Section~\ref{sec:exploit-algebraic-structures}). Developments in this direction
might have a huge impact on the computations of (signature) \grobner{} bases in
the near future and are promising in decreasing the complexity of computations.
\item Using linear algebra for the reduction process as illustrated in
Section~\ref{sec:f4-f5} is another field where a lot more optimizations can be
expected. At the moment, restrictions to \sreductions{} lead to restrictions
swapping rows during the Gaussian Elimination. Getting more flexible and
possibly able to use (at least some of) the ideas from~\cite{FL10b} is still an
open problem.
\item If we are only interested in computing a \grobner{} basis for some input
system, can one generalize the usage of signatures and find an intermediate
representation between sig-poly pairs $(\sig\alpha,\proj\alpha) \in \module
\times \ring$ and full module representations $\alpha \in \module$? Where is the
breaking point of using more terms from the module representation in order to
interreduce the syzygy elements even further and not adding too much overhead in
time and memory?
\end{enumerate}

Even though quite different notations are used by researchers,
the algorithms are two of a kind, mostly they are even just the same. We hope
that this survey helps to give a better understanding on signature-based
\grobner{} basis algorithms. Moreover, we would like to give researchers
new to this area a guide to find their way through the enormous number of
publications that have been released on this topic over the last years.
Even more, we hope to encourage
experts with this survey to collaborate and to push the field of \grobner{} basis
computations even further.


\end{document}